\documentclass[10pt]{article}
\usepackage{epsf}
\usepackage{amsmath}

\allowdisplaybreaks

\usepackage[showframe=false]{geometry}
\usepackage{changepage}

\usepackage{epsfig}
\usepackage{amssymb}

\usepackage{amsthm}
\usepackage{setspace}
\usepackage{cite}

\usepackage{algorithmic}  % This package provides an algorithmic environment fo describing algorithms.
\usepackage{algorithm}

\usepackage{shadow}
\usepackage{fancybox}
\usepackage{fancyhdr}

\usepackage{color}
\usepackage[usenames,dvipsnames,svgnames,table]{xcolor}
\newcommand{\bl}[1]{\textcolor{blue}{#1}}
\newcommand{\red}[1]{\textcolor{red}{#1}}
\newcommand{\gr}[1]{\textcolor{green}{#1}}

\definecolor{mypurple}{rgb}{.4,.0,.5}

\usepackage[hyphens]{url}

\usepackage[colorlinks=true,
            linkcolor=black,
            urlcolor=blue,
            citecolor=purple]{hyperref}

\usepackage{breakurl}
\def\w{{\bf w}}

\def\y{{\bf y}}

\def\x{{\bf x}}

\def\x{{\mathbf x}}

\def\w{{\bf w}}

\def\x{{\bf x}}
\def\y{{\bf y}}

\def\h{{\bf h}}

\def\be{\begin{equation}}
\def\ee{\end{equation}}
\def\ba{\left[\begin{array}}
\def\ea{\end{array}\right]}

\def\w{{\bf w}}

\def\x{{\bf x}}
\def\y{{\bf y}}

\def\1{{\bf 1}}

\def\g{{\bf g}}
\def\0{{\bf 0}}

\def\erfinv{\mbox{erfinv}}
\def\erf{\mbox{erf}}
\def\erfc{\mbox{erfc}}

\def\erfinv{\mbox{erfinv}}

\def\Sw{S_w}

\def\Sw{S_w}

\def\mR{{\mathbb R}}

\def\psiint{\Psi_{int}}
\def\psiext{\Psi_{ext}}
\def\psicom{\Psi_{com}}
\def\psinet{\Psi_{net}}

\def\lp{\left (}
\def\rp{\right )}

\newtheorem{theorem}{Theorem}

\begin{document}

\begin{singlespace}

\title {Box constrained $\ell_1$ optimization in random linear systems -- asymptotics %A tight variant of Gordon's escape through a mesh theorem
%\footnote{ This work was supported in
%part.}
}
\author{
\textsc{Mihailo Stojnic\footnote
{e-mail: {\tt flatoyer@gmail.com}} }}
\date{}
\maketitle

\centerline{{\bf Abstract}} \vspace*{0.1in}

In this paper we consider box constrained adaptations of $\ell_1$ optimization heuristic when applied for solving random linear systems. These are typically employed when on top of being sparse the systems' solutions are also known to be confined in a specific way to an interval on the real axis. Two particular $\ell_1$ adaptations (to which we will refer as the \emph{binary} $\ell_1$ and \emph{box} $\ell_1$) will be discussed in great detail. Many of their properties will be addressed with a special emphasis on the so-called phase transitions (PT) phenomena and the large deviation principles (LDP). We will fully characterize these through two different mathematical approaches, the first one that is purely probabilistic in nature and the second one that connects to high-dimensional geometry. Of particular interest we will find that for many fairly hard mathematical problems a collection of pretty elegant characterizations of their final solutions will turn out to exist.

\vspace*{0.25in} \noindent {\bf Index Terms: Phase transitions; large deviations;
linear systems of equations; binary/box $\ell_1$}.

\end{singlespace}

%%%%%%%%%%%%%%%%%%%%%%%%%%%%%%%%%%%%%%%%%%%%%%%%%%%%%%%%%%%%%%%%%
\section{Introduction}
\label{sec:back}
%%%%%%%%%%%%%%%%%%%%%%%%%%%%%%%%%%%%%%%%%%%%%%%%%%%%%%%%%%%%%%%%%

This paper provides a detailed mathematical study of specific properties of the well known $\ell_1$ heuristic when used for solving linear systems of equations known to have solutions of particular form. These systems assume an $m\times n$ ($m\leq n$) system matrix $A$ and an $n$ dimensional vector $\tilde{\x}$ with real entries (for short we write $A\in \mR^{m\times n}$ and $\tilde{\x}\in \mR^{n}$). Then the standard matrix-vector multiplication of $A$ and $\tilde{\x}$ gives
\begin{equation}
\y=A\tilde{\x}. \label{eq:defy}
\end{equation}
One is then interested is finding $\tilde{\x}$ if $A$ and $\y$ in (\ref{eq:defy}) are given (clearly, by (\ref{eq:defy}) such an $\tilde{\x}$ obviously exists). A particularly interesting variant of this problem that attracted a lot of attention over last several decades is the under-determined scenario with structured solutions. Namely, as is well known, in the under-determined scenario $m<n$ and if $A$ is full rank (which will typically be assumed throughout the entire paper) the problem has multiple solutions and in many applications would not be among the best posed ones. However, through additional structuring of $\tilde{\x}$ one can make the above problem typically well posed (so that it actually has a unique solution). A type of structuring that has been of great interest for a long time assumes the so-called sparse solutions i.e. the sparse $\tilde{\x}$ and it is precisely in solving the linear systems known to have this type of solutions where the above mentioned $\ell_1$ heuristic has been very successful. A heuristic type of explanation for this is the following simple line of arguments. One first recognizes that finding the sparsest $\tilde{\x}$ such that (\ref{eq:defy}) holds amounts to solving
\begin{eqnarray}
\mbox{min} & & \|\x\|_{0}\nonumber \\
\mbox{subject to} & & A\x=\y, \label{eq:l0}
\end{eqnarray}
where $\|\x\|_0$ is the so-called $\ell_0$ (quasi) norm of $\x$ that basically counts the number of nonzero entries of $\x$ (of course, from this point on the assumption will always be that there is at least one $\x$ that satisfies the constraints in (\ref{eq:l0}), essentially $\tilde{\x}$ in (\ref{eq:defy})). (\ref{eq:l0}) is of course well known to be notoriously hard to solve exactly. Nonetheless, one observes that $q=1$ is the smallest $q$ such that $\|\x\|_q=(\sum_{i=1}^{n}|\x_i|_q)^{\frac{1}{q}}$ is a convex function and relaxes (\ref{eq:l0}) so that it becomes
\begin{eqnarray}
\mbox{min} & & \|\x\|_{1}\nonumber \\
\mbox{subject to} & & A\x=\y. \label{eq:l1}
\end{eqnarray}
(\ref{eq:l1}) is of course a much easier optimization problem than (\ref{eq:l0}). In fact, not only is it a convex optimization problem due to the convexity of $\|\x\|_{1}$, it is actually a linear program relatively easily solvable in polynomial time (of course, there are many other heuristics/relaxations of (\ref{eq:l0}) that one can alternatively employ see, e.g. \cite{JATGomp,NeVe07,DTDSomp,NT08,DaiMil08,DonMalMon09}; however our concern in this paper will be precisely the minimization of the $\ell_1$ norm of $\x$ from (\ref{eq:l1}) and its variants that we will discuss below as they continue to stand, in our view, as an unbeatable benchmark when it comes to solving linear systems with sparsely structured solutions). Being a much easier optimization problem than (\ref{eq:l0}) is, of course, a good feature of the $\ell_1$. However, on its own that would not be enough for its a massive use. Its excellent solving abilities and the existence of rigorous mathematical results that confirm such abilities contribute a great deal to the $\ell_1$'s popularity as well. Moreover, while the practical applicability has been known for quite some time, the analytical progress flourished over the last decade. There has been a lot of great work in recent years about various aspects of the $\ell_1$. As the $\ell_1$ in its core form (\ref{eq:l1}) will not be the central point of this paper we leave a thorough discussion about its properties to review papers and here mention only the key milestones when its comes to its performance characterizations, namely \cite{CRT,Donoho06CS} where the initial, qualitative results were presented and \cite{DonohoPol,DonohoUnsigned,StojnicCSetam09,StojnicUpper10} where the $\ell_1$'s exact performance characterizations were obtained. These, in our view, mathematically solidified the importance of (\ref{eq:l1}) in studying linear inverse problems.

In this paper we will consider an upgrade to the standard sparse structuring mentioned above. Namely, we will be interested in unknown vectors that in addition to being sparse are also known to be from a given interval. When stated like this, one then recognizes that these kinds of vectors are not that much different from any vectors (simply one can always design an interval so that all components of any vector are from such an interval; obviously, we will throughout the paper consider so to say practically realistic scenarios, i.e. vectors that have finite components). To remove this ambiguity we will first introduce the so-called \emph{binary} sparse vectors (later in the paper we will expand this definition so that it includes vectors that more faithfully resemble the ones with the elements from a given interval). Namely, the binary vectors will have each of their components equal either to zero or to one (more on this or similar discrete type of unknown vectors as well as on their potential applications can be found in e.g. \cite{DaiMil09,DTbern,ManRec09,ManFer09,DonMalMon09,StojnicISIT2010binary}). While it will be fairly obvious later on, we still take the opportunity right here at the beginning to emphasize that there is really nothing specific about zero and one and that instead of them one can choose basically any two real numbers and all of what we will present below will hold with minimal/trival adjustments. Additionally, we will call binary vectors $k$ sparse if they have $k$ components equal to one and the remaining ones equal to zero. It is also relatively easy to note that the binary sparse vectors are a subclass of the so-called nonnegative sparse vectors studied in e.g. \cite{DTbern,StojnicCSetam09,YinZhang05nonneg,Stojnicl1RegPosasymldp}. One can of course still use the standard $\ell_1$ to solve under-determined systems with nonnegative or binary sparse solutions. However, as it is by now well known (see, e.g. \cite{DonohoPol,DonohoUnsigned,StojnicCSetam09,StojnicUpper10}), a substantial performance improvement can be achieved if one slightly modifies the standard $\ell_1$ from (\ref{eq:l1}). For the nonnegative case such a modification consists of adding the positivity constraints on the elements of the unknown $\x$ (we typically call such a modification of the standard $\ell_1$, the nonnegative $\ell_1$). In a similar fashion, for the binary sparse case the following modification of (\ref{eq:l1}) is typically considered (see e.g. \cite{DTbern,StojnicISIT2010binary})
\begin{eqnarray}
\mbox{min} & & \|\x\|_1\nonumber \\
\mbox{subject to} & & A\x=\y \nonumber \\
& & 0\leq \x_i\leq 1, 1\leq i\leq n. \label{eq:l1bin}
\end{eqnarray}
The above problem, to which we will refer as the \emph{binary} or \emph{box} $\ell_1$, is fairly similar to the standard $\ell_1$ from (\ref{eq:l1}). When it comes to the binary vectors (similarly to what was the case for the nonnegative vectors), one expects that (\ref{eq:l1bin}) should have a bit better recovery abilities than the standard $\ell_1$ as it incorporates the a priori available knowledge that the elements of the unknown sparse vectors are constrained to be in $[0,1]$ interval (in fact, not only should it have better recovery abilities than the standard $\ell_1$, it should actually have better recovery abilities than the nonnegative $\ell_1$ as well). \cite{StojnicISIT2010binary} rigorously showed that this is indeed true. More importantly, in a statistical context, \cite{StojnicISIT2010binary} precisely quantified by how much the algorithm from (\ref{eq:l1bin}) improves on both, the standard and the nonnegative $\ell_1$. In the following sections we will in detail recall on the results from \cite{StojnicISIT2010binary}. Here we briefly emphasize the difference between what was done in \cite{StojnicISIT2010binary} and what will be done here. The results of \cite{StojnicISIT2010binary} relate to the so-called phase-transition (PT) phenomena (these are of course the same phenomena that appeared in  \cite{DonohoPol,DonohoUnsigned,StojnicCSetam09,StojnicUpper10} among the key properties that the standard and the nonnegative $\ell_1$ exhibit). Basically, in the standard linear regime (regime where $n$ is large, $m=\alpha n$, $k=\beta n$, and $\alpha$ and $\beta$ are constants independent of $n$) \cite{StojnicISIT2010binary} precisely characterized the so-called ``breaking points" where these phase transitions happen (essentially the highest possible $\beta$ for which the solution of (\ref{eq:l1bin}) with overwhelming probability matches the sparsest solution of (\ref{eq:l0}) for a fixed $\alpha$; under overwhelming probability we will in this paper consider probability over statistics of $A$ that is no more than a number exponentially decaying in $n$ away from $1$). On the other hand, here, we will rely on the concepts introduced in \cite{Stojnicl1RegPosasymldp} and will take a look at the phase transitions from a different angle. Following \cite{Stojnicl1RegPosasymldp}, we will connect the phase transitions to the so-called \emph{large deviations principle} (LDP) from the classical probability theory and provide their explicit characterizations when viewed in such a way. We will do so for the binary/box $\ell_1$ from (\ref{eq:l1bin}) when used as a heuristic for finding two types of sparse unknown vectors constrained to have elements from a real interval: the first one being the above introduced binary sparse vectors and the second one being the so-called box-constrained vectors that we will introduce later on. Moreover, we will do so through two seemingly different approaches, one that is purely probabilistic and another one that has a nice connection to the high-dimensional geometry.

We will split the presentation into several sections, but two of them will of course be dominant. We will start by discussing the phase transitions of the binary $\ell_1$. After that we will move to the LDP characterizations and their connections with the PTs. In the later sections of the paper we will show how the PT and LDP results that we will create for the $\ell_1$ from (\ref{eq:l1bin}) when used for finding the binary sparse vectors can be modified so that they fit the usage of such $\ell_1$ for finding the above mentioned box-constrained sparse vectors.

%%%%%%%%%%%%%%%%%%%%%%%%%%%%%%%%%%%%%%%%%%%%%%%%%%%%%%%%%%%%%%%%%
\section{Binary $\ell_1$}
\label{sec:binl1}
%%%%%%%%%%%%%%%%%%%%%%%%%%%%%%%%%%%%%%%%%%%%%%%%%%%%%%%%%%%%%%%%%

In this section we will revisit the phase transitions (PTs) of the $\ell_1$ from (\ref{eq:l1bin}) and then we will in great detail study the corresponding LDPs. From this point on we will make a clear distinction in the used terminology when it comes to the binary and box $\ell_1$. Namely, we will
exclusively refer to the $\ell_1$ from (\ref{eq:l1bin}), the \emph{binary} $\ell_1$, when it is used for solving systems known to have binary solutions. On the other hand, the term \emph{box} $\ell_1$ will be exclusively reserved for the usage of the $\ell_1$ from (\ref{eq:l1bin}) for solving systems known to have box-constrained solutions which, as mentioned earlier, we will introduce later on.

%%%%%%%%%%%%%%%%%%%%%%%%%%%%%%%%%%%%%%%%%%%%%%%%%%%%%%%%%%%%%%%%%
\subsection{Phase transitions}
\label{sec:phasetrans}
%%%%%%%%%%%%%%%%%%%%%%%%%%%%%%%%%%%%%%%%%%%%%%%%%%%%%%%%%%%%%%%%%

Naturally, we start by recalling on the definitions of the PTs. These are of course generally well known, so we briefly state them without too much detailing (for a more comprehensive view, a long line of our work \cite{StojnicCSetam09,StojnicISIT2010binary,Stojnicl1RegPosasymldp,StojnicCSetamBlock09,StojnicUpper10} can be consulted). To that end, we say that for any given constant $0< \alpha\leq 1$ and \emph{any} given binary $\x$ with a given fixed location of its nonzero components there will be a maximum allowable value of $\beta$ such that
(\ref{eq:l1bin}) finds that given $\x$ with overwhelming
probability. We will refer to this maximum allowable value of
$\beta$ as the \emph{weak threshold/breaking point} and will denote it by $\beta_{w}$ (see, e.g. \cite{StojnicICASSP09,StojnicCSetam09,StojnicTowBettCompSens13,StojnicICASSP10knownsupp,Stojnicl1RegPosasymldp}). Correspondingly, we also say that the algorithm exhibits the \emph{weak} phase transition (i.e. \emph{weak} PT). Under fully characterizing the weak phase transition one then considers determining the so-called \emph{weak} PT curve in $(\alpha,\beta)$ plane so that for any pair $(\alpha,\beta)$ that is below this curve the algorithm (here (\ref{eq:l1bin})) succeeds with overwhelming probability in solving (\ref{eq:l0}); otherwise it fails. In addition to the weak phase transitions, one can define various other forms of phase transitions. However, we stop short of discussing these in greater details as they will not be the main subject of this paper (more on them though can be found in e.g. \cite{DonohoPol,StojnicCSetam09,StojnicUpper10,StojnicUpperSec13,StojnicLiftStrSec13,Stojnicl1RegPosasymldp}).

As mentioned earlier and as is by now well known, \cite{DonohoPol,DonohoUnsigned,StojnicCSetam09,StojnicUpper10} fully characterized the standard $\ell_1$ PT (\cite{DonohoPol,DonohoUnsigned} through a high-dimensional geometry and \cite{StojnicCSetam09,StojnicUpper10} through a purely probabilistic approach). In \cite{StojnicISIT2010binary} we went a step further and fully characterized the binary $\ell_1$ PT as well. The following theorem summarizes the obtained characterization.
\begin{theorem}(\cite{StojnicISIT2010binary} Exact binary $\ell_1$'s weak threshold/PT)
Let $A$ be an $m\times n$ matrix in (\ref{eq:l0})
with i.i.d. standard normal components. Let
the unknown $\x$ that solves (\ref{eq:l0}) be binary $k$-sparse. Further, let the locations of nonzero elements of $\x$ be arbitrarily chosen but fixed. Assume that the nonzero elements of $\x$ are equal to one. Let $k,m,n$ be large
and let $\alpha_w=\frac{m}{n}$ and $\beta_w=\frac{k}{n}$ be constants
independent of $m$ and $n$. Let $\erfinv$ be the inverse of the standard error function associated with zero-mean unit variance Gaussian random variable.  Further, let $\alpha_w$ and $\beta_w$ satisfy the following \textbf{fundamental characterization of the \emph{binary} $\ell_1$'s PT}

\begin{center}
\shadowbox{$
%\begin{equation}
\xi^{(bin)}_{\alpha_{w}}(\beta_w)\triangleq\psi^{(bin)}_{\beta_w}(\alpha_{w})\triangleq
\frac{(1-2\beta_w)\sqrt{\frac{2}{\pi}}e^{-\lp\erfinv\lp\frac{1-2\alpha_w}{1-2\beta_w}\rp\rp^2}}{2\alpha_w\sqrt{2}\erfinv \lp\frac{1-2\alpha_w}{1-2\beta_w}\rp}=1.
%\end{equation}
$}
-\vspace{-.5in}\begin{equation}
\label{eq:thmweaktheta2}
\end{equation}
\end{center}

Then:
\begin{enumerate}
\item If $\alpha>\alpha_w$ then with overwhelming probability the solution of (\ref{eq:l1bin}) is the binary $k$-sparse $\x$ that solves (\ref{eq:l0}).
\item If $\alpha<\alpha_w$ then with overwhelming probability the binary $k$-sparse $\x$ with given fixed locations of nonzero components is the solution of (\ref{eq:l0}) and is \textbf{not} the solution of (\ref{eq:l1bin}).
    \end{enumerate}
\label{thm:thmweakthr}
\end{theorem}
\begin{proof}
The first part was established in \cite{StojnicISIT2010binary} and the second one followed automatically from considerations in \cite{StojnicUpper10,StojnicRegRndDlt10}.
\end{proof}

%%%%%%%%%%%%%%%%%%%%%%%%%%%%%%%%%%%%%%%%%%%%%%%%%%%%%%%%%%%%%%%%%
\subsubsection{Doubling the number of equations}
\label{sec:propxi}
%%%%%%%%%%%%%%%%%%%%%%%%%%%%%%%%%%%%%%%%%%%%%%%%%%%%%%%%%%%%%%%%%

The fundamental PT characterizations given in the above theorem are indeed well defined. Namely, for any given fixed $\alpha\in (0,\frac{1}{2})$ there will be a unique $\beta\in(0,\alpha)$ such that $\xi^{(bin)}_{\alpha}(\beta)=1$ and for any given fixed $\beta\in (0,\frac{1}{2})$ there will be a unique $\alpha\in(\beta,\frac{1}{2})$ such that $\psi^{(bin)}_{\beta}(\alpha)=1$. This follows immediately after one first notes that the change $\beta\leftarrow 2\beta$ and $\alpha\leftarrow 2\alpha$ transforms the above characterizations into the corresponding ones obtained for the standard $\ell_1$ in \cite{StojnicCSetam09,StojnicUpper10,Stojnicl1RegPosasymldp} and then recalls that in \cite{Stojnicl1RegPosasymldp} it was explicitly shown that these characterizations are indeed unambiguous. What is perhaps a bit more interesting (especially from a practical point of view) is the so-called doubling the number of equations phenomenon. Namely, as the above mentioned change $\beta\leftarrow 2\beta$, $\alpha\leftarrow 2\alpha$ indicates, the binary $\ell_1$ for the same $\frac{\beta}{\alpha}$ ratio needs exactly two times smaller number of equations. This can be clearly seen from Figure \ref{fig:weakl1PTbin} where we show the theoretical PT curves for the binary $\ell_1$ that one can obtain based on (\ref{eq:thmweaktheta2}). In addition to the binary $\ell_1$ PT curve we also show the corresponding PT curves for the standard and nonnegative $\ell_1$ PTs. As arrows in Figure \ref{fig:weakl1PTbin} indicate, to achieve the same $\beta/\alpha$ ratio that the binary $\ell_1$ achieves, the standard $\ell_1$ needs exactly two times larger $\alpha$.
\begin{figure}[htb]
%\begin{minipage}[b]{.5\linewidth}
\centering
\centerline{\epsfig{figure=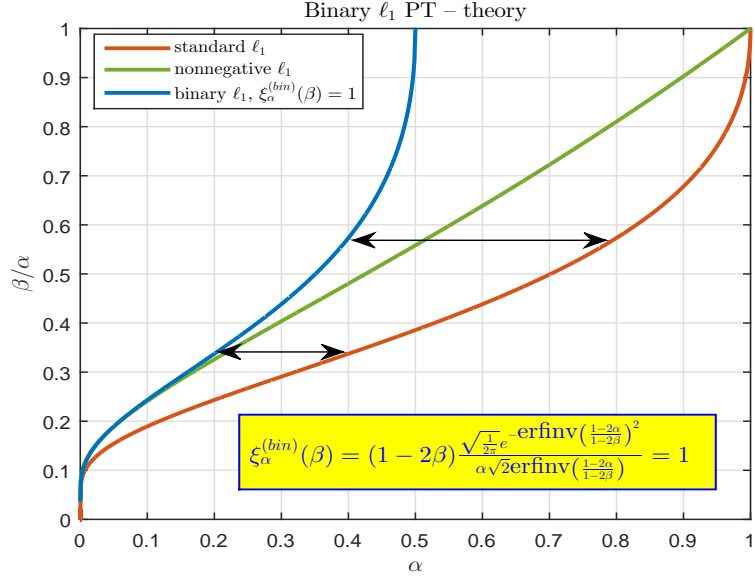,width=11.5cm,height=8cm}}
%\end{minipage}
%\begin{minipage}[b]{.5\linewidth}
%\centering
%\centerline{\epsfig{figure=finprerral08.eps,width=9cm,height=6.5cm}}
%\end{minipage}
\caption{Binary $\ell_1$'s weak PT; $\{(\alpha,\beta)|\xi^{(bin)}_{\alpha}(\beta)=1\}$}
\label{fig:weakl1PTbin}
\end{figure}

%\begin{figure}[htb]
%\begin{minipage}[b]{.5\linewidth}
%\centering
%\centerline{\epsfig{figure=CSetamBlockWeak.eps,width=7.5cm,height=7cm}}
%%\end{minipage}
%%\begin{minipage}[b]{.5\linewidth}
%%\centering
%%\centerline{\epsfig{figure=finprerral08.eps,width=9cm,height=6.5cm}}
%\end{minipage}
%\begin{minipage}[b]{.5\linewidth}
%\centering
%\centerline{\epsfig{figure=SimulBlSpWeakd151.eps,width=7.5cm,height=7cm}}
%%\end{minipage}
%%\begin{minipage}[b]{.5\linewidth}
%%\centering
%%\centerline{\epsfig{figure=finprerral08.eps,width=9cm,height=6.5cm}}
%\end{minipage}
%\caption{\emph{Weak} threshold, $\ell_2/\ell_1$-optimization; theory -- left, simulations -- right}
%\label{fig:weak}
%\end{figure}

%%%%%%%%%%%%%%%%%%%%%%%%%%%%%%%%%%%%%%%%%%%%%%%%%%%%%%%%%%%%%%%%%
\subsection{Large deviations}
\label{sec:ldp}
%%%%%%%%%%%%%%%%%%%%%%%%%%%%%%%%%%%%%%%%%%%%%%%%%%%%%%%%%%%%%%%%%

In this section we discuss the binary $\ell_1$ LDP characterizations that will provide a significantly richer spectrum of information about the above discussed PTs -- namely, they will explain the algorithms behaviour not only at the breaking points/thresholds but also in the entire transition zone around these points. The key difference between the standard PTs and the LDPs that we will discuss below will be in the exactness of the characterizations of the rates at which the probabilities of algorithms' success (failure) tend to zero as the systems dimensions deviate from the ones that satisfy the PT curves (i.e. the breaking points of the algorithms' success). To achieve full exactness in characterizing these rates, we will below determine the LDPs relying on the connection between the PTs and the LDPs that we established in \cite{Stojnicl1RegPosasymldp}. Consequently, we will also try to emulate the strategies designed in \cite{Stojnicl1RegPosasymldp}, Moreover, to make the exposition easier to follow, we will try to make everything look as parallel to what was done in \cite{Stojnicl1RegPosasymldp} as possible (many repetitive steps though will be skipped and the emphasis will be on those that bring the key differences).

As is usual the case with many of the strategies that we designed, we start things off by recalling on a couple of fundamentally important technical results that we established in \cite{StojnicCSetam09,StojnicICASSP09,StojnicICASSP10knownsupp,StojnicISIT2010binary}. To ensure the clarity and simplicity of the exposition, we will without loss of generality assume that the elements $\x_{1},\x_{2},\dots,\x_{n-k}$ of $\x$ are equal to zero and that the elements $\x_{n-k+1},\x_{n-k+2},\dots,\x_n$ are all equal to one (we emphasize that it is of course not known to the algorithm beforehand which elements are equal to one, however it is assumed to be known that each element of the unknown vector $\x$ in (\ref{eq:l1bin}) is either zero or one; the above assumption is of course only for the concreteness purposes of the analysis that will be presented below and is of course in an agreement with the requirement that the weak phase transition imposes). Relying on the breakthrough observations of \cite{StojnicCSetam09,StojnicICASSP09}, we in \cite{StojnicISIT2010binary} established the following theorem which is one of the key engines behind the entire machinery developed in \cite{StojnicCSetam09,StojnicICASSP09,StojnicISIT2010binary}.
\begin{theorem}(\cite{StojnicCSetam09,StojnicICASSP09,StojnicISIT2010binary} Nonzero elements of binary $\x$ have fixed location)
Assume that an $m\times n$ system matrix $A$ is given. Let $\x$
be a binary $k$ sparse vector. Also let $\x_1=\x_2=\dots=\x_{n-k}=0$. Assume that the nonzero elements of $\x$ are equal to one. Further, assume that $\y\triangleq A\x$ and that $\w$ is
an $n\times 1$ vector such that $\w_i\geq 0,1\leq i\leq n-k$, and $\w_i\leq 0,n-k+1\leq i\leq n$. If
\begin{equation}
(\forall \w\in \textbf{R}^n | A\w=0) \quad  -\sum_{i=n-k+1}^{n} \w_i<\sum_{i=1}^{n-k}\w_{i},
\end{equation}
then the solutions of (\ref{eq:l0}) and (\ref{eq:l1bin}) coincide. Moreover, if
\begin{equation}
(\exists \w\in \textbf{R}^n | A\w=0) \quad  -\sum_{i=n-k+1}^{n} \w_i\geq \sum_{i=1}^{n-k}\w_{i},
\label{eq:thmeqgen}
\end{equation}
then the solution of (\ref{eq:l0}) and is not the solution of (\ref{eq:l1bin}).
\label{thm:thmknownsuppcond}
\end{theorem}
To facilitate the exposition we set
\begin{equation}
\Sw\triangleq\{\w\in S^{n-1}| \quad -\sum_{i=n-k+1}^{n} \w_i<\sum_{i=1}^{n-k}\w_{i},\quad \w_i\geq 0,1\leq i\leq (n-k), \quad \mbox{and} \quad\w_i\leq 0,(n-k)+1\leq i\leq n-k\},\label{eq:defSwpr}
\end{equation}
and as in \cite{Stojnicl1RegPosasymldp}, we first provide a detailed analysis of the so-called upper tail of the LDP characterizations (as for the standard $\ell_1$ LDPs, it will turn out that the minimal adaptations of the upper tail analysis automatically settle the lower tail as well).

%%%%%%%%%%%%%%%%%%%%%%%%%%%%%%%%%%%%%%%%%%%%%%%%%%%%%%%%%%%%%%%%%
\subsubsection{Upper tail}
\label{sec:uppertail}
%%%%%%%%%%%%%%%%%%%%%%%%%%%%%%%%%%%%%%%%%%%%%%%%%%%%%%%%%%%%%%%%%

We will first consider the LDPs upper tail, which means, the points $(\alpha,\beta)$ such that $\alpha\geq \alpha_w$ where $\alpha_w$ is such that $\psi^{(bin)}_{\beta}(\alpha_w)=\xi^{(bin)}_{\alpha_w}(\beta)=1$. Assuming that the elements of $A$ are i.i.d. standard normals and following \cite{Stojnicl1RegPosasymldp}, we write
\begin{equation}
P_{err}\triangleq P(\min_{\w\in S_w}\|A\w\|_2\leq 0)=P(\max_{\w\in S_w}\min_{\|\y\|_2=1}(\y^T A\w )\geq 0)\leq
\min_{c_3\geq 0} e^{-\frac{c_3^2}{2}}Ee^{-c_3\|\g\|_2}Ee^{c_3w(\h,S_w)},
\label{eq:ldpprob3}
\end{equation}
where $P_{err}$ is the so-called probability of error/failure, i.e. the probability that (\ref{eq:l1bin}) fails to produce the solution of (\ref{eq:l0}) and
\begin{eqnarray}
w(\h,\Sw)\triangleq\max_{\w\in \Sw} (\h^T\w) = \max_{\bar{\y}\in \mR^{n}} & &  \sum_{i=1}^{n} \h_i \bar{\y}_i\nonumber \\
\mbox{subject to} &  & \bar{\y}_i\geq 0, 0\leq i\leq n\nonumber \\
& & \sum_{i=n-k+1}^{n}\bar{\y}_i\geq \sum_{i=1}^{n-k} \bar{\y}_i \nonumber \\
& & \sum_{i=1}^{n}\bar{\y}_i^2\leq 1,\label{eq:workww2}
\end{eqnarray}
with the elements of $\h$ being the i.i.d. standard normals. As in \cite{StojnicCSetam09,StojnicCSetamBlock09,StojnicBlockasymldpfinn15,StojnicLiftStrSec13,StojnicICASSP10knownsupp,StojnicTowBettCompSens13,Stojnicl1RegPosasymldp,StojnicISIT2010binary} one writes
\begin{eqnarray}
w(\h,\Sw) = -\max_{\nu\geq 0,\gamma\geq 0}\min_{\bar{\y}} & & \sum_{i=1}^{n} -\h_i \bar{\y}_i+\nu\sum_{i=1}^{n-k}\bar{\y}_i
-\nu\sum_{i=n-k+1}^{n}\bar{\y}_i+\gamma\sum_{i=1}^{n}\bar{\y}_i^2-\gamma\nonumber \\
\mbox{subject to} & & \bar{\y}_i\geq 0, 0\leq i\leq n,\label{eq:ldpwhSw0}
\end{eqnarray}
and finally
\begin{eqnarray}
w(\h,\Sw) & = & \min_{\nu\geq0,\gamma\geq 0} \frac{\sum_{i=1}^{n-k}\max(\h_i-\nu,0)^2+\sum_{i=n-k+1}^{n}\max(\h_i+\nu,0)^2}{4\gamma}+\gamma\nonumber \\
& = & \min_{\nu\geq0}\sqrt{\sum_{i=1}^{n-k}\max(\h_i-\nu,0)^2+\sum_{i=n-k+1}^{n}\max(\h_i+\nu,0)^2}.\label{eq:ldpwhSw}
\end{eqnarray}
We summarize the above methodology to upper bound $P_{err}$ in the following theorem.
\begin{theorem}
Let $A$ be an $m\times n$ matrix in (\ref{eq:l0})
with i.i.d. standard normal components. Let
the unknown $\x$ in (\ref{eq:l0}) be binary $k$-sparse and let the locations of nonzero elements of $\x$ be arbitrarily chosen but fixed. Assume that the nonzero elements of $\x$ are equal to one. Let $P_{err}$ be the probability that the solution of (\ref{eq:l1bin}) is not the solution of (\ref{eq:l0}). Then
\begin{equation}
P_{err} \leq  \min_{c_3\geq 0}e^{-\frac{c_3^2}{2}}e^{-c_3\|\g\|_2}Ee^{c_3w(\h,S_w)}=\min_{c_3\geq 0}\left (e^{-\frac{c_3^2}{2}}\frac{1}{\sqrt{2\pi}^m}\int_{\g}e^{-\sum_{i=1}^{m}\g_i^2/2-c_3\|\g\|_2}d\g \min_{\nu\geq 0,\gamma\geq\frac{c_3}{2}} w_1^{n-k}w_2^{k}e^{c_3\gamma}\right ),
\label{eq:ldpthm1perrub1}
\end{equation}
where
\begin{eqnarray}
% \nonumber % Remove numbering (before each equation)
w_1 &=& \frac{1}{\sqrt{2\pi}}\int_{h}e^{-h^2/2}e^{c_3\max(h-\nu,0)^2/4/\gamma}dh
  =\frac{1}{2}\lp\frac{e^{\frac{c_3\nu^2/4/\gamma}{1-c_3/2/\gamma}}}{\sqrt{1-c_3/2/\gamma}}\erfc\left (\frac{\nu}{\sqrt{2}\sqrt{1-c_3/2/\gamma}}\right )+\erf\left (\frac{\nu}{\sqrt{2}}\right )+1\rp\nonumber \\
w_2 &=& \frac{1}{\sqrt{2\pi}}\int_{h}e^{-h^2/2}e^{c_3\max(h+\nu,0)^2/4/\gamma}dh
  =\frac{1}{2}\lp\frac{e^{\frac{c_3\nu^2/4/\gamma}{1-c_3/2/\gamma}}}{\sqrt{1-c_3/2/\gamma}}\erfc\left (\frac{-\nu}{\sqrt{2}\sqrt{1-c_3/2/\gamma}}\right )+\erf\left (\frac{-\nu}{\sqrt{2}}\right )+1\rp.\nonumber \\
\label{eq:ldpthm1perrub2}
\end{eqnarray}\label{thm:ldp1}
\end{theorem}
\begin{proof}
Follows from the above considerations and ultimately through the mechanisms developed in \cite{StojnicCSetam09,StojnicCSetamBlock09,StojnicBlockasymldpfinn15,StojnicLiftStrSec13,StojnicICASSP10knownsupp,StojnicTowBettCompSens13,Stojnicl1RegPosasymldp,StojnicISIT2010binary}.
\end{proof}
The above theorem clearly provides an upper bound that holds for any integers $m$, $k$, and $n$ (provided $k\leq m\leq n$ so that the results make sense). Below we will be interested in the LDP type of results which naturally assume the $n\rightarrow\infty$ asymptotic regime (the same is of course true for the PT types of results and Theorem \ref{thm:thmweakthr}). Following \cite{Stojnicl1RegPosasymldp}, we consider the decay rate of $P_{err}$, namely $I^{(bin)}_{err}(\alpha,\beta)$,
\begin{equation}\label{eq:ldpasymp1}
  I^{(bin)}_{err}(\alpha,\beta)\triangleq\lim_{n\rightarrow\infty}\frac{\log{P_{err}}}{n}.
\end{equation}
and based on Theorem \ref{thm:ldp1} we have the following LDP type of theorem.
\begin{theorem}
Assume the setup of Theorem \ref{thm:ldp1}. Further, let integers $m$, $k$, and $n$ be large ($k\leq m\leq n$) such that $\beta=\frac{k}{n}$ and $\alpha=\frac{m}{n}$ are constants independent of $n$. Assume that a pair $(\alpha,\beta)$  is given. Also, assume the following scaling: $c_3\rightarrow c_3\sqrt{n}$ and $\gamma\rightarrow\gamma\sqrt{n}$. Then
\begin{eqnarray}
I^{(bin)}_{err}(\alpha,\beta) & \triangleq &\lim_{n\rightarrow\infty}\frac{\log{P_{err}}}{n}\nonumber \\
& \leq & \min_{c_3\geq 0}\left (-\frac{(c_3)^2}{2}+I_{sph}+\min_{\nu\geq 0,\gamma\geq \frac{c_3}{2}} ((1-\beta)\log{w_1}+\beta\log{w_2}+c_3\gamma)\right )
\triangleq I_{err,u}^{(bin,ub)}(\alpha,\beta),\nonumber \\
\label{eq:ldpthm2Ierrub1}
\end{eqnarray}
where
\begin{eqnarray}
% \nonumber % Remove numbering (before each equation)
I_{sph} &=& \widehat{\gamma}c_3-\frac{\alpha }{2}\log\left (1-\frac{c_3}{2\widehat{\gamma}}\right )\nonumber \\
  \widehat{\gamma} &=& \frac{c_3-\sqrt{(c_3)^2+4\alpha}}{4}\nonumber \\
w_1 &=& \frac{1}{\sqrt{2\pi}}\int_{h}e^{-h^2/2}e^{c_3\max(h-\nu,0)^2/4/\gamma}dh
  =\frac{1}{2}\lp\frac{e^{\frac{c_3\nu^2/4/\gamma}{1-c_3/2/\gamma}}}{\sqrt{1-c_3/2/\gamma}}\erfc\left (\frac{\nu}{\sqrt{2}\sqrt{1-c_3/2/\gamma}}\right )+\erf\left (\frac{\nu}{\sqrt{2}}\right )+1\rp\nonumber \\
w_2 &=& \frac{1}{\sqrt{2\pi}}\int_{h}e^{-h^2/2}e^{c_3\max(h+\nu,0)^2/4/\gamma}dh
  =\frac{1}{2}\lp\frac{e^{\frac{c_3\nu^2/4/\gamma}{1-c_3/2/\gamma}}}{\sqrt{1-c_3/2/\gamma}}\erfc\left (\frac{-\nu}{\sqrt{2}\sqrt{1-c_3/2/\gamma}}\right )+\erf\left (\frac{-\nu}{\sqrt{2}}\right )+1\rp.\nonumber \\
\label{eq:ldpthm2perrub2}
\end{eqnarray}\label{thm:ldp2}
\end{theorem}
\begin{proof} Follows in a fashion analogous to the one employed in \cite{Stojnicl1RegPosasymldp}.
\end{proof}
The above optimization problem can be solved numerically and that would be enough to provide the estimates for the rate of $P_{err}$'s decay. Instead of relying on a numerical solving we will below present an explicit solution. We will try to follow at least to some degree the methodology introduced in \cite{Stojnicl1RegPosasymldp}. However, the technical considerations will be a bit more involved and our presentation will occasionally deviate from what was presented in \cite{Stojnicl1RegPosasymldp}. Moreover, it will turn out that the optimizing quantities and eventually the LDPs rate functions will exhibit a behavior substantially different from the one observed in \cite{Stojnicl1RegPosasymldp} when the standard $\ell_1$ was considered.

%%%%%%%%%%%%%%%%%%%%%%%%%%%%%%%%%%%%%%%%%%%%%%%%%%%%%%%%%%%%%%%%%
\subsubsection{Determining $I_{err,u}^{(bin,ub)}$}
\label{sec:analysisIerr}
%%%%%%%%%%%%%%%%%%%%%%%%%%%%%%%%%%%%%%%%%%%%%%%%%%%%%%%%%%%%%%%%%

As in \cite{Stojnicl1RegPosasymldp} we start by setting
\begin{equation}\label{eq:detanalIerr1}
  A_{0}\triangleq\sqrt{1-\frac{c_3}{2\gamma}},
\end{equation}
and then consider the following optimization problem
\begin{equation}\label{eq:detanalIeer2}
I_{err,u}^{(bin,ub)}(\alpha,\beta)\triangleq \min_{c_3\geq 0,\nu\geq 0,A_0\leq 1}\zeta_{\alpha,\beta}(c_3,\nu,A_0),
\end{equation}
where
\begin{eqnarray}
% \nonumber % Remove numbering (before each equation)
\zeta_{\alpha,\beta}(c_3,\nu,A_0)&=&\left (-\frac{c_3^2}{2}+I_{sph}+(1-\beta)\log{w_1}+\beta\log{w_2}+\frac{c_3^2}{2(1-A_0^2)}\right )\nonumber \\
I_{sph} &=& \widehat{\gamma}c_3-\frac{\alpha }{2}\log\left (1-\frac{c_3}{2\widehat{\gamma}}\right )\nonumber \\
  \widehat{\gamma} &=& \frac{c_3-\sqrt{(c_3)^2+4\alpha}}{4}\nonumber \\
w_1 &=& \frac{1}{2}\lp
  \frac{e^{\frac{(1-A_0^2)\nu^2}{2A_0^2}}}{A_0}\erfc\left (\frac{\nu}{\sqrt{2}A_0}\right )+\erf\left (\frac{\nu}{\sqrt{2}}\right )+1\rp\nonumber \\
w_2 &=& \frac{1}{2}\lp
  \frac{e^{\frac{(1-A_0^2)\nu^2}{2A_0^2}}}{A_0}\erfc\left (\frac{-\nu}{\sqrt{2}A_0}\right )+\erf\left (\frac{-\nu}{\sqrt{2}}\right )+1\rp.
\label{eq:detanalIeer3}
\end{eqnarray}
To solve the optimization problem in (\ref{eq:detanalIeer2}) we will compute the derivatives of $\zeta_{\alpha,\beta}(c_3,\nu,A_0)$ with respect to $c_3$, $\nu$, and $A_0$ and solve the following system of three equations
\begin{equation}\label{eq:detanalIeer3a}
  \frac{d \zeta_{\alpha,\beta}(c_3,\nu,A_0)}{dc_3}=  \frac{d \zeta_{\alpha,\beta}(c_3,\nu,A_0)}{d\nu}=  \frac{d \zeta_{\alpha,\beta}(c_3,\nu,A_0)}{dA_0}=0.
\end{equation}
We will in optimizations below consider what will refer to as the hard regime, i.e. we will consider $(\alpha,\beta)$ that ensure that the optima of the underlying optimizations are achieved nontrivially i.e. not on the boundaries (analysis of possible boundary optima is a small subcase and highly trivial compared to what we will present throughout the paper and we leave it as an easy exercise). Now, clearly, the above problem is fairly hard and at first glance it does not seem that there is much of a reason to believe that even after presumably lengthy computation of the above derivatives one would arrive anywhere close to the explicit solution. However, by a pure magic of mathematics this will turn out to be false and one can actually indeed solve the above system of equations. Before reaching the level where this will be clear a decent amount of patience may be needed. A few quick observations will also turn out to be very useful. We start with one of them that relates to the derivative with respect to $c_3$ and observe that for this derivative from \cite{Stojnicl1RegPosasymldp} one immediately has
\begin{equation}
% \nonumber % Remove numbering (before each equation)
\frac{d\zeta_{\alpha,\beta}(c_3,\nu,A_0)}{dc_3}=-c_3+\frac{c_3}{1-A_0^2}+\frac{c_3-\sqrt{(c_3)^2+4\alpha}}{2}.
\label{eq:detanalIeer9}
\end{equation}
Setting further the above derivative to zero implies
\begin{equation}\label{eq:detanalIeer12}
  c_3=\frac{(1-A_0^2)\sqrt{\alpha}}{A_0}.
\end{equation}
The derivatives with respect to $\nu$ and $A_0$ are a bit more involved. For the derivative with respect to $\nu$ we have
\begin{eqnarray}
% \nonumber % Remove numbering (before each equation)
\frac{d\zeta_{\alpha,\beta}(c_3,\nu,A_0)}{d\nu}&=&\frac{d}{d\nu}\left (-\frac{c_3^2}{2}+I_{sph}+(1-\beta)\log{w_1}+\beta\log{w_2}+\frac{c_3^2}{2(1-A_0^2)}\right )\nonumber \\
&=& \frac{1-\beta}{w_1}\left (\frac{(1-A_0^2)\nu}{A_0^2}\frac{e^{\frac{(1-A_0^2)\nu^2}{2A_0^2}}}{A_0}\erfc\left (\frac{\nu}{\sqrt{2}A_0}\right )-\frac{1-A_0^2}{A_0^2}\frac{2e^{-\frac{\nu^2}{2}}}{\sqrt{2}\sqrt{\pi}}\right)\nonumber \\
&& -\frac{\beta}{w_2}\left (-\frac{(1-A_0^2)\nu}{A_0^2}\frac{e^{\frac{(1-A_0^2)\nu^2}{2A_0^2}}}{A_0}\erfc\left (\frac{-\nu}{\sqrt{2}A_0}\right )-\frac{1-A_0^2}{A_0^2}\frac{2e^{-\frac{\nu^2}{2}}}{\sqrt{2}\sqrt{\pi}}\right ).
\label{eq:detanalIeer4}
\end{eqnarray}
To facilitate the exposition we set
\begin{eqnarray}
y_1 & = & \frac{\nu}{\sqrt{2}}\nonumber \\
y_2 & = & \frac{\nu}{\sqrt{2}A_0}=\frac{y_1}{A_0}.\label{eq:detanalIeer4a}
\end{eqnarray}
Then we also set
\begin{eqnarray}
z_{1,\nu} & = & \frac{(1-\beta)y_1\lp\sqrt{2}y_2e^{y_2^2}\erfc(y_2)-\sqrt{\frac{2}{\pi}}\rp}{y_2e^{y_2^2}\erfc(y_2)+y_1e^{y_1^2}\erfc(-y_1)}\nonumber \\
z_{2,\nu} & = & -\frac{\beta y_1\lp -\sqrt{2}y_2e^{y_2^2}\erfc(-y_2)-\sqrt{\frac{2}{\pi}}\rp}{y_2e^{y_2^2}\erfc(-y_2)+y_1e^{y_1^2}\erfc(y_1)}.\label{eq:detanalIeer4b}
\end{eqnarray}
Combining (\ref{eq:detanalIeer3}), (\ref{eq:detanalIeer4}), (\ref{eq:detanalIeer4a}), and (\ref{eq:detanalIeer4b}) we obtain
\begin{eqnarray}
% \nonumber % Remove numbering (before each equation)
\frac{d\zeta_{\alpha,\beta}(c_3,\nu,A_0)}{d\nu}
&=& \frac{1-\beta}{w_1}\left (\frac{(1-A_0^2)\nu}{A_0^2}\frac{e^{\frac{(1-A_0^2)\nu^2}{2A_0^2}}}{A_0}\erfc\left (\frac{\nu}{\sqrt{2}A_0}\right )-\frac{1-A_0^2}{A_0^2}\frac{2e^{-\frac{\nu^2}{2}}}{\sqrt{2}\sqrt{\pi}}\right)\nonumber \\
&& -\frac{\beta}{w_2}\left (-\frac{(1-A_0^2)\nu}{A_0^2}\frac{e^{\frac{(1-A_0^2)\nu^2}{2A_0^2}}}{A_0}\erfc\left (\frac{-\nu}{\sqrt{2}A_0}\right )-\frac{1-A_0^2}{A_0^2}\frac{2e^{-\frac{\nu^2}{2}}}{\sqrt{2}\sqrt{\pi}}\right )\nonumber \\
& = & \frac{1-A_0^2}{A_0^2}\lp z_{1,\nu}+z_{2,\nu} \rp.
\label{eq:detanalIeer4c}
\end{eqnarray}
From (\ref{eq:detanalIeer4b}) and (\ref{eq:detanalIeer4c}) we then also have
\begin{eqnarray}
% \nonumber % Remove numbering (before each equation)
& & \frac{d\zeta_{\alpha,\beta}(c_3,\nu,A_0)}{d\nu}
 =  \frac{1-A_0^2}{A_0^2}\lp z_{1,\nu}+z_{2,\nu} \rp=0\nonumber \\
& \Longrightarrow & \lp z_{1,\nu}+z_{2,\nu} \rp=\frac{(1-\beta)y_1\lp\sqrt{2}y_2e^{y_2^2}\erfc(y_2)-\sqrt{\frac{2}{\pi}}\rp}{y_2e^{y_2^2}\erfc(y_2)+y_1e^{y_1^2}\erfc(-y_1)}
-\frac{\beta y_1\lp -\sqrt{2}y_2e^{y_2^2}\erfc(-y_2)-\sqrt{\frac{2}{\pi}}\rp}{y_2e^{y_2^2}\erfc(-y_2)+y_1e^{y_1^2}\erfc(y_1)}=0.\nonumber \\
\label{eq:detanalIeer4d}
\end{eqnarray}
Now we switch to the derivative with respect to $A_0$
\begin{eqnarray}
% \nonumber % Remove numbering (before each equation)
\frac{d\zeta_{\alpha,\beta}(c_3,\nu,A_0)}{dA_0}&=&\frac{d}{dA_0}\left (-\frac{c_3^2}{2}+I_{sph}+(1-\beta)\log{w_1}+\beta\log{w_2}+\frac{c_3^2}{2(1-A_0^2)}\right )\nonumber \\
&=& (1-\beta)\frac{d\log{w_1}}{dA_0}+\beta\frac{d\log{w_2}}{dA_0}+\frac{c_3^2A_0}{(1-A_0^2)^2}\nonumber \\
&=& (1-\beta)\frac{d\log{w_1}}{dA_0}+\beta\frac{d\log{w_2}}{dA_0}+\frac{\alpha}{A_0}.\nonumber \\
\label{eq:detanalIeer10}
\end{eqnarray}
From \cite{Stojnicl1RegPosasymldp} we have
\begin{eqnarray}
% \nonumber % Remove numbering (before each equation)
\frac{d\log{w_1}}{dA_0}=\frac{d\log{(\frac{1}{A_0}e^{\frac{\nu^2}{2A_0^2}}\erfc(\frac{\nu}{\sqrt{2}A_0})+e^{\frac{\nu^2}{2}}(\erf(\frac{\nu}{\sqrt{2}})+1))}}{dA_0}=
-\frac{e^{\frac{\nu^2}{2A_0^2}}(A_0^2+\nu^2)\erfc(\frac{\nu}{\sqrt{2}A_0})-\sqrt{\frac{2}{\pi}}A_0\nu}
{A_0^3(e^{\frac{\nu^2}{2A_0^2}}\erfc(\frac{\nu}{\sqrt{2}A_0})+A_0e^{\frac{\nu^2}{2}}(\erf(\frac{\nu}{\sqrt{2}})+1))}.\nonumber \\
\label{eq:detanalIeer11a}
\end{eqnarray}
Analogously to (\ref{eq:detanalIeer11a}) we also have
\begin{eqnarray}
% \nonumber % Remove numbering (before each equation)
\frac{d\log{w_2}}{dA_0}=\frac{d\log{(\frac{1}{A_0}e^{\frac{\nu^2}{2A_0^2}}\erfc(\frac{-\nu}{\sqrt{2}A_0})+e^{\frac{\nu^2}{2}}(\erf(\frac{-\nu}{\sqrt{2}})+1))}}{dA_0}=
-\frac{e^{\frac{\nu^2}{2A_0^2}}(A_0^2+\nu^2)\erfc(\frac{-\nu}{\sqrt{2}A_0})+\sqrt{\frac{2}{\pi}}A_0\nu}
{A_0^3(e^{\frac{\nu^2}{2A_0^2}}\erfc(\frac{-\nu}{\sqrt{2}A_0})+A_0e^{\frac{\nu^2}{2}}(\erf(\frac{-\nu}{\sqrt{2}})+1))}.\nonumber \\
\label{eq:detanalIeer11b}
\end{eqnarray}
Similarly to what was done in (\ref{eq:detanalIeer4b}) we set
\begin{eqnarray}
z_{1,A_0} & = &
-(1-\beta)\frac{y_2^2}{y_1}\frac{\lp(1+2y_2^2)e^{y_2^2}\erfc(y_2)-\sqrt{2}y_2\sqrt{\frac{2}{\pi}}\rp}{(y_2e^{y_2^2}\erfc(y_2)+y_1 e^{y_1^2}\erfc(-y_1))}
\nonumber \\
z_{2,A_0} & = & -\beta\frac{y_2^2}{y_1}\frac{\lp (1+2y_2^2)e^{y_2^2}\erfc(-y_2)+\sqrt{2}y_2\sqrt{\frac{2}{\pi}}\rp}{y_2e^{y_2^2}\erfc(-y_2)+y_1e^{y_1^2}\erfc(y_1)}.\label{eq:detanalIeer11c}
\end{eqnarray}
Combining (\ref{eq:detanalIeer4}), (\ref{eq:detanalIeer10}), (\ref{eq:detanalIeer11a}), (\ref{eq:detanalIeer11b}), and (\ref{eq:detanalIeer11c}) we obtain
\begin{eqnarray}
% \nonumber % Remove numbering (before each equation)
\frac{d\zeta_{\alpha,\beta}(c_3,\nu,A_0)}{dA_0}
= (1-\beta)\frac{d\log{w_1}}{dA_0}+\beta\frac{d\log{w_2}}{dA_0}+\frac{\alpha}{A_0}=z_{1,A_0}+z_{1,A_0}+\frac{\alpha y_2}{y_1}.\nonumber \\
\label{eq:detanalIeer11d}
\end{eqnarray}
From (\ref{eq:detanalIeer11c}) and (\ref{eq:detanalIeer11d}) we also have
\begin{eqnarray}
% \nonumber % Remove numbering (before each equation)
& & \frac{d\zeta_{\alpha,\beta}(c_3,\nu,A_0)}{dA_0}
 =  z_{1,A_0}+z_{1,A_0}+\frac{\alpha y_2}{y_1}=0\nonumber \\
& \Longrightarrow &  -(1-\beta)y_2\frac{\lp(1+2y_2^2)e^{y_2^2}\erfc(y_2)-\sqrt{2}y_2\sqrt{\frac{2}{\pi}}\rp}{y_2e^{y_2^2}\erfc(y_2)+y_1 e^{y_1^2}\erfc(-y_1)}
-\beta y_2\frac{\lp (1+2y_2^2)e^{y_2^2}\erfc(-y_2)+\sqrt{2}y_2\sqrt{\frac{2}{\pi}}\rp}{y_2e^{y_2^2}\erfc(-y_2)+y_1e^{y_1^2}\erfc(y_1)}+
\alpha
=0.\nonumber \\
\label{eq:detanalIeer11e}
\end{eqnarray}
Combining further (\ref{eq:detanalIeer4d}) and (\ref{eq:detanalIeer11e}) we also have
\begin{eqnarray}
% \nonumber % Remove numbering (before each equation)
& & \frac{d\zeta_{\alpha,\beta}(c_3,\nu,A_0)}{dA_0}
 =  z_{1,A_0}+z_{1,A_0}+\frac{\alpha y_2}{y_1}=0\nonumber \\
& \Longrightarrow &  -(1-\beta)y_2\frac{ e^{y_2^2}\erfc(y_2)}{y_2e^{y_2^2}\erfc(y_2)+y_1 e^{y_1^2}\erfc(-y_1)}
-\beta y_2\frac{e^{y_2^2}\erfc(-y_2)}{y_2e^{y_2^2}\erfc(-y_2)+y_1e^{y_1^2}\erfc(y_1)}+
\alpha
=0\nonumber \\
& \Longrightarrow &  -(1-\beta)\frac{\sqrt{\frac{1}{\pi}}}{y_2e^{y_2^2}\erfc(y_2)+y_1 e^{y_1^2}\erfc(-y_1)}
+\beta \frac{\sqrt{\frac{1}{\pi}}}{y_2e^{y_2^2}\erfc(-y_2)+y_1e^{y_1^2}\erfc(y_1)}+
\alpha
=0\nonumber \\
& \Longrightarrow &  -\frac{1-\beta}{y_2e^{y_2^2}\erfc(y_2)+y_1 e^{y_1^2}\erfc(-y_1)}
+ \frac{\beta}{y_2e^{y_2^2}\erfc(-y_2)+y_1e^{y_1^2}\erfc(y_1)}+
\alpha\sqrt{\pi}
=0.\nonumber \\
\label{eq:detanalIeer11f}
\end{eqnarray}
From (\ref{eq:detanalIeer4d}) we further have
\begin{eqnarray}
\frac{\beta}{y_2e^{y_2^2}\erfc(-y_2)+y_1e^{y_1^2}\erfc(y_1)}=
\frac{(1-\beta)}{(y_2e^{y_2^2}\erfc(y_2)+y_1e^{y_1^2}\erfc(-y_1))}
\frac{\lp\sqrt{2}y_2e^{y_2^2}\erfc(y_2)-\sqrt{\frac{2}{\pi}}\rp}{\lp -\sqrt{2}y_2e^{y_2^2}\erfc(-y_2)-\sqrt{\frac{2}{\pi}}\rp}.\nonumber \\
\label{eq:detanalIeer11g}
\end{eqnarray}
Combining further (\ref{eq:detanalIeer11f}) and (\ref{eq:detanalIeer11g})
\begin{eqnarray}
% \nonumber % Remove numbering (before each equation)
& & \frac{d\zeta_{\alpha,\beta}(c_3,\nu,A_0)}{dA_0}
 =  z_{1,A_0}+z_{1,A_0}+\frac{\alpha y_2}{y_1}=0\nonumber \\
& \Longrightarrow &  -\frac{1-\beta}{y_2e^{y_2^2}\erfc(y_2)+y_1 e^{y_1^2}\erfc(-y_1)}
+ \frac{\beta}{y_2e^{y_2^2}\erfc(-y_2)+y_1e^{y_1^2}\erfc(y_1)}+
\alpha\sqrt{\pi}
=0\nonumber \\
& \Longrightarrow &  \frac{1-\beta}{y_2e^{y_2^2}\erfc(y_2)+y_1 e^{y_1^2}\erfc(-y_1)}
\lp -1+
\frac{\lp\sqrt{2}y_2e^{y_2^2}\erfc(y_2)-\sqrt{\frac{2}{\pi}}\rp}{\lp -\sqrt{2}y_2e^{y_2^2}\erfc(-y_2)-\sqrt{\frac{2}{\pi}}\rp}\rp+
\alpha\sqrt{\pi}
=0\nonumber \\
& \Longrightarrow &  y_2e^{y_2^2}\erfc(y_2)+y_1 e^{y_1^2}\erfc(-y_1)=\frac{1-\beta}{\alpha}
\lp
\frac{y_2e^{y_2^2}\erfc(-y_2)+y_2e^{y_2^2}\erfc(y_2)}{\sqrt{\pi}y_2e^{y_2^2}\erfc(-y_2)+1}\rp\nonumber \\
& \Longrightarrow &  y_1 e^{y_1^2}\erfc(-y_1)=\frac{1-\beta}{\alpha}
\lp
\frac{2y_2e^{y_2^2}}{\sqrt{\pi}y_2e^{y_2^2}\erfc(-y_2)+1}\rp-y_2e^{y_2^2}\erfc(y_2)\triangleq f_1(y_2;\alpha,\beta).\nonumber \\
\label{eq:detanalIeer11h}
\end{eqnarray}
One can also combine (\ref{eq:detanalIeer11f}) and (\ref{eq:detanalIeer11g}) in the following alternative way to obtain
\begin{eqnarray}
% \nonumber % Remove numbering (before each equation)
& & \frac{d\zeta_{\alpha,\beta}(c_3,\nu,A_0)}{dA_0}
 =  z_{1,A_0}+z_{1,A_0}+\frac{\alpha y_2}{y_1}=0\nonumber \\
& \Longrightarrow &  -\frac{1-\beta}{y_2e^{y_2^2}\erfc(y_2)+y_1 e^{y_1^2}\erfc(-y_1)}
+ \frac{\beta}{y_2e^{y_2^2}\erfc(-y_2)+y_1e^{y_1^2}\erfc(y_1)}+
\alpha\sqrt{\pi}
=0\nonumber \\
& \Longrightarrow &  \frac{\beta}{y_2e^{y_2^2}\erfc(-y_2)+y_1 e^{y_1^2}\erfc(y_1)}
\lp 1-
\frac{\lp -\sqrt{2}y_2e^{y_2^2}\erfc(-y_2)-\sqrt{\frac{2}{\pi}}\rp}{\lp\sqrt{2}y_2e^{y_2^2}\erfc(y_2)-\sqrt{\frac{2}{\pi}}\rp}\rp+
\alpha\sqrt{\pi}
=0\nonumber \\
& \Longrightarrow &  y_2e^{y_2^2}\erfc(-y_2)+y_1 e^{y_1^2}\erfc(y_1)=\frac{\beta}{\alpha}
\lp
\frac{y_2e^{y_2^2}\erfc(-y_2)+y_2e^{y_2^2}\erfc(y_2)}{-\sqrt{\pi}y_2e^{y_2^2}\erfc(y_2)+1}\rp\nonumber \\
& \Longrightarrow &  y_1 e^{y_1^2}\erfc(y_1)=\frac{\beta}{\alpha}
\lp
\frac{2y_2e^{y_2^2}}{-\sqrt{\pi}y_2e^{y_2^2}\erfc(y_2)+1}\rp-y_2e^{y_2^2}\erfc(-y_2)=-f_1(-y_2;\alpha,1-\beta).\nonumber \\
\label{eq:detanalIeer11i}
\end{eqnarray}
%From (\ref{eq:detanalIeer11h}) and (\ref{eq:detanalIeer11i}) one also has
%\begin{equation}
%\frac{1+\erf(y_1)}{1-\erf(y_1)}= \frac{\erfc(-y_1)}{\erfc(y_1)}=\frac{\frac{1-\beta}{\alpha}
%\lp
%\frac{2y_2e^{y_2^2}}{\sqrt{\pi}y_2e^{y_2^2}\erfc(-y_2)+1}\rp-y_2e^{y_2^2}\erfc(y_2)}{\frac{\beta}{\alpha}
%\lp
%\frac{2y_2e^{y_2^2}}{-\sqrt{\pi}y_2e^{y_2^2}\erfc(y_2)+1}\rp-y_2e^{y_2^2}\erfc(-y_2)}=
%\frac{\frac{1-\beta}{\alpha}
%\lp
%\frac{2}{\sqrt{\pi}y_2e^{y_2^2}\erfc(-y_2)+1}\rp-\erfc(y_2)}{\frac{\beta}{\alpha}
%\lp
%\frac{2}{-\sqrt{\pi}y_2e^{y_2^2}\erfc(y_2)+1}\rp-\erfc(-y_2)}. \\
%\label{eq:detanalIeer11j}
%\end{equation}
%Finally after solving over $\erf(y_1)$ we obtain
%\begin{equation}
%\erf(y_1)=
%\frac{\frac{1-\beta}{\alpha}
%\lp
%\frac{2}{\sqrt{\pi}y_2e^{y_2^2}\erfc(-y_2)+1}\rp-\erfc(y_2)-\lp\frac{\beta}{\alpha}
%\lp
%\frac{2}{-\sqrt{\pi}y_2e^{y_2^2}\erfc(y_2)+1}\rp-\erfc(-y_2)\rp}{\frac{1-\beta}{\alpha}
%\lp
%\frac{2}{\sqrt{\pi}y_2e^{y_2^2}\erfc(-y_2)+1}\rp-\erfc(y_2)+\frac{\beta}{\alpha}
%\lp
%\frac{2}{-\sqrt{\pi}y_2e^{y_2^2}\erfc(y_2)+1}\rp-\erfc(-y_2)},\nonumber \\
%\label{eq:detanalIeer11k}
%\end{equation}
%or
%\begin{equation}
%\erf(y_1)=
%\frac{
%\frac{1-\beta}{\sqrt{\pi}y_2e^{y_2^2}\erfc(-y_2)+1}-
%\frac{\beta}{-\sqrt{\pi}y_2e^{y_2^2}\erfc(y_2)+1}+\alpha\erf(y_2)}{
%\frac{1-\beta}{\sqrt{\pi}y_2e^{y_2^2}\erfc(-y_2)+1}+
%\frac{\beta}{-\sqrt{\pi}y_2e^{y_2^2}\erfc(y_2)+1}-\alpha}\triangleq f_1(y_2;\alpha,\beta).\nonumber \\
%\label{eq:detanalIeer11l}
%\end{equation}
From (\ref{eq:detanalIeer11h}) and (\ref{eq:detanalIeer11i}) one also has
\begin{equation}
\frac{1+\erf(y_1)}{1-\erf(y_1)}= \frac{\erfc(-y_1)}{\erfc(y_1)}=-\frac{f_1(y_2;\alpha,\beta)}{f_1(-y_2;\alpha,1-\beta)}. \\
\label{eq:detanalIeer11j}
\end{equation}
Finally after solving over $\erf(y_1)$ we obtain
\begin{eqnarray}
& & \erf(y_1)  =
\frac{f_1(y_2;\alpha,\beta)+f_1(-y_2;\alpha,1-\beta)}{f_1(y_2;\alpha,\beta)-f_1(-y_2;\alpha,1-\beta)}\nonumber \\
&\Longleftrightarrow & y_1  =\erfinv\lp
\frac{f_1(y_2;\alpha,\beta)+f_1(-y_2;\alpha,1-\beta)}{f_1(y_2;\alpha,\beta)-f_1(-y_2;\alpha,1-\beta)}\rp.
\label{eq:detanalIeer11k}
\end{eqnarray}
A combination of (\ref{eq:detanalIeer11h}) (or (\ref{eq:detanalIeer11i})) and (\ref{eq:detanalIeer11k}) gives
\begin{eqnarray}
2\erfinv\lp
\frac{f_1(y_2;\alpha,\beta)+f_1(-y_2;\alpha,1-\beta)}{f_1(y_2;\alpha,\beta)-f_1(-y_2;\alpha,1-\beta)}\rp
\frac{e^{\erfinv\lp
\frac{f_1(y_2;\alpha,\beta)+f_1(-y_2;\alpha,1-\beta)}{f_1(y_2;\alpha,\beta)-f_1(-y_2;\alpha,1-\beta)}\rp^2}}
{f_1(y_2;\alpha,\beta)-f_1(-y_2;\alpha,1-\beta)}
=1,
\label{eq:detanalIeer11l}
\end{eqnarray}
which can be used to determine $y_2$. Once $y_2$ is determined one can obtain $y_1$ from (\ref{eq:detanalIeer11h}) and (\ref{eq:detanalIeer11i}) (basically (\ref{eq:detanalIeer11k})) and using (\ref{eq:detanalIerr1}), (\ref{eq:detanalIeer12}), and (\ref{eq:detanalIeer4a}), $\nu$, $A_0$, $c_3$, and $\gamma$ from the following
\begin{eqnarray}
  \nu & = & y_2\sqrt{2} \nonumber \\
  A_0 & = & \frac{y_1}{y_2} \nonumber \\
  c_3 & = & \frac{(1-A_0^2)\sqrt{\alpha}}{A_0}\nonumber \\
  \gamma & = & \frac{c_3}{2(1-A_0)}.
\label{eq:detanalIeer11m}
\end{eqnarray}
For the above choice one can then finally compute $\zeta_{\alpha,\beta}(c_3,\nu,A_0)$ using (\ref{eq:detanalIeer3}) in the following way. First, we note that from \cite{Stojnicl1RegPosasymldp} one has
\begin{equation}\label{eq:detanalIeer11ma}
  I_{sph}=-\frac{(1-A_0^2)\alpha}{2}+\alpha\log(A_0).
\end{equation}
Now, using (\ref{eq:detanalIeer11h}), (\ref{eq:detanalIeer11i}), (\ref{eq:detanalIeer11m}), and (\ref{eq:detanalIeer11ma}) we have
\begin{eqnarray}
I^{(bun,ub)}_{err,u}=\zeta_{\alpha,\beta}(c_3,\nu,A_0)&=& -\frac{c_3^2}{2}+I_{sph}+(1-\beta)\log{w_1}+\beta\log{w_2}+\frac{c_3^2}{2(1-A_0^2)}\nonumber \\
& = & -\frac{c_3^2}{2}-\frac{(1-A_0^2)\alpha}{2}+\alpha\log(A_0)+(1-\beta)\log{w_1}+\beta\log{w_2}+\frac{c_3^2}{2(1-A_0^2)} \nonumber \\
& = & \alpha\log(A_0)+(1-\beta)\log{w_1}+\beta\log{w_2} \nonumber \\
& = & \alpha\log\lp\frac{y_1}{y_2}\rp+(1-\beta)\log \lp e^{y_2^2}y_2\erfc(y_2)+e^{y_1^2}y_1\erfc(-y_1)\rp\nonumber \\
& & +\beta\log\lp e^{y_2^2}y_2\erfc(-y_2)+e^{y_1^2}y_1\erfc(y_1)\rp -y_1^2-\log(y_1)-\log(2)\nonumber \\
& = & \alpha\log\lp\frac{y_1}{y_2}\rp+(1-\beta)\log \lp \frac{1-\beta}{\alpha\lp\sqrt{\pi}e^{y_2^2}y_2\erfc(-y_2)+1\rp}\rp\nonumber \\
& & +\beta\log\lp \frac{\beta}{\alpha\lp -\sqrt{\pi}e^{y_2^2}y_2\erfc(y_2)+1\rp}\rp + y_2^2-y_1^2-\log\lp\frac{y_1}{y_2}\rp.\nonumber \\
& = & (\alpha-1)\log\lp\frac{y_1}{y_2}\rp+(1-\beta)\log \lp \frac{1-\beta}{\alpha\lp\sqrt{\pi}e^{y_2^2}y_2\erfc(-y_2)+1\rp}\rp\nonumber \\
& & +\beta\log\lp \frac{\beta}{\alpha\lp -\sqrt{\pi}e^{y_2^2}y_2\erfc(y_2)+1\rp}\rp + y_2^2-y_1^2.\nonumber \\
\label{eq:detanalIeer11mb}
\end{eqnarray}
From the discussion above we have that the choice for $\nu$, $A_0$, and $c_3$ given in (\ref{eq:detanalIeer11l}) and (\ref{eq:detanalIeer11m}) ensures that (\ref{eq:detanalIeer3a}) is satisfied. Examining further the properties of the underlying functions one can also argue that this choice is not only a stationary point, but also a global optimum in (\ref{eq:detanalIeer2}). We skip pursuing these considerations further here. Instead, we will connect what was obtained above to another set of optimizing quantities that we will obtain through a different set of considerations and present later on (in fact, quite a lot more will turn out to be true, not only will the choice for $\nu$, $A_0$, and $c_3$ given in (\ref{eq:detanalIeer11l}) and (\ref{eq:detanalIeer11m}) turn out to be precisely the one that solves the optimization in (\ref{eq:detanalIeer2}) but also precisely the one that determines $I^{(bin)}_{err}(\alpha,\beta)$). Here though, we would like to point out that in our view it is quite remarkable that given the hardness of the initial optimization problem a closed form solution of (\ref{eq:detanalIeer2}) could still be obtained. We summarize the above results in the following theorem.
\begin{theorem}
Assume the setup of Theorem \ref{thm:ldp2} and assume that a pair $(\alpha,\beta)$  is given. Let $\alpha>\alpha_w$ where $\alpha_w$ is such that $\psi^{(bin)}_{\beta}(\alpha_w)=\xi^{(bin)}_{\alpha_w}(\beta)=1$. Set
\begin{equation}\label{eq:thmldp3l1PTa}
 f_1(y_2;\alpha,\beta) \triangleq  \frac{1-\beta}{\alpha}
\lp
\frac{2y_2e^{y_2^2}}{\sqrt{\pi}y_2e^{y_2^2}\erfc(-y_2)+1}\rp-y_2e^{y_2^2}\erfc(y_2).
\end{equation}
Also let $y_2$ and $y_1$ satisfy the following \textbf{fundamental characterizations of the binary $\ell_1$'s LDP:}

%\begin{center}
%\shadowbox{$
\begin{eqnarray}
2\erfinv\lp
\frac{f_1(y_2;\alpha,\beta)+f_1(-y_2;\alpha,1-\beta)}{f_1(y_2;\alpha,\beta)-f_1(-y_2;\alpha,1-\beta)}\rp
\frac{e^{\erfinv\lp
\frac{f_1(y_2;\alpha,\beta)+f_1(-y_2;\alpha,1-\beta)}{f_1(y_2;\alpha,\beta)-f_1(-y_2;\alpha,1-\beta)}\rp^2}}
{f_1(y_2;\alpha,\beta)-f_1(-y_2;\alpha,1-\beta)}
& = & 1.\nonumber \\
\label{eq:thmldp3l1PT}
\end{eqnarray}
and
\begin{eqnarray}
y_1  =  \erfinv\lp
\frac{f_1(y_2;\alpha,\beta)+f_1(-y_2;\alpha,1-\beta)}{f_1(y_2;\alpha,\beta)-f_1(-y_2;\alpha,1-\beta)}\rp.\nonumber \\
\label{eq:thmldp3l1PTa}
\end{eqnarray}
%$}
%-\vspace{-.5in}\begin{equation}
%\label{eq:thmldp3l1PT}
%\end{equation}
%\end{center}

\noindent Then
\begin{eqnarray}
I^{(bin)}_{err}(\alpha,\beta)\triangleq\lim_{n\rightarrow\infty}\frac{\log{P_{err}}}{n}
& \leq &
(\alpha-1)\log\lp\frac{y_1}{y_2}\rp+(1-\beta)\log \lp \frac{1-\beta}{\alpha\lp\sqrt{\pi}e^{y_2^2}y_2\erfc(-y_2)+1\rp}\rp\nonumber \\
& & +\beta\log\lp \frac{\beta}{\alpha\lp -\sqrt{\pi}e^{y_2^2}y_2\erfc(y_2)+1\rp}\rp + y_2^2-y_1^2 \triangleq I^{(bin)}_{ldp}(\alpha,\beta).
\label{eq:ldpthm3Ierrub1}
\end{eqnarray}
\noindent Moreover, for the above choice of $y_2$ and  $y_1$, $\nu$, $A_0$, $c_3$, and $\gamma$ in (\ref{eq:detanalIeer11m}) achieve an optimum in (\ref{eq:detanalIeer2}) (and ultimately in (\ref{eq:ldpthm2Ierrub1})).
\label{thm:ldp3}
\end{theorem}
\begin{proof} Follows from the above discussion.
\end{proof}
The above results for the upper tail of the binary $\ell_1$ LDP, remain correct in the lower tail regime. For the completeness we in the following section provide a short argument to confirm that this is indeed correct.

%%%%%%%%%%%%%%%%%%%%%%%%%%%%%%%%%%%%%%%%%%%%%%%%%%%%%%%%%%%%%%%%%
\subsubsection{Lower tail}
\label{sec:lowertail}
%%%%%%%%%%%%%%%%%%%%%%%%%%%%%%%%%%%%%%%%%%%%%%%%%%%%%%%%%%%%%%%%%

%We will split the presentation in this subsection into two parts. In the first part we will discuss the key components of the analysis needed to establish the lower bound on the lower tail.  In the second part we will then design a reverse strategy that produces the corresponding upper bound and ultimately makes the lower tail estimates fully exact.
%
%%%%%%%%%%%%%%%%%%%%%%%%%%%%%%%%%%%%%%%%%%%%%%%%%%%%%%%%%%%%%%%%%%
%\subsubsection{Lower bound}
%\label{sec:lowerlowertail}
%%%%%%%%%%%%%%%%%%%%%%%%%%%%%%%%%%%%%%%%%%%%%%%%%%%%%%%%%%%%%%%%%%

In this section we will quickly formalize the above statements about the lower tail type of large deviations. We rely on the strategy introduced in
\cite{Stojnicl1RegPosasymldp} and quickly have
\begin{equation}
P_{cor}
\leq  \min_{t_1}\min_{c_3\geq 0}
Ee^{c_3\|\g\|_2}Ee^{-c_3w(\h,S_w)}e^{-c_3t_1}/P(g\geq t_1),
\label{eq:ldpprob3lower}
\end{equation}
where obviously $P_{cor}=1-P_{err}$ is the probability that (\ref{eq:l1bin}) does produce the solution of (\ref{eq:l0}). We then have for the rate of $P_{cor}$'s decay
\begin{equation}\label{eq:ldpasymp1lower}
  I^{(bin)}_{cor}(\alpha,\beta)\triangleq\lim_{n\rightarrow\infty}\frac{\log{P_{cor}}}{n}.
\end{equation}
The following theorem is then the lower tail analogue to Theorem \ref{thm:ldp2}.
\begin{theorem}
Assume the setup of Theorem \ref{thm:ldp2}. Then
\begin{eqnarray}
I^{(bin)}_{cor}(\alpha,\beta)& \triangleq & \lim_{n\rightarrow\infty}\frac{\log{P_{cor}}}{n}\nonumber \\
& \leq  &\min_{c_3\geq 0}\left (-\frac{c_3^2}{2}+I_{sph}^++\max_{\nu\geq 0,\gamma^{(s)}\geq 0} ((1-\beta\log{w_1}+\beta\log{w_2}-c_3\gamma)\right )\triangleq I_{cor,l}^{(p,ub)}(\alpha,\beta), \nonumber \\
\label{eq:ldpthm2Icorub1}
\end{eqnarray}
where
\begin{eqnarray}
% \nonumber % Remove numbering (before each equation)
I_{sph}^+ &=& \widehat{\gamma_+}c_3-\frac{\alpha d}{2}\log\left (1-\frac{c_3}{2\widehat{\gamma_+}}\right )\nonumber \\
  \widehat{\gamma_+} &=& \frac{2c_3+\sqrt{4c_3^2+16\alpha }}{8}\nonumber \\
w_1 &=& \frac{1}{2}\lp\frac{1}{\sqrt{2\pi}}\int_{h}e^{-h^2/2}e^{-c_3\max(h-\nu,0)^2/4/\gamma}dh
  =\frac{e^{\frac{-c_3\nu^2/4/\gamma}{1+c_3/2/\gamma}}}{\sqrt{1+c_3/2/\gamma}}\erfc\left (\frac{\nu}{\sqrt{2}\sqrt{1+c_3/2/\gamma}}\right )+\erf\left (\frac{\nu}{\sqrt{2}}\right )+1\rp\nonumber \\
w_2 &=& \frac{1}{2}\lp\frac{1}{\sqrt{2\pi}}\int_{h}e^{-h^2/2}e^{-c_3\max(h+\nu,0)^2/4/\gamma}dh
  =\frac{e^{\frac{-c_3\nu^2/4/\gamma}{1+c_3/2/\gamma}}}{\sqrt{1+c_3/2/\gamma}}\erfc\left (\frac{-\nu}{\sqrt{2}\sqrt{1+c_3/2/\gamma}}\right )+\erf\left (\frac{-\nu}{\sqrt{2}}\right )+1\rp.\nonumber \\
.\label{eq:ldpthm2perrub2lower}
\end{eqnarray}\label{thm:ldp2lower}
\end{theorem}
\begin{proof} Follows in exactly the same way as the result for the lower tail of the standard $\ell_1$ LDP in \cite{Stojnicl1RegPosasymldp}.
\end{proof}
Instead of repeating the procedure from the previous section to solve the above optimization problem, one can just quickly observe that
the change $c_3\rightarrow -c_3$ gives as in Section \ref{sec:analysisIerr}
\begin{equation}\label{eq:detanalIcor1}
  A_{0}\triangleq\sqrt{1-\frac{c_3}{2\gamma}},
\end{equation}
and
\begin{equation}\label{eq:detanalIcor2}
I_{cor,l}^{(p,ub)}(\alpha,\beta)\triangleq \min_{c_3\leq 0}\max_{\nu\geq 0,A_0\leq 1}\zeta_{\alpha,\beta}(c_3,\nu,A_0)
\end{equation}
where
\begin{eqnarray}
% \nonumber % Remove numbering (before each equation)
\zeta_{\alpha,\beta}(c_3,\nu,A_0)&=&\left (-\frac{c_3^2}{2}+I_{sph}^++(1-\beta)\log{w_1}+\beta\log{w_2}+\frac{c_3^2}{2(1-A_0^2)}\right )\nonumber \\
I_{sph}^+ &=& -\widehat{\gamma^+}c_3-\frac{\alpha }{2}\log\left (1+\frac{c_3}{2\widehat{\gamma^+}}\right )\nonumber \\
  \widehat{\gamma^+} &=& \frac{-c_3+\sqrt{c_3^2+4\alpha}}{4}=-\widehat{\gamma},\nonumber \\
\label{eq:detanalIcor3}
\end{eqnarray}
and $w_1$ and $w_2$ are as in (\ref{eq:detanalIeer3}). Also, $\zeta_{\alpha,\beta}(c_3,\nu,A_0)$ defined in (\ref{eq:detanalIcor3}) is exactly the same as the corresponding one in (\ref{eq:detanalIeer3}) which means that one can proceed with the computation of all the derivatives as earlier and the values we have chosen for $c_3$, $\nu$, $\gamma$, and $A_0$ in the upper tail regime will have the same form. The following theorem summarizes the final results (this is of course nothing but a lower tail analogue to Theorem \ref{thm:ldp3}).
\begin{theorem}
Assume the setup of Theorem \ref{thm:ldp3} and assume that a pair $(\alpha,\beta)$  is given. Differently from Theorem \ref{thm:ldp3}, let $\alpha<\alpha_w$ where $\alpha_w$ is such that $\psi^{(bin)}_{\beta}(\alpha_w)=\xi^{(bin)}_{\alpha_w}(\beta)=1$. Also let $y_2$ and $y_1$ satisfy the \textbf{fundamental \emph{binary} $\ell_1$'s LDP characterizations} as in Theorem \ref{thm:ldp3}. Then choosing $\nu$, $c_3$, and $\gamma$ in the optimization problem in (\ref{eq:ldpthm2Icorub1}) as $\nu$, $-c_3$, and $\gamma$  from Theorem \ref{thm:ldp3} (or equivalently, choosing $\nu$, $c_3$, and $A_0$ in the optimization problem in (\ref{eq:detanalIcor2}) as $\nu$, $c_3$, and $A_0$  from Theorem \ref{thm:ldp3}) gives needed
$\zeta_{\alpha,\beta}(c_3,\nu,A_0)$.
\label{thm:ldp3lower}
\end{theorem}
\begin{proof} Follows from the considerations leading to Theorem \ref{thm:ldp3}.
\end{proof}
Clearly, in terms of the translation of the upper tail results to the lower tail, there is really not much difference compared to the standard $\ell_1$ and its considerations from \cite{Stojnicl1RegPosasymldp}. Along the same lines, here we will have for $\alpha<\alpha_w$, $y_1>y_2$ which means $A_0>1$ and finally $c_3<0$. In the upper tail regime (i.e. in Theorem \ref{thm:ldp3}) the reasoning is reversed and $c_3>0$.

%%%%%%%%%%%%%%%%%%%%%%%%%%%%%%%%%%%%%%%%%%%%%%%%%%%%%%%%%%%%%%%%%
\subsection{High-dimensional geometry}
\label{sec:hdg}
%%%%%%%%%%%%%%%%%%%%%%%%%%%%%%%%%%%%%%%%%%%%%%%%%%%%%%%%%%%%%%%%%

The previous section introduced a purely probabilistic approach to deal with the LDPs. In this section though, we take a different path and analyze the LDPs via a high-dimensional integral geometry approach. As in the previous section, we will here again mostly focus on the upper tail regime (the results for the lower tail will automatically follow). In mathematical terms, we again assume that we are given a pair $(\alpha,\beta)$ and that $\alpha>\alpha_w$ (where $\alpha_w$ is such that $\psi^{(bin)}_{\beta}(\alpha_w)=\xi^{(bin)}_{\alpha_w}(\beta)=1$). We will rely on the following observations from \cite{Stojnicl1BnBxfinn}
\begin{equation}\label{eq:hdg1}
  \Psi^{(bin)}_{net}(\alpha,\beta)=I^{(bin)}_{err}(\alpha,\beta)\triangleq\lim_{n\rightarrow\infty}\frac{\log{P_{err}}}{n}
  =\max_{\gamma_g\in(0,\min(1-\alpha,\beta))}
  \lp \psicom+\psiint+\psiext\rp,
\end{equation}
where
\begin{eqnarray}
H(x) & = & x\log(x)+(1-x)\log(1-x)\nonumber \\
\psicom & = & -(1-\beta)H\lp \frac{1-\alpha-\gamma_g}{1-\beta} \rp-\beta H\lp \frac{\beta-\gamma_g}{\beta} \rp\nonumber \\
\psiint & = & \min_{y_i\geq 0} (\alpha y_i^2 +(\alpha-\beta+\gamma_g)\log(\erfc(y_i))+(\beta-\gamma_g)\log(\erfc(-y_i)))- \alpha \log(2)\nonumber\\
\psiext & = & \max_{y_e\geq 0} (-\alpha y_e^2 +(1-\alpha-\gamma_g)\log(\erfc(-y_e))+\gamma_g\log(\erfc(y_e)))-(1-\alpha)\log(2). \label{eq:hdg2}
\end{eqnarray}
One can of course solve the above problem numerically. Here, we will raise the bar a bit higher and look for an explicit characterization of the solution. To that end we start with the following transformation of the above optimization problem
\begin{equation}\label{eq:hdg1a}
\max_{\gamma_g\in(0,\min(1-\alpha,\beta)),y_e\geq 0}\min_{y_i\geq 0} \zeta_{\alpha,\beta}(\gamma_g,y_e,y_i)
\end{equation}
where
\begin{eqnarray}\label{eq:hdg1b}
\zeta_{\alpha,\beta}(\gamma_g,y_e,y_i) & = & -(1-\beta)H\lp \frac{1-\alpha-\gamma_g}{1-\beta} \rp-\beta H\lp \frac{\beta-\gamma_g}{\beta} \rp\nonumber \\
&& +
(\alpha y_i^2 +(\alpha-\beta+\gamma_g)\log(\erfc(y_i))+(\beta-\gamma_g)\log(\erfc(-y_i)))- \alpha \log(2) \nonumber \\
& & +
(-\alpha y_e^2 +(1-\alpha-\gamma_g)\log(\erfc(-y_e))+\gamma_g\log(\erfc(y_e)))-(1-\alpha)\log(2).
\end{eqnarray}
Before proceeding further, we will below establish a few nice properties of $\zeta_{\alpha,\beta}(\gamma_g,y_e,y_i)$. Namely, that for any fixed $(\alpha,\beta)\in (0,1)\times(0,\alpha)$, we will show that $\zeta_{\alpha,\beta}(\gamma_g,y_e,y_i)$ is concave in $\gamma_g$ and $y_e$ and convex in $y_i$ in the optimizing domain of interest.

%%%%%%%%%%%%%%%%%%%%%%%%%%%%%%%%%%%%%%%%%%%%%%%%%%%%%%%%%%%%%%%%%
\subsubsection{Concavity in $\gamma_g$}
\label{sec:hdgconcgammag}
%%%%%%%%%%%%%%%%%%%%%%%%%%%%%%%%%%%%%%%%%%%%%%%%%%%%%%%%%%%%%%%%%

We start by computing the first derivative with respect to $\gamma_g$
\begin{eqnarray}\label{eq:hdg1o}
\frac{d\zeta_{\alpha,\beta}(\gamma_g,y_e,y_i)}{d\gamma_g}
& = &-(1-\beta)\frac{dH\lp \frac{1-\alpha-\gamma_g}{1-\beta} \rp}{d\gamma_g}-\beta\frac{d H\lp \frac{\beta-\gamma_g}{\beta} \rp}{d\gamma_g}+\log\lp\frac{\erfc(y_i)\erfc(y_e)}{\erfc(-y_i)\erfc(-y_e)}\rp\nonumber \\
& = &\log\lp \frac{\frac{1-\alpha-\gamma_g}{1-\beta}}{1-\frac{1-\alpha-\gamma_g}{1-\beta}} \rp+\log \lp \frac{\frac{\beta-\gamma_g}{\beta}}
{1-\frac{\beta-\gamma_g}{\beta}} \rp+\log\lp\frac{\erfc(y_i)\erfc(y_e)}{\erfc(-y_i)\erfc(-y_e)}\rp\nonumber \\
& = &\log\lp \frac{1-\alpha-\gamma_g}{\alpha-\beta+\gamma_g} \rp+\log \lp \frac{\beta-\gamma_g}
{\gamma_g} \rp+\log\lp\frac{\erfc(y_i)\erfc(y_e)}{\erfc(-y_i)\erfc(-y_e)}\rp.
\end{eqnarray}
We then also have for the second derivative
\begin{eqnarray}\label{eq:hdg1oa}
\frac{d^2\zeta_{\alpha,\beta}(\gamma_g,y_e,y_i)}{d\gamma_g^2}
& = & \frac{d\lp\log\lp \frac{1-\alpha-\gamma_g}{\alpha-\beta+\gamma_g} \rp+\log \lp \frac{\beta-\gamma_g}
{\gamma_g} \rp+\log\lp\frac{\erfc(y_i)\erfc(y_e)}{\erfc(-y_i)\erfc(-y_e)}\rp\rp}{d\gamma_g}\nonumber \\
& = & \frac{-\frac{1-\alpha-\gamma_g}{(\alpha-\beta+\gamma_g)^2}-\frac{1}{\alpha-\beta+\gamma_g}}
{\frac{1-\alpha-\gamma_g}{\alpha-\beta+\gamma_g}}
+\frac{-\frac{\beta-\gamma_g}
{\gamma_g^2}-\frac{1}
{\gamma_g}}{\frac{\beta-\gamma_g}
{\gamma_g}}\nonumber \\
& = & \frac{-(1-\beta)}
{(1-\alpha-\gamma_g)(\alpha-\beta+\gamma_g)}
+\frac{-\beta}{\gamma_g(\beta-\gamma_g)}
\nonumber \\
& < & 0.
\end{eqnarray}

%%%%%%%%%%%%%%%%%%%%%%%%%%%%%%%%%%%%%%%%%%%%%%%%%%%%%%%%%%%%%%%%%
\subsubsection{Convexity in $y_i$}
\label{sec:hdgconvyi}
%%%%%%%%%%%%%%%%%%%%%%%%%%%%%%%%%%%%%%%%%%%%%%%%%%%%%%%%%%%%%%%%%

To check convexity in $y_i$ and concavity in $y_e$ we will rely on the following
\begin{eqnarray}\label{eq:hdg1c}
\frac{d\log(\erfc(x))}{dx}=-\frac{2e^{-x^2}}{\sqrt{\pi}\erfc(x)} \quad \mbox{and} \quad \frac{d\log(\erfc(-x))}{dx}=\frac{2e^{-x^2}}{\sqrt{\pi}\erfc(-x)},
\end{eqnarray}
and
\begin{eqnarray}\label{eq:hdg1d}
\frac{d^2\log(\erfc(x))}{dx^2}=\frac{4e^{-x^2}\sqrt{\pi}x\erfc(x)-4e^{-2x^2}}{\pi\erfc(x)^2} \quad \mbox{and} \quad \frac{d^2\log(\erfc(-x))}{dx^2}=\frac{-4e^{-x^2}\sqrt{\pi}x\erfc(-x)-4e^{-2x^2}}{\pi\erfc(-x)^2}.
\end{eqnarray}
As we will eventually need both, the first and the second derivative with respect to $y_i$, we find it convenient to compute them at the same time right here. For the first derivative we have
\begin{eqnarray}\label{eq:hdg1e}
\frac{d\zeta_{\alpha,\beta}(\gamma_g,y_e,y_i)}{dy_i}
& = &\frac{d(\alpha y_i^2 +(\alpha-\beta+\gamma_g)\log(\erfc(y_i))+(\beta-\gamma_g)\log(\erfc(-y_i)))}{dy_i}\nonumber \\
& = & 2\alpha y_i-\frac{2(\alpha-\beta+\gamma_g)e^{-y_i^2}}{\sqrt{\pi}\erfc(y_i)}+\frac{2(\beta-\gamma_g)e^{-y_i^2}}{\sqrt{\pi}\erfc(-y_i)},
\end{eqnarray}
and for the second
\begin{eqnarray}\label{eq:hdg1f}
\frac{d^2\zeta_{\alpha,\beta}(\gamma_g,y_e,y_i)}{dy_i^2}
& = &\frac{d^2(\alpha y_i^2 +(\alpha-\beta+\gamma_g)\log(\erfc(y_i))+(\beta-\gamma_g)\log(\erfc(-y_i)))}{dy_i^2}\nonumber \\
& = & 2\alpha+ (\alpha-\beta+\gamma_g)\frac{4e^{-y_i^2}\sqrt{\pi}y_i\erfc(y_i)-4e^{-2y_i^2}}{\pi\erfc(y_i)^2}\nonumber \\
& & + (\beta-\gamma_g)\frac{-4e^{-y_i^2}\sqrt{\pi}y_i\erfc(-y_i)-4e^{-2y_i^2}}{\pi\erfc(-y_i)^2}\nonumber \\
& = & 2(\alpha-\beta+\gamma_g)\frac{\pi\erfc(y_i)^2+2e^{-y_i^2}\sqrt{\pi}y_i\erfc(y_i)-2e^{-2y_i^2}}{\pi\erfc(y_i)^2}\nonumber \\
& & + 2(\beta-\gamma_g)\frac{\pi\erfc(-y_i)^2-2e^{-y_i^2}\sqrt{\pi}y_i\erfc(-y_i)-2e^{-2y_i^2}}{\pi\erfc(-y_i)^2}.\nonumber \\
\end{eqnarray}
In \cite{Stojnicl1RegPosasymldp} it was argued that $\lp\pi\erfc(y_i)^2+2e^{-y_i^2}\sqrt{\pi}y_i\erfc(y_i)-2e^{-2y_i^2}\rp>0$. In order to show that the above derivative is nonnegative it is then enough to show that $\lp\pi\erfc(-y_i)^2-2e^{-y_i^2}\sqrt{\pi}y_i\erfc(-y_i)-2e^{-2y_i^2}\rp$ is an increasing function of $y_i$ (this is technically needed only on $y_i\geq 0$). This will be automatically implied if $\lp\sqrt{\pi}\erfc(-y_i)-2e^{-y_i^2}y_i\rp$ is increasing in $y_i$. The following simple argument shows that this is indeed the case.
\begin{eqnarray}\label{eq:hdg1g}
\frac{d\lp\sqrt{\pi}\erfc(-y_i)-2e^{-y_i^2}y_i\rp}{dy_i}=2e^{-y_i^2}-2e^{-y_i^2}+4y_i^2\geq 0.
\end{eqnarray}
All of the above then implies that $\zeta_{\alpha,\beta}(\gamma_g,y_e,y_i)$ is indeed convex in $y_i$ in the domain of interest (of course, a bit more is true based on the above but we stop short of discussing it as it goes beyond what is needed here).

%%%%%%%%%%%%%%%%%%%%%%%%%%%%%%%%%%%%%%%%%%%%%%%%%%%%%%%%%%%%%%%%%
\subsubsection{Concavity in $y_e$}
\label{sec:hdgconcye}
%%%%%%%%%%%%%%%%%%%%%%%%%%%%%%%%%%%%%%%%%%%%%%%%%%%%%%%%%%%%%%%%%

To check concavity in $y_e$ we proceed as above. For the first derivative we have
\begin{eqnarray}\label{eq:hdg1h}
\frac{d\zeta_{\alpha,\beta}(\gamma_g,y_e,y_i)}{dy_e}
& = &\frac{d(-\alpha y_e^2 +(1-\alpha-\gamma_g)\log(\erfc(-y_e))+\gamma_g\log(\erfc(y_e)))}{dy_e}\nonumber \\
& = & -2\alpha y_e+\frac{2(1-\alpha-\gamma_g)e^{-y_e^2}}{\sqrt{\pi}\erfc(-y_e)}-\frac{2\gamma_ge^{-y_e^2}}{\sqrt{\pi}\erfc(y_e)},
\end{eqnarray}
and for the second
\begin{eqnarray}\label{eq:hdg1i}
\frac{d^2\zeta_{\alpha,\beta}(\gamma_g,y_e,y_i)}{dy_e^2}
& = &\frac{d^2(-\alpha y_e^2 +(1-\alpha-\gamma_g)\log(\erfc(-y_e))+\gamma_g\log(\erfc(y_e)))}{dy_e^2}\nonumber \\
& = & -2\alpha+ (1-\alpha-\gamma_g)\frac{-4e^{-y_e^2}\sqrt{\pi}y_e\erfc(-y_e)-4e^{-2y_e^2}}{\pi\erfc(-y_e)^2}\nonumber \\
& & +\gamma_g\frac{4e^{-y_e^2}\sqrt{\pi}y_e\erfc(y_e)-4e^{-2y_e^2}}{\pi\erfc(y_e)^2}\nonumber \\
\end{eqnarray}
If one can show that what multiplies $\gamma_g$ in the above expression is not positive then the second derivative would also not be positive. To that end we have
\begin{eqnarray}\label{eq:hdg1ja}
-\gamma_g\frac{-e^{y_e^2}\sqrt{\pi}y_e\erfc(-y_e)-1}{\pi\erfc(-y_e)^2}
+\gamma_g\frac{e^{y_e^2}\sqrt{\pi}y_e\erfc(y_e)-1}{\pi\erfc(y_e)^2}
=\gamma_g\frac{(2\sqrt{\pi} y_e\erfc(-y_e)\erfc(y_e) e^{y_e^2}-4\erf(y_e))
}{\pi\erfc(y_e)^2\erfc(-y_e)^2}, \nonumber \\
\end{eqnarray}
and if we show that $(2\sqrt{\pi} y_e\erfc(-y_e)\erfc(y_e) e^{y_e^2}-4\erf(y_e))\leq 0$ when $y_e\geq 0$ then the second derivative in (\ref{eq:hdg1i}) would not be positive. There are many ways how this can be shown. Here we rely on the following well known inequalities
\begin{equation}
\frac{2}{\sqrt{\pi}}\frac{e^{-y^2}}{y+\sqrt{y^2+2}}< \erfc(y)\leq \frac{2}{\sqrt{\pi}}\frac{e^{-y^2}}{y+\sqrt{y^2+\frac{4}{\pi}}}.
 \label{eq:hdg10a}
\end{equation}
Using (\ref{eq:hdg10a}) we have
\begin{eqnarray}\label{eq:hdg1j}
2\sqrt{\pi} y_e\erfc(-y_e)\erfc(y_e) e^{y_e^2}-4\erf(y_e)\leq
\frac{4y_e\erfc(-y_e)}{y_e+\sqrt{y_e^2+\frac{4}{\pi}}}-4\erf(y_e)
=\lp\frac{4\erfc(-y_e)}{1+\sqrt{1+\frac{4}{\pi y_e^2}}}-\frac{4\erfc(-y_e)}{\frac{1}{\erf(y_e)}+1}\rp. \nonumber \\
\end{eqnarray}
From (\ref{eq:hdg1j}), we have that if $\erf(y_e)\geq \frac{y_e}{\sqrt{y_e^2+\frac{4}{\pi}}}$ then $(2\sqrt{\pi} y_e\erfc(-y_e)\erfc(y_e) e^{y_e^2}-4\erf(y_e))\leq 0$. Based on (\ref{eq:hdg10a}) we have that
\begin{eqnarray}\label{eq:hdg1k}
\frac{y_e}{\sqrt{y_e^2+\frac{4}{\pi}}}\leq 1-\frac{2}{\sqrt{\pi}} \frac{e^{-y_e^2}}{y_e+\sqrt{y_e^2+\frac{4}{\pi}}} \nonumber \\
\end{eqnarray}
implies
\begin{eqnarray}\label{eq:hdg1l}
\erf(y_e)\geq \frac{y_e}{\sqrt{y_e^2+\frac{4}{\pi}}}. \nonumber \\
\end{eqnarray}
Transforming (\ref{eq:hdg1k}) further we have
\begin{eqnarray}\label{eq:hdg1m}
& & \frac{y_e}{\sqrt{y_e^2+\frac{4}{\pi}}}\leq 1-\frac{2}{\sqrt{\pi}} \frac{e^{-y_e^2}}{y_e+\sqrt{y_e^2+\frac{4}{\pi}}} \nonumber \\
& \Longleftrightarrow & e^{-y_e^2}
\leq \frac{2}{\sqrt{y_e^2\pi+4}}\nonumber \\
& \Longleftrightarrow &
y_e^2\frac{\pi}{4}+1\leq e^{2y_e^2}.\nonumber \\
\end{eqnarray}
The last inequality holds and one then also has based on the above that (\ref{eq:hdg1l}) holds which based on (\ref{eq:hdg1j}) implies that the right side of the equality in (\ref{eq:hdg1ja}) is not positive. This is then enough to conclude that the second derivative in (\ref{eq:hdg1i}) is not positive and $\zeta_{\alpha,\beta}(\gamma_g,y_e,y_i)$ is indeed concave in $y_e$ on $y_e\geq 0$.

%%%%%%%%%%%%%%%%%%%%%%%%%%%%%%%%%%%%%%%%%%%%%%%%%%%%%%%%%%%%%%%%%
\subsubsection{Solving the derivative equations}
\label{sec:hdgdereqns}
%%%%%%%%%%%%%%%%%%%%%%%%%%%%%%%%%%%%%%%%%%%%%%%%%%%%%%%%%%%%%%%%%

While the above properties of $\zeta_{\alpha,\beta}(\gamma_g,y_e,y_i)$ are nice and welcome one still needs to solve the following system of equations
\begin{equation}\label{eq:hdg1n}
  \frac{d\zeta_{\alpha,\beta}(\gamma_g,y_e,y_i)}{d\gamma_g}= \frac{d\zeta_{\alpha,\beta}(\gamma_g,y_e,y_i)}{dy_e}
=  \frac{d\zeta_{\alpha,\beta}(\gamma_g,y_e,y_i)}{dy_i}=0.
\end{equation}
The derivatives with respect to $\gamma_g$, $y_i$, ad $y_e$ are computed in (\ref{eq:hdg1o}), (\ref{eq:hdg1e}), and (\ref{eq:hdg1h}), respectively.
Setting the derivative in (\ref{eq:hdg1e}) to zero gives
\begin{eqnarray}\label{eq:hdg1p}
& & \gamma_g\lp\frac{1}{\erfc(y_i)}+ \frac{1}{\erfc(-y_i)}\rp=
\sqrt{\pi}\alpha e^{y_i^2} y_i-\frac{(\alpha-\beta)}{\erfc(y_i)}+\frac{(\beta)}{\erfc(-y_i)}\nonumber \\
&\Longleftrightarrow & \gamma_g=
\frac{\alpha}{2}\lp\sqrt{\pi} e^{y_i^2} y_i\erfc(y_i)\erfc(-y_i)-\erfc(-y_i)\rp+\beta.
\end{eqnarray}
After setting
\begin{equation}\label{eq:hdg1q}
  A_{bin}=\frac{\alpha-\beta+\gamma_g}{1-\alpha-\gamma_g}\frac{\gamma_g}{\beta-\gamma_g},
\end{equation}
from (\ref{eq:hdg1o}) one has
\begin{eqnarray}\label{eq:hdg1r}
& &   A_{bin}=\frac{\erfc(y_i)\erfc(y_e)}{\erfc(-y_i)\erfc(-y_e)}\nonumber \\
& \Longleftrightarrow &  \frac{\erfc(-y_e)}{\erfc(y_e)}=\frac{\erfc(y_i)}{\erfc(-y_i)A_{bin} }\nonumber \\
& \Longleftrightarrow &  y_e=\erfinv\lp\frac{\erfc(y_i)-A_{bin}\erfc(-y_i)}{\erfc(y_i)+A_{bin}\erfc(-y_i)}\rp.
\end{eqnarray}
Plugging $\gamma_g$ from (\ref{eq:hdg1p}) and $y_e$ from (\ref{eq:hdg1r}) into the right side of (\ref{eq:hdg1h}) and equaling it to zero one finally obtains a single equation with $y_i$ as the only unknown. After finding $y_i$ in this way one can reuse it to obtain $\gamma_g$ through (\ref{eq:hdg1p}) and $y_e$ through (\ref{eq:hdg1r}). Instead of doing this we will present a different path that will be more connected to what was presented in Section \ref{sec:ldp}. We start by setting the derivative in (\ref{eq:hdg1h}) to zero to obtain
\begin{eqnarray}\label{eq:hdg1s}
& & \gamma_g\lp\frac{1}{\erfc(y_e)}+ \frac{1}{\erfc(-y_e)}\rp=
-\sqrt{\pi}\alpha e^{y_e^2} y_e+\frac{1-\alpha}{\erfc(-y_e)}\nonumber \\
&\Longleftrightarrow & \gamma_g=
\frac{\alpha}{2}\lp -\sqrt{\pi} e^{y_e^2} y_e\erfc(-y_e)-1\rp\erfc(y_e)+\frac{\erfc(y_e)}{2}.
\end{eqnarray}
Plugging $\gamma_g$ from (\ref{eq:hdg1p}) into (\ref{eq:hdg1q}) gives
\begin{equation}\label{eq:hdg1t}
  A_{bin}=\frac{\alpha+\frac{\alpha}{2}\lp\sqrt{\pi} e^{y_i^2} y_i\erfc(y_i)\erfc(-y_i)-\erfc(-y_i)\rp}{1-\alpha-\lp\frac{\alpha}{2}\lp\sqrt{\pi} e^{y_i^2} y_i\erfc(y_i)\erfc(-y_i)-\erfc(-y_i)\rp+\beta\rp}\frac{\frac{\alpha}{2}\lp\sqrt{\pi} e^{y_i^2} y_i\erfc(y_i)\erfc(-y_i)-\erfc(-y_i)\rp+\beta}
  {-\lp\frac{\alpha}{2}\lp\sqrt{\pi} e^{y_i^2} y_i\erfc(y_i)\erfc(-y_i)-\erfc(-y_i)\rp\rp}.
\end{equation}
After a few additional transformations we have
\begin{equation}\label{eq:hdg1u}
  A_{bin}=\frac{2\beta+\alpha\lp\sqrt{\pi} e^{y_i^2} y_i\erfc(y_i)-1\rp\erfc(-y_i)}{2(1-\beta)-\alpha\lp\sqrt{\pi} e^{y_i^2} y_i\erfc(-y_i)+1\rp\erfc(y_i)}\frac{\lp\sqrt{\pi} e^{y_i^2} y_i\erfc(-y_i)+1\rp}
  {\lp-\sqrt{\pi} e^{y_i^2} y_i\erfc(y_i)+1\rp}\frac{\erfc(y_i)}{\erfc(-y_i)}.
\end{equation}
A combination of (\ref{eq:hdg1r}) and (\ref{eq:hdg1u}) gives
\begin{eqnarray}\label{eq:hdg1v}
  \frac{\erfc(-y_e)}{\erfc(y_e)}=\frac{\erfc(y_i)}{\erfc(-y_i)A_{bin} }
  =\frac{2(1-\beta)-\alpha\lp\sqrt{\pi} e^{y_i^2} y_i\erfc(-y_i)+1\rp\erfc(y_i)}{2\beta+\alpha\lp\sqrt{\pi} e^{y_i^2} y_i\erfc(y_i)-1\rp\erfc(-y_i)}\frac{\lp-\sqrt{\pi} e^{y_i^2} y_i\erfc(y_i)+1\rp}{\lp\sqrt{\pi} e^{y_i^2} y_i\erfc(-y_i)+1\rp}.
\end{eqnarray}
Now, from (\ref{eq:detanalIeer11h}) and (\ref{eq:detanalIeer11i}) we have
\begin{eqnarray}\label{eq:hdg1w}
\frac{\frac{1-\beta}{\alpha}
\lp
\frac{2y_2e^{y_2^2}}{\sqrt{\pi}y_2e^{y_2^2}\erfc(-y_2)+1}\rp-y_2e^{y_2^2}\erfc(y_2)}
{\frac{\beta}{\alpha}
\lp
\frac{2y_2e^{y_2^2}}{-\sqrt{\pi}y_2e^{y_2^2}\erfc(y_2)+1}\rp-y_2e^{y_2^2}\erfc(-y_2)}
=\frac{f_1(y_2;\alpha,\beta)}{-f_1(-y_2;\alpha,1-\beta)},
\end{eqnarray}
and
\begin{eqnarray}\label{eq:hdg1x}
\frac{
2(1-\beta)-\alpha\lp\sqrt{\pi}y_2e^{y_2^2}\erfc(-y_2)+1\rp\erfc(y_2)}
{
2\beta-\alpha\lp-\sqrt{\pi}y_2e^{y_2^2}\erfc(y_2)+1\rp\erfc(-y_2)}\frac{\lp-\sqrt{\pi}y_2e^{y_2^2}\erfc(y_2)+1\rp}{\lp\sqrt{\pi}y_2e^{y_2^2}\erfc(-y_2)+1\rp}
=\frac{f_1(y_2;\alpha,\beta)}{-f_1(-y_2;\alpha,1-\beta)}.
\end{eqnarray}
Connecting (\ref{eq:hdg1v}) and (\ref{eq:hdg1x}) we also have
\begin{equation}\label{eq:hdg1y}
  \frac{\erfc(-y_e)}{\erfc(y_e)}
  =\frac{2(1-\beta)-\alpha\lp\sqrt{\pi} e^{y_i^2} y_i\erfc(-y_i)+1\rp\erfc(y_i)}{2\beta+\alpha\lp\sqrt{\pi} e^{y_i^2} y_i\erfc(y_i)-1\rp\erfc(-y_i)}\frac{\lp-\sqrt{\pi} e^{y_i^2} y_i\erfc(y_i)+1\rp}{\lp\sqrt{\pi} e^{y_i^2} y_i\erfc(-y_i)+1\rp}
  =\frac{f_1(y_i;\alpha,\beta)}{-f_1(-y_i;\alpha,1-\beta)}.\\
\end{equation}
Connecting further (\ref{eq:detanalIeer11j}), (\ref{eq:detanalIeer11k}), and (\ref{eq:hdg1y}) gives
\begin{eqnarray}
& & \erf(y_e)  =
\frac{f_1(y_i;\alpha,\beta)+f_1(-y_i;\alpha,1-\beta)}{f_1(y_i;\alpha,\beta)-f_1(-y_i;\alpha,1-\beta)}\nonumber \\
&\Longleftrightarrow & y_e  =\erfinv\lp
\frac{f_1(y_i;\alpha,\beta)+f_1(-y_i;\alpha,1-\beta)}{f_1(y_i;\alpha,\beta)-f_1(-y_i;\alpha,1-\beta)}\rp.
\label{eq:hdg1z}
\end{eqnarray}
A combination of (\ref{eq:hdg1p}) and (\ref{eq:hdg1s})
\begin{multline}\label{eq:hdg1za}
\gamma_g=
\frac{\alpha}{2}\lp -\sqrt{\pi} e^{y_e^2} y_e\erfc(-y_e)-1\rp\erfc(y_e)+\frac{\erfc(y_e)}{2}
=\frac{\alpha}{2}\lp\sqrt{\pi} e^{y_i^2} y_i\erfc(y_i)\erfc(-y_i)-\erfc(-y_i)\rp+\beta\\
=\frac{-f_1(-y_i;\alpha,1-\beta)\alpha\lp -\sqrt{\pi} e^{y_i^2} y_i\erfc(y_i)+1\rp}{2e^{y_i^2} y_i}.
\end{multline}
Transforming further one also has
\begin{multline}\label{eq:hdg1zb}
\alpha\lp -\sqrt{\pi} e^{y_e^2} y_e\erfc(-y_e)-1\rp+1
=\frac{\lp f_1(y_i;\alpha,\beta)-f_1(-y_i;\alpha,1-\beta)\rp\alpha\lp -\sqrt{\pi} e^{y_i^2} y_i\erfc(y_i)+1\rp}{2e^{y_i^2} y_i},
\end{multline}
and
\begin{multline}\label{eq:hdg1zc}
e^{y_e^2} y_e\erfc(-y_e)
=\frac{\lp f_1(y_i;\alpha,\beta)-f_1(-y_i;\alpha,1-\beta)\rp\alpha\lp -\sqrt{\pi} e^{y_i^2} y_i\erfc(y_i)+1\rp-2e^{y_i^2} y_i+2e^{y_i^2} y_i\alpha}{-\sqrt{\pi}2e^{y_i^2} y_i\alpha}.
\end{multline}
Now we will argue that the right side in (\ref{eq:hdg1zc}) is equal to $f_1(y_i;\alpha,\beta)$. That will follow if
\begin{multline}\label{eq:hdg1zd}
\lp f_1(y_i;\alpha,\beta)-f_1(-y_i;\alpha,1-\beta)\rp\alpha\lp -\sqrt{\pi} e^{y_i^2} y_i\erfc(y_i)+1\rp-2e^{y_i^2} y_i+2e^{y_i^2} y_i\alpha=-\sqrt{\pi}2e^{y_i^2} y_i\alpha f_1(y_i;\alpha,\beta),
\end{multline}
or
\begin{multline}\label{eq:hdg1ze}
-f_1(-y_i;\alpha,1-\beta)\alpha\lp -\sqrt{\pi} e^{y_i^2} y_i\erfc(y_i)+1\rp-2e^{y_i^2} y_i+2e^{y_i^2} y_i\alpha=
 f_1(y_i;\alpha,\beta)\alpha\lp -\sqrt{\pi} e^{y_i^2} y_i\erfc(-y_i)-1\rp.
\end{multline}
Utilizing (\ref{eq:detanalIeer11j}) and (\ref{eq:detanalIeer11k}), (\ref{eq:hdg1ze}) can be rewritten in the following way
\begin{multline}\label{eq:hdg1zf}
\frac{\alpha}{2}\lp\sqrt{\pi} e^{y_i^2} y_i\erfc(y_i)\erfc(-y_i)-\erfc(-y_i)\rp+\beta-1+\alpha=
-\lp(1-\beta)
-\lp\sqrt{\pi}y_ie^{y_i^2}\erfc(-y_i)+1\rp\frac{\alpha}{2}\erfc(y_i)\rp.
\end{multline}
Since (\ref{eq:hdg1zf}) indeed holds one then has that (\ref{eq:hdg1ze}) also holds and through (\ref{eq:hdg1zd}) we have that the right side of (\ref{eq:hdg1zc}) is indeed equal to $f_1(y_i;\alpha,\beta)$. Combining further (\ref{eq:hdg1z}) and (\ref{eq:hdg1zc}) we finally arrive at
\begin{eqnarray}\label{eq:hdg1zg}
& & e^{y_e^2} y_e\erfc(-y_e)
=f_1(y_i;\alpha,\beta)\nonumber \\
& \Longleftrightarrow &  2\erfinv\lp
\frac{f_1(y_i;\alpha,\beta)+f_1(-y_i;\alpha,1-\beta)}{f_1(y_i;\alpha,\beta)-f_1(-y_i;\alpha,1-\beta)}\rp
\frac{e^{\erfinv\lp
\frac{f_1(y_i;\alpha,\beta)+f_1(-y_i;\alpha,1-\beta)}{f_1(y_i;\alpha,\beta)-f_1(-y_i;\alpha,1-\beta)}\rp
^2}}{f_1(y_i;\alpha,\beta)-f_1(-y_i;\alpha,1-\beta)}
=1.
\end{eqnarray}
(\ref{eq:hdg1zg}) is enough to determine $y_i$. Then from (\ref{eq:hdg1z}) one can determine $y_e$ and from (\ref{eq:hdg1p}) or (\ref{eq:hdg1s}) $\gamma_g$ (of course, comparing (\ref{eq:hdg1zg}) and (\ref{eq:hdg1z}) to (\ref{eq:detanalIeer11l}) and (\ref{eq:detanalIeer11k}), respectively one observes that $y_i=y_2$ and $y_e=y_1$). One can then use these values for $y_i$, $y_e$, and $\gamma_g$ to determine the optimal value of $\zeta_{\alpha,\beta}(\gamma_g,y_e,y_i)$ in (\ref{eq:hdg1a}) through (\ref{eq:hdg1b}). Such a value will give $I^{(bin)}_{err}(\alpha,\beta)$ in (\ref{eq:hdg1}). Here we will try to be a bit more explicit. Namely, we will connect the optimal $\zeta_{\alpha,\beta}(\gamma_g,y_e,y_i)$ to what we presented in Section \ref{sec:ldp}.

%%%%%%%%%%%%%%%%%%%%%%%%%%%%%%%%%%%%%%%%%%%%%%%%%%%%%%%%%%%%%%%%%
\subsubsection{Computing $I^{(bin)}_{err}(\alpha,\beta)$}
\label{sec:hdgcompI}
%%%%%%%%%%%%%%%%%%%%%%%%%%%%%%%%%%%%%%%%%%%%%%%%%%%%%%%%%%%%%%%%%

We start by noting that $\zeta_{\alpha,\beta}(\gamma_g,y_e,y_i)$ given in (\ref{eq:hdg1b}) consists of three components, namely, $\psicom$, $\psiint$, and $\psiext$ given in (\ref{eq:hdg2}). Instead of working directly with $\zeta_{\alpha,\beta}(\gamma_g,y_e,y_i)$ we will first deal with each of $\psicom$, $\psiint$, and $\psiext$. From this point on we assume that $y_i$, $y_e$, and $\gamma_g$ take the values determined through the procedure explained above. Then we have for $\psicom$
\begin{eqnarray}\label{eq:hdg1zh}
  \psicom &=& -(1-\beta)H\lp \frac{1-\alpha-\gamma_g}{1-\beta} \rp-\beta H\lp \frac{\beta-\gamma_g}{\beta} \rp \nonumber \\
&=& -(1-\alpha-\gamma_g)\log\lp \frac{1-\alpha-\gamma_g}{1-\beta}\rp-(\alpha+\gamma_g-\beta)\log\lp\frac{\alpha+\gamma_q-\beta}{1-\beta}\rp\nonumber \\
&& -(\beta-\gamma_g) \log\lp \frac{\beta-\gamma_g}{\beta} \rp-\gamma_g \log\lp \frac{\gamma_g}{\beta} \rp \nonumber \\
&=& -\frac{1}{2}\lp 2(1-\beta)-\alpha\lp\sqrt{\pi} e^{y_i^2} y_i\erfc(-y_i)+1\rp\erfc(y_i)\rp\nonumber \\
& & \times
\log\lp \frac{\lp2(1-\beta)-\alpha\lp\sqrt{\pi} e^{y_i^2} y_i\erfc(-y_i)+1\rp\erfc(y_i)\rp}{2(1-\beta)}\rp \nonumber \\
&& -\frac{\alpha}{2}\erfc(y_i)\lp\sqrt{\pi} e^{y_i^2} y_i\erfc(-y_i)+1\rp\log\lp\frac{\alpha\erfc(y_i)\lp\sqrt{\pi} e^{y_i^2} y_i\erfc(-y_i)+1\rp}{2(1-\beta)}\rp\nonumber \\
&& -\frac{\alpha}{2}\erfc(-y_i)\lp-\sqrt{\pi} e^{y_i^2} y_i\erfc(y_i)+1\rp \log\lp \frac{\alpha\erfc(-y_i)\lp-\sqrt{\pi} e^{y_i^2} y_i\erfc(y_i)+1\rp}{2\beta} \rp\nonumber \\
&& -\frac{1}{2}\lp 2\beta+\alpha\lp\sqrt{\pi} e^{y_i^2} y_i\erfc(y_i)-1\rp\erfc(-y_i)\rp\log\lp \frac{2\beta+\alpha\lp\sqrt{\pi} e^{y_i^2} y_i\erfc(y_i)-1\rp\erfc(-y_i)}{2\beta} \rp. \nonumber \\
\end{eqnarray}
Similarly we have for $\psiint$
\begin{eqnarray}\label{eq:hdg1zi}
  \psiint &=& \alpha y_i^2 +(\alpha-\beta+\gamma_g)\log(\erfc(y_i))+(\beta-\gamma_g)\log(\erfc(-y_i))- \alpha \log(2) \nonumber \\
&=& \alpha y_i^2 +\frac{\alpha}{2}\erfc(y_i)\lp\sqrt{\pi} e^{y_i^2} y_i\erfc(-y_i)+1\rp\log(\erfc(y_i))\nonumber \\
& & +\frac{\alpha}{2}\erfc(-y_i)\lp-\sqrt{\pi} e^{y_i^2} y_i\erfc(y_i)+1\rp\log(\erfc(-y_i))- \alpha \log(2), \nonumber \\
\end{eqnarray}
and finally for $\psiext$
\begin{eqnarray}\label{eq:hdg1zj}
  \psiext &=& -\alpha y_e^2 +(1-\alpha-\gamma_g)\log(\erfc(-y_e))+\gamma_g\log(\erfc(y_e))-(1-\alpha)\log(2) \nonumber \\
&=& -\alpha y_e^2 +\frac{1}{2}\lp 2(1-\beta)-\alpha\lp\sqrt{\pi} e^{y_i^2} y_i\erfc(-y_i)+1\rp\erfc(y_i)\rp
\log(\erfc(-y_e))\nonumber \\
& & +\frac{1}{2}\lp 2\beta+\alpha\lp\sqrt{\pi} e^{y_i^2} y_i\erfc(y_i)-1\rp\erfc(-y_i)\rp\log(\erfc(y_e))-(1-\alpha)\log(2). \nonumber \\
\end{eqnarray}
Combining (\ref{eq:hdg1zh}), (\ref{eq:hdg1zi}), and (\ref{eq:hdg1zj}) we also have
\begin{eqnarray}\label{eq:hdg1zk}
  \psinet^{(bin)}
&=& -\frac{1}{2}\lp 2(1-\beta)-\alpha\lp\sqrt{\pi} e^{y_i^2} y_i\erfc(-y_i)+1\rp\erfc(y_i)\rp \nonumber \\
& & \times
\log\lp \frac{\lp2(1-\beta)-\alpha\lp\sqrt{\pi} e^{y_i^2} y_i\erfc(-y_i)+1\rp\erfc(y_i)\rp}{2(1-\beta)\erfc(-y_e)}\rp \nonumber \\
&& -\frac{\alpha}{2}\erfc(y_i)\lp\sqrt{\pi} e^{y_i^2} y_i\erfc(-y_i)+1\rp\log\lp\frac{\alpha\lp\sqrt{\pi} e^{y_i^2} y_i\erfc(-y_i)+1\rp}{2(1-\beta)}\rp\nonumber \\
&& -\frac{\alpha}{2}\erfc(-y_i)\lp-\sqrt{\pi} e^{y_i^2} y_i\erfc(y_i)+1\rp \log\lp \frac{\alpha\lp-\sqrt{\pi} e^{y_i^2} y_i\erfc(y_i)+1\rp}{2\beta} \rp\nonumber \\
&& -\frac{1}{2}\lp 2\beta+\alpha\lp\sqrt{\pi} e^{y_i^2} y_i\erfc(y_i)-1\rp\erfc(-y_i)\rp\log\lp \frac{2\beta+\alpha\lp\sqrt{\pi} e^{y_i^2} y_i\erfc(y_i)-1\rp\erfc(-y_i)}{2\beta\erfc(y_e)} \rp \nonumber \\
&& +\alpha y_i^2-\alpha y_e^2-\log(2).
\end{eqnarray}
Utilizing again (\ref{eq:detanalIeer11j}) and (\ref{eq:detanalIeer11k}), (\ref{eq:hdg1zk}) can be further transformed
\begin{eqnarray}\label{eq:hdg1zl}
  \psinet^{(bin)}
&=& -\frac{1}{2}\lp 2(1-\beta)-\alpha\lp\sqrt{\pi} e^{y_i^2} y_i\erfc(-y_i)+1\rp\erfc(y_i)\rp \nonumber \\
& & \times
\log\lp \frac{ y_ee^{y_e^2}\erfc(-y_e)\alpha\lp\sqrt{\pi}y_ie^{y_i^2}\erfc(-y_i)+1\rp}{2(1-\beta)\erfc(-y_e)y_ie^{y_i^2}}\rp \nonumber \\
&& -\frac{\alpha}{2}\erfc(y_i)\lp\sqrt{\pi} e^{y_i^2} y_i\erfc(-y_i)+1\rp\log\lp\frac{\alpha\lp\sqrt{\pi} e^{y_i^2} y_i\erfc(-y_i)+1\rp}{2(1-\beta)}\rp\nonumber \\
&& -\frac{\alpha}{2}\erfc(-y_i)\lp-\sqrt{\pi} e^{y_i^2} y_i\erfc(y_i)+1\rp \log\lp \frac{\alpha\lp-\sqrt{\pi} e^{y_i^2} y_i\erfc(y_i)+1\rp}{2\beta} \rp\nonumber \\
&& -\frac{1}{2}\lp 2\beta+\alpha\lp\sqrt{\pi} e^{y_i^2} y_i\erfc(y_i)-1\rp\erfc(-y_i)\rp\log\lp \frac{y_ee^{y_e^2}\erfc(y_e)\alpha\lp -\sqrt{\pi}y_ie^{y_i^2}\erfc(y_i)+1\rp}{2\beta\erfc(y_e)y_ie^{y_i^2}} \rp \nonumber \\
&& +\alpha y_i^2-\alpha y_e^2-\log(2).
\end{eqnarray}
Continuing further we also obtain
\begin{eqnarray}\label{eq:hdg1zm}
  \psinet^{(bin)}
&=& -\frac{1}{2}\lp 2(1-\beta)-\alpha\lp\sqrt{\pi} e^{y_i^2} y_i\erfc(-y_i)+1\rp\erfc(y_i)\rp\log\lp \frac{ y_ee^{y_e^2}}{y_ie^{y_i^2}}\rp \nonumber \\
&& -\frac{1}{2}\lp 2\beta+\alpha\lp\sqrt{\pi} e^{y_i^2} y_i\erfc(y_i)-1\rp\erfc(-y_i)\rp\log\lp \frac{y_ee^{y_e^2}}{y_ie^{y_i^2}} \rp \nonumber \\
&& -\frac{1}{2}\lp 2(1-\beta)-\alpha\lp\sqrt{\pi} e^{y_i^2} y_i\erfc(-y_i)+1\rp\erfc(y_i)\rp\log\lp \frac{ \alpha\lp\sqrt{\pi}y_ie^{y_i^2}\erfc(-y_i)+1\rp}{2(1-\beta)}\rp \nonumber \\
&& -\frac{1}{2}\lp 2\beta+\alpha\lp\sqrt{\pi} e^{y_i^2} y_i\erfc(y_i)-1\rp\erfc(-y_i)\rp\log\lp \frac{\alpha\lp -\sqrt{\pi}y_ie^{y_i^2}\erfc(y_i)+1\rp}{2\beta} \rp \nonumber \\
&& -\frac{\alpha}{2}\erfc(y_i)\lp\sqrt{\pi} e^{y_i^2} y_i\erfc(-y_i)+1\rp\log\lp\frac{\alpha\lp\sqrt{\pi} e^{y_i^2} y_i\erfc(-y_i)+1\rp}{2(1-\beta)}\rp\nonumber \\
&& -\frac{\alpha}{2}\erfc(-y_i)\lp-\sqrt{\pi} e^{y_i^2} y_i\erfc(y_i)+1\rp \log\lp \frac{\alpha\lp-\sqrt{\pi} e^{y_i^2} y_i\erfc(y_i)+1\rp}{2\beta} \rp\nonumber \\
&& +\alpha y_i^2-\alpha y_e^2-\log(2),
\end{eqnarray}
and
\begin{eqnarray}\label{eq:hdg1zn}
  \psinet^{(bin)}
&=& -(1-\alpha)\log\lp \frac{ y_ee^{y_e^2}}{y_ie^{y_i^2}}\rp +(1-\beta)\log\lp \frac{2(1-\beta)}{ \alpha\lp\sqrt{\pi}y_ie^{y_i^2}\erfc(-y_i)+1\rp}\rp \nonumber \\
&& +\beta\log\lp \frac{2\beta}{\alpha\lp -\sqrt{\pi}y_ie^{y_i^2}\erfc(y_i)+1\rp}\rp  +\alpha y_i^2-\alpha y_e^2-\log(2).
\end{eqnarray}
A couple of additional cosmetic changes finally give
\begin{eqnarray}\label{eq:hdg1zo}
  \psinet^{(bin)}=I^{(bin)}_{err}(\alpha,\beta)=
&=& (\alpha-1)\log\lp \frac{ y_e}{y_i}\rp +(1-\beta)\log\lp \frac{1-\beta}{ \alpha\lp\sqrt{\pi}y_ie^{y_i^2}\erfc(-y_i)+1\rp}\rp \nonumber \\
&& +\beta\log\lp \frac{\beta}{\alpha\lp -\sqrt{\pi}y_ie^{y_i^2}\erfc(y_i)+1\rp}\rp  + y_i^2-y_e^2=I^{(bin)}_{ldp}(\alpha,\beta).
\end{eqnarray}
Comparing (\ref{eq:detanalIeer11mb}) and (\ref{eq:hdg1zo}) one recognizes that $I^{(bin)}_{err}=I^{(bin)}_{ldp}=I^{(bin,ub)}_{err,u}$ (of course $y_e\leftrightarrow y_1$ and $y_i\leftrightarrow y_2$) and observes that the choice for $\nu$, $A_0$, $c_3$, and $\gamma$ made in (\ref{eq:detanalIeer11m}) is indeed optimal. Moreover, in the lower tail regime ($\alpha<\alpha_w$, where $\alpha_w$ is such that $\psi^{(bin)}_{\beta}(\alpha_w)=\xi^{(bin)}_{\alpha_w}(\beta)=1$) considerations from \cite{Stojnicl1BnBxfinn} ensure that one also has
\begin{equation}\label{eq:hdg1aa}
  \Psi^{(bin)}_{net}(\alpha,\beta)=I^{(bin)}_{cor}(\alpha,\beta)\triangleq\lim_{n\rightarrow\infty}\frac{\log{P_{cor}}}{n}=\psicom+\psiint-\psiext,
\end{equation}
where $\psicom$, $\psiint$, and $\psiext$ are as in (\ref{eq:hdg2}). The above considerations then automatically fully characterize the binary $\ell_1$'s LDP. The characterization is summarized in the following theorem.
\begin{theorem}[Binary $\ell_1$'s LDP]
Assume the setup of Theorem \ref{thm:thmweakthr} and assume that a pair $(\alpha,\beta)$ is given. Let $P_{err}$ be the probability that the solutions of (\ref{eq:l0}) and (\ref{eq:l1bin}) coincide and let $P_{cor}$ be the probability that the solutions of (\ref{eq:l0}) and (\ref{eq:l1bin}) do \emph{not} coincide. Set
\begin{equation}\label{eq:thmfinalldpl11}
 f_1(y_2;\alpha,\beta) \triangleq  \frac{1-\beta}{\alpha}
\lp
\frac{2y_2e^{y_2^2}}{\sqrt{\pi}y_2e^{y_2^2}\erfc(-y_2)+1}\rp-y_2e^{y_2^2}\erfc(y_2).
\end{equation}
Also let $y_2$ and $y_1$ satisfy the following \textbf{fundamental characterizations of the binary $\ell_1$'s LDP} and achieve the optimum in (\ref{eq:hdg1a}):

%\begin{center}
%\shadowbox{$
\begin{eqnarray}
2\erfinv\lp
\frac{f_1(y_2;\alpha,\beta)+f_1(-y_2;\alpha,1-\beta)}{f_1(y_2;\alpha,\beta)-f_1(-y_2;\alpha,1-\beta)}\rp
\frac{e^{\erfinv\lp
\frac{f_1(y_2;\alpha,\beta)+f_1(-y_2;\alpha,1-\beta)}{f_1(y_2;\alpha,\beta)-f_1(-y_2;\alpha,1-\beta)}\rp^2}}
{f_1(y_2;\alpha,\beta)-f_1(-y_2;\alpha,1-\beta)}
& = & 1.\nonumber \\
\label{eq:thmfinalldpl12a}
\end{eqnarray}
and
\begin{eqnarray}
y_1  =  \erfinv\lp
\frac{f_1(y_2;\alpha,\beta)+f_1(-y_2;\alpha,1-\beta)}{f_1(y_2;\alpha,\beta)-f_1(-y_2;\alpha,1-\beta)}\rp.\nonumber \\
\label{eq:thmfinalldpl12b}
\end{eqnarray}
%$}
%-\vspace{-.5in}\begin{equation}
%\label{eq:thmldp3l1PT}
%\end{equation}
%\end{center}

\noindent Finally, let $I^{(bin)}_{ldp}(\alpha,\beta)$ be defined through the following \textbf{binary $\ell_1$'s fundamental LDP rate function} characterization
\begin{eqnarray}
I^{(bin)}_{ldp}(\alpha,\beta)&\triangleq & (\alpha-1)\log\lp\frac{y_1}{y_2}\rp+(1-\beta)\log \lp \frac{1-\beta}{\alpha\lp\sqrt{\pi}e^{y_2^2}y_2\erfc(-y_2)+1\rp}\rp\nonumber \\
& & +\beta\log\lp \frac{\beta}{\alpha\lp -\sqrt{\pi}e^{y_2^2}y_2\erfc(y_2)+1\rp}\rp + y_2^2-y_1^2.
\label{eq:thmfinalldpl13}
\end{eqnarray}
Then if $\alpha>\alpha_w$
\begin{equation}
I^{(bin)}_{err}(\alpha,\beta)\triangleq\lim_{n\rightarrow\infty}\frac{\log{P_{err}}}{n}=I^{(bin)}_{ldp}(\alpha,\beta).\label{eq:thmfinalldpl14}
\end{equation}
Moreover, if $\alpha<\alpha_w$
\begin{equation}
I^{(bin)}_{cor}(\alpha,\beta)\triangleq\lim_{n\rightarrow\infty}\frac{\log{P_{cor}}}{n}=I^{(bin)}_{ldp}(\alpha,\beta).\label{eq:thmfinalldpl15}
\end{equation}\label{thm:finalldpl1}
\end{theorem}
\begin{proof} Follows from the above discussion.
\end{proof}

%%%%%%%%%%%%%%%%%%%%%%%%%%%%%%%%%%%%%%%%%%%%%%%%%%%%%%%%%%%%%%%%%
\subsubsection{Reestablishing the phase transitions}
\label{sec:rephasetrans}
%%%%%%%%%%%%%%%%%%%%%%%%%%%%%%%%%%%%%%%%%%%%%%%%%%%%%%%%%%%%%%%%%

In this section we will show how one can reestablish the phase transition results from Section \ref{sec:phasetrans} utilizing Theorem \ref{thm:finalldpl1} and the above considerations leading up to Theorem \ref{thm:finalldpl1}. We start by focusing at pairs $(\alpha,\beta)$  for which $y_1=y_2$ in Theorem \ref{thm:finalldpl1} (of course, we may not know a priori if such pairs do exist; nonetheless, we will make such an assumption and since we will not contradict it through the derivation below, it will follow that the assumption is actually correct). From (\ref{eq:detanalIeer11h}) and (\ref{eq:detanalIeer11i}) we have
\begin{eqnarray}\label{eq:rephasetran1}
y_2e^{y_2^2}\erfc(-y_2) & = & y_1e^{y_1^2}\erfc(-y_1)=\frac{1-\beta}{\alpha}
\lp
\frac{2y_2e^{y_2^2}}{\sqrt{\pi}y_2e^{y_2^2}\erfc(-y_2)+1}\rp-y_2e^{y_2^2}\erfc(y_2)=f_1(y_2;\alpha,\beta)\nonumber \\
 y_2e^{y_2^2}\erfc(y_2) & = & y_1e^{y_1^2}\erfc(y_1) =\frac{\beta}{\alpha}
\lp
\frac{2y_2e^{y_2^2}}{-\sqrt{\pi}y_2e^{y_2^2}\erfc(y_2)+1}\rp-y_2e^{y_2^2}\erfc(-y_2)=-f_1(-y_2;\alpha,1-\beta).\nonumber \\
\end{eqnarray}
A couple of simple algebraic operations transform (\ref{eq:rephasetran1}) into the following
\begin{eqnarray}\label{eq:rephasetran2}
1=\frac{1-\beta}{\alpha}
\lp
\frac{1}{\sqrt{\pi}y_2e^{y_2^2}\erfc(-y_2)+1}\rp\nonumber \\
1 =\frac{\beta}{\alpha}
\lp
\frac{1}{-\sqrt{\pi}y_2e^{y_2^2}\erfc(y_2)+1}\rp.
\end{eqnarray}
Transforming a bit further we also have
\begin{eqnarray}\label{eq:rephasetran3}
\sqrt{\pi}y_2e^{y_2^2}\erfc(-y_2) & = & \frac{1-\beta}{\alpha}
-1\nonumber \\
\sqrt{\pi} y_2e^{y_2^2}\erfc(y_2) & = & 1-\frac{\beta}{\alpha},
\end{eqnarray}
and
\begin{eqnarray}\label{eq:rephasetran4}
\frac{1+\erf(y_2)}{1-\erf(y_2)}=\frac{\erfc(-y_2)}{\erfc(y_2)}=\frac{1-\alpha-\beta}{\alpha-\beta}.
\end{eqnarray}
From (\ref{eq:rephasetran4}) we then easily have
\begin{eqnarray}\label{eq:rephasetran5}
& & \erf(y_2)=\frac{\frac{1-\alpha-\beta}{\alpha-\beta}-1}{\frac{1-\alpha-\beta}{\alpha-\beta}+1}=\frac{1-2\alpha}{1-2\beta}\nonumber \\
\Longleftrightarrow & & y_2=\erfinv\lp\frac{1-2\alpha}{1-2\beta}\rp.
\end{eqnarray}
A combination of (\ref{eq:rephasetran3}) and (\ref{eq:rephasetran5}) then gives
\begin{eqnarray}\label{eq:rephasetran6}
\sqrt{\pi}y_2e^{y_2^2}\erfc(-y_2) & = & \sqrt{\pi}y_2e^{y_2^2}(1+\erf(y_2))=\sqrt{\pi}y_2e^{y_2^2}\frac{2-2\beta-2\alpha}{1-2\beta}=
\frac{1-\beta-\alpha}{\alpha}.
\end{eqnarray}
Finally combining (\ref{eq:rephasetran5}) and (\ref{eq:rephasetran6}) we obtain
\begin{eqnarray}\label{eq:rephasetran7}
\frac{(1-2\beta)e^{-y_2^2}}{2\sqrt{\pi}\alpha y_2}=\frac{(1-2\beta)e^{-\lp\erfinv\lp\frac{1-2\alpha}{1-2\beta}\rp\rp^2}}{2\sqrt{\pi}\alpha \erfinv\lp\frac{1-2\alpha}{1-2\beta}\rp}=1.
\end{eqnarray}
It is not that hard to see that (\ref{eq:rephasetran7}) is exactly the same as (\ref{eq:thmweaktheta2}). Moreover, from (\ref{eq:rephasetran1}) we also have
\begin{eqnarray}\label{eq:rephasetran8}
f_1(y_2;\alpha,\beta)-f_1(-y_2;\alpha,1-\beta)=
y_2e^{y_2^2}\erfc(-y_2)+ y_2e^{y_2^2}\erfc(y_2)=2y_2e^{y_2^2},
\end{eqnarray}
and
\begin{eqnarray}\label{eq:rephasetran9}
f_1(y_2;\alpha,\beta)+f_1(-y_2;\alpha,1-\beta)=
y_2e^{y_2^2}\erfc(-y_2)- y_2e^{y_2^2}\erfc(y_2)=2y_2e^{y_2^2}\erf(y_2).
\end{eqnarray}
Now, if $(\alpha,\beta)$ are such that (\ref{eq:rephasetran7}) holds then $y_1=y_2$ and (\ref{eq:rephasetran8}) and (\ref{eq:rephasetran9}) ensure that
(\ref{eq:thmfinalldpl12a}) and (\ref{eq:thmfinalldpl12b}) hold and that
$I^{(bin)}_{ldp}(\alpha,\beta)=0$ in (\ref{eq:thmfinalldpl13}) which is exactly the value that $I^{(bin)}_{ldp}(\alpha,\beta)$ takes at the phase transition.

%%%%%%%%%%%%%%%%%%%%%%%%%%%%%%%%%%%%%%%%%%%%%%%%%%%%%%%%%%%%%%%%%
\subsection{Theoretical and numerical LDP results}
\label{sec:thnumresuts}
%%%%%%%%%%%%%%%%%%%%%%%%%%%%%%%%%%%%%%%%%%%%%%%%%%%%%%%%%%%%%%%%%

In this section we present in a bit more concrete way what is actually proven in Theorem \ref{thm:finalldpl1}. The theoretical LDP rate function curve that can obtained based on Theorem \ref{thm:finalldpl1} is shown in Figure \ref{fig:l1regldpIerrub}. Two different values for $\beta$ were selected, $\beta=0.22933$ (which can be obtained from the PT curve for $\alpha=0.4$) and $\beta=\frac{1}{3}$. Also, for $\beta=0.22933$, we in addition to Figure \ref{fig:l1regldpIerrub} provide Table \ref{tab:Ildptab1} that contains the numerical values for all the quantities of interest in Theorems \ref{thm:ldp3} and \ref{thm:finalldpl1}. Finally, in Figure \ref{fig:weakl1LDPthrsim} and Table \ref{tab:Ildptab2} we show how the simulated values compare to the theoretical ones. As can be observed, even for fairly small dimensions (of order $100$) one already approaches the theoretical curves (the theoretical curves are of course derived for an infinite dimensional asymptotic regime).

%\begin{figure}[htb]
%\begin{minipage}[b]{.5\linewidth}
%\centering
%\centerline{\epsfig{figure=CSetamBlockWeak.eps,width=7.5cm,height=7cm}}
%%\end{minipage}
%%\begin{minipage}[b]{.5\linewidth}
%%\centering
%%\centerline{\epsfig{figure=finprerral08.eps,width=9cm,height=6.5cm}}
%\end{minipage}
%\begin{minipage}[b]{.5\linewidth}
%\centering
%\centerline{\epsfig{figure=SimulBlSpWeakd151.eps,width=7.5cm,height=7cm}}
%%\end{minipage}
%%\begin{minipage}[b]{.5\linewidth}
%%\centering
%%\centerline{\epsfig{figure=finprerral08.eps,width=9cm,height=6.5cm}}
%\end{minipage}
%\caption{\emph{Weak} threshold, $\ell_2/\ell_1$-optimization; theory -- left, simulations -- right}
%\label{fig:weak}
%\end{figure}

\begin{figure}[htb]
\begin{minipage}[b]{.5\linewidth}
\centering
\centerline{\epsfig{figure=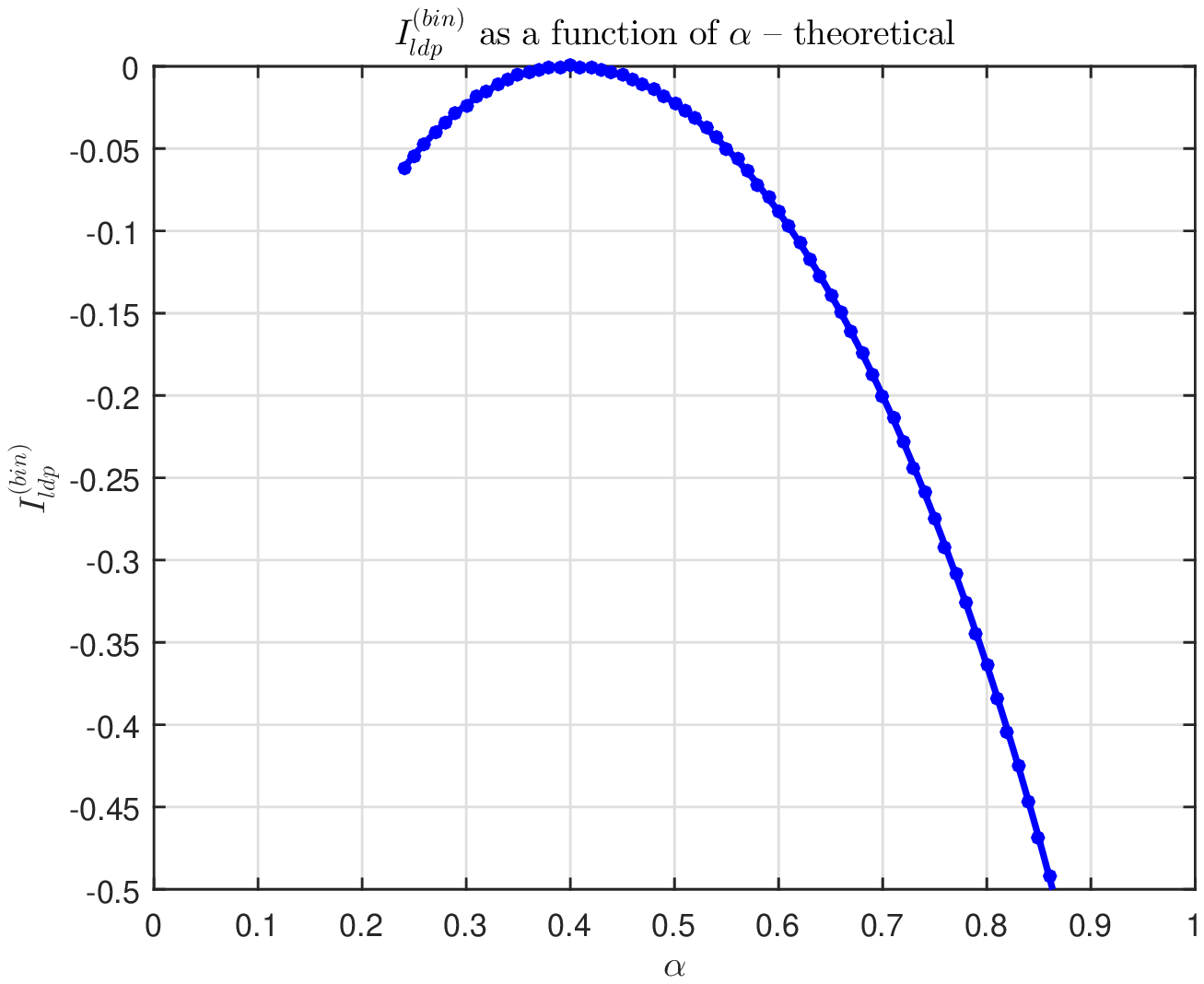,width=9cm,height=7cm}}
%\end{minipage}
%\begin{minipage}[b]{.5\linewidth}
%\centering
%\centerline{\epsfig{figure=finprerral08.eps,width=9cm,height=6.5cm}}
\end{minipage}
\begin{minipage}[b]{.5\linewidth}
\centering
\centerline{\epsfig{figure=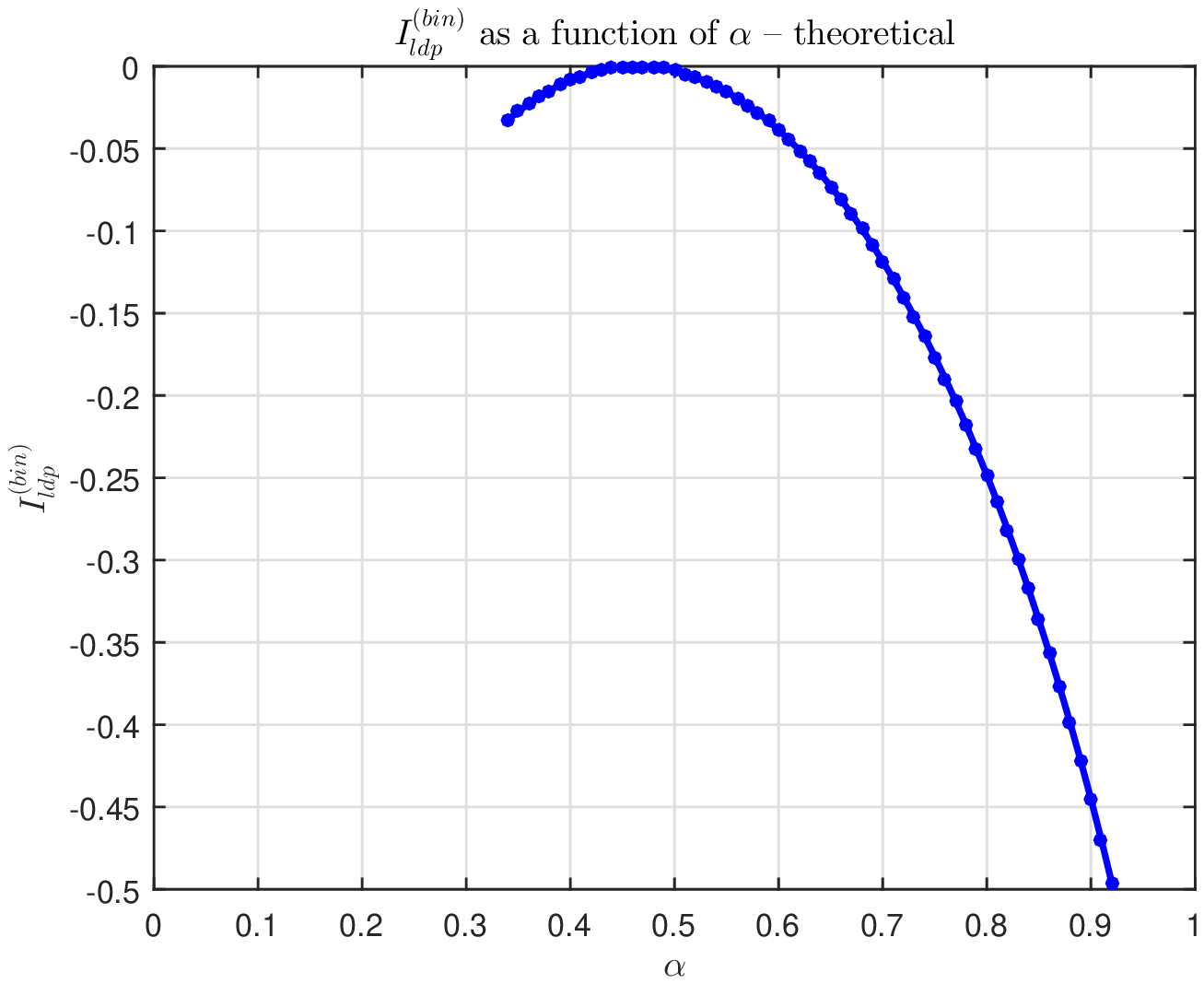,width=9cm,height=7cm}}
%\end{minipage}
%\begin{minipage}[b]{.5\linewidth}
%\centering
%\centerline{\epsfig{figure=finprerral08.eps,width=9cm,height=6.5cm}}
\end{minipage}
\caption{$I^{(bin)}_{ldp}$ as a function of $\alpha$; left -- $\beta=0.22933$; right -- $\beta=\frac{1}{3}$}
\label{fig:l1regldpIerrub}
\end{figure}

\begin{table}[h]
\caption{A collection of values for $y_1$, $\gamma_g$, $A_{bin}$, $y_2$, $\nu$, $A_0$, $c_3$, $\gamma$, and $I^{(bin)}_{ldp}$ in Theorem \ref{thm:ldp3}; $\beta=0.22933$}\vspace{.1in}
\hspace{-0in}\centering
\begin{tabular}{||c||c|c|c|c|c||}\hline\hline
$\alpha$     & $ 0.30 $ & $ 0.35 $ & $ 0.40 $ & $ 0.45 $ & $ 0.50 $ \\ \hline\hline
$y_1$        & $ 0.3923 $ & $ 0.3666 $ & $ 0.3401 $ & $ 0.3127 $ & $ 0.2846 $ \\ \hline
$\gamma_g$& $ 0.1026 $ & $ 0.0866 $ & $ 0.0723 $ & $ 0.0596 $ & $ 0.0484 $ \\ \hline
$A_{bin}$    & $ 0.2347 $ & $ 0.2234 $ & $ 0.2120 $ & $ 0.2006 $ & $ 0.1892 $ \\ \hline
$y_2$        & $ 0.2431 $ & $ 0.2907 $ & $ 0.3401 $ & $ 0.3910 $ & $ 0.4432 $ \\ \hline\hline
%$\beta_1$    & $ -0.0510 $ & $ 0.0552 $ & $ 0.1786 $ & $ 0.3243 $ & $ 0.5000 $ \\ \hline
%$\beta_0$    & $ -0.4129 $ & $ -0.0694 $ & $ 0.1787 $ & $ 0.3618 $ & $ 0.5000 $ \\ \hline\hline
$\nu$        & $ 0.5548 $ & $ 0.5185 $ & $ 0.4810 $ & $ 0.4422 $ & $ 0.4025 $ \\ \hline
$A_0$        & $ 1.6137 $ & $ 1.2611 $ & $ 1.0000 $ & $ 0.7997 $ & $ 0.6421 $ \\ \hline
$c_3$        & $ -0.5445 $ & $ -0.2770 $ & $ 0.0000 $ & $ 0.3023 $ & $ 0.6471 $ \\ \hline
$\gamma$     & $ 0.1697 $ & $ 0.2346 $ & $ 0.3162 $ & $ 0.4194 $ & $ 0.5506 $ \\ \hline\hline
$I^{(bin)}_{ldp}$    & $ \mathbf{-0.0234} $ & $ \mathbf{-0.0058} $ & $ \mathbf{0.0000} $ & $ \mathbf{-0.0056} $ & $ \mathbf{-0.0223} $ \\ \hline\hline
\end{tabular}
\label{tab:Ildptab1}
\end{table}

\begin{figure}[htb]
%\begin{minipage}[b]{.5\linewidth}
\centering
\centerline{\epsfig{figure=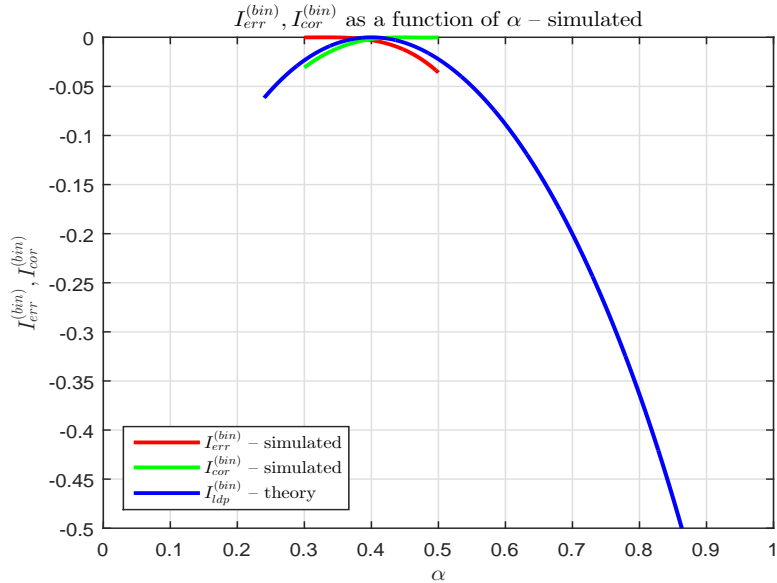,width=11.5cm,height=8cm}}
%\end{minipage}
%\begin{minipage}[b]{.5\linewidth}
%\centering
%\centerline{\epsfig{figure=finprerral08.eps,width=9cm,height=6.5cm}}
%\end{minipage}
\caption{Binary $\ell_1$'s weak LDP rate function -- theory and simulation; $\beta=0.22933$}
\label{fig:weakl1LDPthrsim}
\end{figure}

\begin{table}[h]
\caption{$I^{(bin)}_{err}$, $I^{(bin)}_{cor}$ -- simulated; $I^{(bin)}_{ldp}$ calculated for $\beta=0.22933$}\vspace{.1in}
\hspace{-0in}\centering
\begin{tabular}{||c||c|c|c|c|c||}\hline\hline
$\alpha$ & $ 0.30 $ & $ 0.35 $ & $ 0.40 $ & $ 0.45 $ & $ 0.50 $\\ \hline\hline
$k$      & $ 32 $ & $ 69 $ & $ 69 $ & $ 69 $ & $ 32 $ \\ \hline
$m$      & $ 42 $ & $ 105 $ & $ 120 $ & $ 135 $ & $ 70 $ \\ \hline
$n$      & $ 140 $ & $ 300 $ & $ 300 $ & $ 300 $ & $ 140 $ \\ \hline\hline
$I^{(bin)}_{err}$ -- simulated & $ -0.0001 $ & $ -0.0001 $ & \red{$ \mathbf{ -0.0029} $} & \red{$ \mathbf{ -0.0130} $} & \red{$ \mathbf{ -0.0359} $} \\ \hline
$I^{(bin)}_{cor}$ -- simulated & \gr{$ \mathbf{-0.0309} $} & \gr{$ \mathbf{ -0.0110} $} & \gr{$ \mathbf{ -0.0018} $} & $ -0.0001 $ & $ -0.0000 $ \\ \hline\hline
$I^{(bin)}_{ldp}$ -- theory   & \bl{$ \mathbf{-0.0234} $} & \bl{$ \mathbf{-0.0058} $} & \bl{$ \mathbf{0.0000} $} & \bl{$ \mathbf{-0.0056} $} & \bl{$ \mathbf{-0.0223} $} \\ \hline\hline
\end{tabular}
\label{tab:Ildptab2}
\end{table}

%%%%%%%%%%%%%%%%%%%%%%%%%%%%%%%%%%%%%%%%%%%%%%%%%%%%%%%%%%%%%%%%%
\section{Box $\ell_1$}
\label{sec:boxl1}
%%%%%%%%%%%%%%%%%%%%%%%%%%%%%%%%%%%%%%%%%%%%%%%%%%%%%%%%%%%%%%%%%

In Section \ref{sec:binl1} we looked at the properties of the $\ell_1$ from (\ref{eq:l1bin}) when employed for solving (\ref{eq:l0}) a priori known to have the solution that is a binary vector. Here we will instead of binary vectors consider the so-called box-constrained vectors (for more on these see, e.g. \cite{DTbern}). These vectors are typically viewed as a more general class of binary vectors. Namely, any component of a box-constrained vector is assumed to be within a real interval (as earlier, without the loss of generality, we will assume that this interval is $[0,1]$). As discussed in the introduction, this broadly defines pretty much any vector. Things are a bit more interesting if one looks at the so-called sparse box-constrained vectors. Such vectors are defined in the following way. Namely, vector $\x$ is called box-constrained $k$ sparse (from this point on we assume being constrained on interval $[0,1]$) if no more than $k$ of its components are not at the edges of the box/interval. In other words, $\x$ is called box-constrained $k$ sparse if no more than $k$ of its components are not equal to zero or one. To solve (\ref{eq:l0}) known to have box-constrained $k$ sparse solution we will employ (\ref{eq:l1bin}) which we conveniently rewrite below
\begin{eqnarray}
\mbox{min} & & \|\x\|_1\nonumber \\
\mbox{subject to} & & A\x=\y \nonumber \\
& & 0\leq \x_i\leq 1, 1\leq i\leq n. \label{eq:boxl1bin}
\end{eqnarray}
As mentioned earlier, we will refer to (\ref{eq:boxl1bin}) as the \emph{box} $\ell_1$ when it is used for finding the box-constrained $k$ sparse solution of (\ref{eq:l0}). Below we will provide a performance analysis of the \emph{box} $\ell_1$. As was the case in Section \ref{sec:binl1} when we discussed the \emph{binary} $\ell_1$, the analysis will focus on the phase transitions (PTs) and the corresponding LDPs. Differently from what we did in Section \ref{sec:binl1} though, here we start things off by focusing on the LDPs first (the PT results will easily follow afterwards).

%%%%%%%%%%%%%%%%%%%%%%%%%%%%%%%%%%%%%%%%%%%%%%%%%%%%%%%%%%%%%%%%%
\subsection{Large deviations}
\label{sec:boxldp}
%%%%%%%%%%%%%%%%%%%%%%%%%%%%%%%%%%%%%%%%%%%%%%%%%%%%%%%%%%%%%%%%%

We will try to follow as closely as possible the derivations for the large deviations of the binary $\ell_1$ that we presented in Section \ref{sec:binl1}. As usual, we will try to skip as many repetitive steps as possible and instead will focus on those that bring key differences. We begin by introducing a theorem that is basically the box $\ell_1$ analogue to the binary $\ell_1$'s Theorem \ref{thm:thmknownsuppcond}. To facilitate the exposition, in addition to assuming that all elements of $\x$ are from $[0,1]$ interval, we will, for any $\mu\in[\frac{1}{2},1]$, without the loss of generality assume that the elements $\x_{1},\x_{2},\dots,\x_{\mu(n-k)}$ of $\x$ are equal to zero and that the elements $\x_{\mu(n-k)+1},\x_{\mu(n-k)+2},\dots,\x_{n-k}$ of $\x$ are equal to one. Minimal modifications of the arguments leading up to Theorem \ref{thm:thmknownsuppcond} produce the following theorem.
\begin{theorem}(\cite{StojnicCSetam09,StojnicICASSP09,StojnicISIT2010binary} Nonzero elements of box-constrained $\x$ have fixed location)
Assume that an $m\times n$ system matrix $A$ is given. Let $\x$
be a box-constrained $k$-sparse vector and let $\mu$ be a real number such that $\mu\in[\frac{1}{2},1]$. Also let each element of $\x$ belong to $[0,1]$ interval and let $\x_1=\x_2=\dots=\x_{\mu(n-k)}=0$ and $\x_{\mu(n-k)+1},\x_{\mu(n-k)+2},\dots,\x_{n-k}=1$. Further, assume that $\y\triangleq A\x$ and that $\w$ is
an $n\times 1$ vector such that $\w_i\geq 0,1\leq i\leq \mu(n-k)$, and $\w_i\leq 0,\mu(n-k)+1\leq i\leq n-k$. If
\begin{equation}
(\forall \w\in \textbf{R}^n | A\w=0) \quad  -\sum_{i=n-k+1}^{n} \w_i<\sum_{i=1}^{n-k}\w_{i},
\end{equation}
then the solutions of (\ref{eq:l0}) and (\ref{eq:boxl1bin}) coincide. Moreover, if
\begin{equation}
(\exists \w\in \textbf{R}^n | A\w=0) \quad  -\sum_{i=n-k+1}^{n} \w_i\geq \sum_{i=1}^{n-k}\w_{i},
\label{eq:boxthmeqgen}
\end{equation}
then the solution of (\ref{eq:l0}) is not the solution of (\ref{eq:boxl1bin}).
\label{thm:boxthmknownsuppcond}
\end{theorem}
\begin{proof}
  Follows by a couple of simple modifications of the arguments leading up to Theorem \ref{thm:thmknownsuppcond}.
\end{proof}
To facilitate the exposition we set
\begin{equation}
\Sw^{(box)}\triangleq\{\w\in S^{n-1}| \quad -\sum_{i=n-k+1}^{n} \w_i<\sum_{i=1}^{n-k}\w_{i},\quad \w_i\geq 0,1\leq i\leq \mu(n-k), \quad \mbox{and} \quad\w_i\leq 0,\mu(n-k)+1\leq i\leq n-k\},\label{eq:boxdefSwpr}
\end{equation}
and as in Section \ref{sec:binl1} (and ultimately \cite{Stojnicl1RegPosasymldp}), we first provide a detailed analysis of the LDPs upper tail (also as in Section \ref{sec:binl1}, a few minimal adaptations of the upper tail analysis automatically settle the lower tail as well).

%%%%%%%%%%%%%%%%%%%%%%%%%%%%%%%%%%%%%%%%%%%%%%%%%%%%%%%%%%%%%%%%%
\subsubsection{Upper tail}
\label{sec:boxuppertail}
%%%%%%%%%%%%%%%%%%%%%%%%%%%%%%%%%%%%%%%%%%%%%%%%%%%%%%%%%%%%%%%%%

We recall that in the upper tail we consider the points $(\alpha,\beta)$ such that $\alpha\geq \alpha_w$ where $\alpha_w$ is the phase transition value for given $\beta$. Assuming that the elements of $A$ are i.i.d. standard normals and following Section \ref{sec:binl1} and \cite{Stojnicl1RegPosasymldp}, we have
\begin{equation}
P^{(box)}_{err}\triangleq P(\min_{\w\in S^{(box)}_w}\|A\w\|_2\leq 0)=P(\max_{\w\in S^{(box)}_w}\min_{\|\y\|_2=1}(\y^T A\w )\geq 0)\leq
\min_{c_3\geq 0} e^{-\frac{c_3^2}{2}}Ee^{-c_3\|\g\|_2}Ee^{c_3w(\h,S^{(box)}_w)},
\label{eq:boxldpprob3}
\end{equation}
where $P^{(box)}_{err}$ is the so-called probability of error/failure, i.e. the probability that (\ref{eq:boxl1bin}) fails to produce the solution of (\ref{eq:l0}) and
\begin{eqnarray}
w(\h,\Sw^{(box)})\triangleq\max_{\w\in \Sw^{(box)}} (\h^T\w) = \max_{\bar{\y}\in \mR^{n}} & &  \sum_{i=1}^{n} \h_i \bar{\y}_i\nonumber \\
\mbox{subject to} &  & \bar{\y}_i\geq 0, 0\leq i\leq n-k\nonumber \\
& & \sum_{i=n-k+1}^{n}\bar{\y}_i\geq \sum_{i=1}^{\mu(n-k)} \bar{\y}_i -\sum_{i=\mu(n-k)+1}^{n-k} \bar{\y}_i\nonumber \\
& & \sum_{i=1}^{n}\bar{\y}_i^2\leq 1.\label{eq:boxworkww2}
\end{eqnarray}
We of course recall that the elements of $\h$ are again the i.i.d. standard normals. As in \cite{StojnicCSetam09,StojnicCSetamBlock09,StojnicBlockasymldpfinn15,StojnicLiftStrSec13,StojnicICASSP10knownsupp,StojnicTowBettCompSens13,Stojnicl1RegPosasymldp,StojnicISIT2010binary} one writes
\begin{eqnarray}
w(\h,\Sw^{(box)}) = -\max_{\nu\geq 0,\gamma\geq 0}\min_{\bar{\y}} & & \sum_{i=1}^{n} -\h_i \bar{\y}_i+\nu\sum_{i=1}^{\mu(n-k)}\bar{\y}_i
-\nu\sum_{i=\mu(n-k)+1}^{n-k}\bar{\y}_i
-\nu\sum_{i=n-k+1}^{n}\bar{\y}_i+\gamma\sum_{i=1}^{n}\bar{\y}_i^2-\gamma\nonumber \\
\mbox{subject to} & & \bar{\y}_i\geq 0, 0\leq i\leq n-k,\label{eq:boxldpwhSw0}
\end{eqnarray}
and finally
\begin{eqnarray}
w(\h,\Sw^{(box)}) & = & \min_{\nu\geq0,\gamma\geq 0} \frac{\sum_{i=1}^{\mu(n-k)}\max(\h_i-\nu,0)^2+\sum_{i=\mu(n-k)+1}^{n-k}\max(\h_i+\nu,0)^2+\sum_{i=n-k+1}^{n}(\h_i+\nu)^2}{4\gamma}+\gamma\nonumber \\
& = & \min_{\nu\geq0}\sqrt{\sum_{i=1}^{\mu(n-k)}\max(\h_i-\nu,0)^2+\sum_{i=\mu(n-k)+1}^{n-k}\max(\h_i+\nu,0)^2+\sum_{i=n-k+1}^{n}(\h_i+\nu)^2}.\label{eq:boxldpwhSw}
\end{eqnarray}
The following theorem summarizes the above methodology to upper bound $P^{(box)}_{err}$.
\begin{theorem}
Let $A$ be an $m\times n$ matrix in (\ref{eq:l0})
with i.i.d. standard normal components. Let $\mu$ be a real number such that $\mu\in [\frac{1}{2},1]$. Further, let
the unknown $\x$ in (\ref{eq:l0}) be box-constrained $k$-sparse and let the locations of the elements of $\x$ from $(0,1)$ be arbitrarily chosen but fixed. Let $P^{(box)}_{err}$ be the probability that the solution of (\ref{eq:boxl1bin}) is not the solution of (\ref{eq:l0}). Then
\begin{eqnarray}
P^{(box)}_{err} & \leq & \min_{c_3\geq 0}e^{-\frac{c_3^2}{2}}e^{-c_3\|\g\|_2}Ee^{c_3w(\h,S^{(box)}_w)} \nonumber \\
& = & \min_{c_3\geq 0}\left (e^{-\frac{c_3^2}{2}}\frac{1}{\sqrt{2\pi}^m}\int_{\g}e^{-\sum_{i=1}^{m}\g_i^2/2-c_3\|\g\|_2}d\g \min_{\nu\geq 0,\gamma\geq\frac{c_3}{2}} w_1^{\mu(n-k)}w_2^{(1-\mu)(n-k)}w_3^{k}e^{c_3\gamma}\right ),
\label{eq:boxldpthm1perrub1}
\end{eqnarray}
where
\begin{eqnarray}
% \nonumber % Remove numbering (before each equation)
w_1 &=& \frac{1}{\sqrt{2\pi}}\int_{h}e^{-h^2/2}e^{c_3\max(h-\nu,0)^2/4/\gamma}dh
  =\frac{1}{2}\lp\frac{e^{\frac{c_3\nu^2/4/\gamma}{1-c_3/2/\gamma}}}{\sqrt{1-c_3/2/\gamma}}\erfc\left (\frac{\nu}{\sqrt{2}\sqrt{1-c_3/2/\gamma}}\right )+\erf\left (\frac{\nu}{\sqrt{2}}\right )+1\rp\nonumber \\
w_2 &=& \frac{1}{\sqrt{2\pi}}\int_{h}e^{-h^2/2}e^{c_3\max(h+\nu,0)^2/4/\gamma}dh
  =\frac{1}{2}\lp\frac{e^{\frac{c_3\nu^2/4/\gamma}{1-c_3/2/\gamma}}}{\sqrt{1-c_3/2/\gamma}}\erfc\left (\frac{-\nu}{\sqrt{2}\sqrt{1-c_3/2/\gamma}}\right )+\erf\left (\frac{-\nu}{\sqrt{2}}\right )+1\rp\nonumber \\
w_3 &=&\frac{1}{\sqrt{2\pi}}\int_{h}e^{-h^2/2}e^{c_3(h+\nu)^2/4/\gamma}dh= \frac{e^{\frac{c_3\nu^2/4/\gamma}{1-c_3/2/\gamma}}}{\sqrt{1-c_3/2/\gamma}}.
\label{eq:boxldpthm1perrub2}
\end{eqnarray}\label{thm:boxldp1}
\end{theorem}
\begin{proof}
Follows from the above considerations, what was presented in Section \ref{sec:binl1}, and ultimately through the mechanisms developed in \cite{StojnicCSetam09,StojnicCSetamBlock09,StojnicBlockasymldpfinn15,StojnicLiftStrSec13,StojnicICASSP10knownsupp,StojnicTowBettCompSens13,Stojnicl1RegPosasymldp,StojnicISIT2010binary}.
\end{proof}
The above upper-bounding strategy works for any allowable integers $m$, $k$, and $n$. Introducing for the decay rate of $P^{(box)}_{err}$, $I^{(box)}_{err}(\alpha,\beta)$,
\begin{equation}\label{eq:boxldpasymp1}
  I^{(box)}_{err}(\alpha,\beta)\triangleq\lim_{n\rightarrow\infty}\frac{\log{P^{(box)}_{err}}}{n}.
\end{equation}
we have, based on Theorem \ref{thm:boxldp1}, the following LDP type of theorem.
\begin{theorem}
Assume the setup of Theorem \ref{thm:boxldp1}. Further, let integers $m$, $k$, and $n$ be large ($k\leq m\leq n$) such that $\beta=\frac{k}{n}$ and $\alpha=\frac{m}{n}$ are constants independent of $n$. Assume that a pair $(\alpha,\beta)$  is given. Also, assume the following scaling: $c_3\rightarrow c_3\sqrt{n}$ and $\gamma\rightarrow\gamma\sqrt{n}$. Then
\begin{eqnarray}
I^{(box)}_{err}(\alpha,\beta) & \triangleq &\lim_{n\rightarrow\infty}\frac{\log{P^{(box)}_{err}}}{n}\nonumber \\
& \leq & \min_{c_3\geq 0}\left (-\frac{(c_3)^2}{2}+I_{sph}+\min_{\nu\geq 0,\gamma\geq \frac{c_3}{2}} (\mu(1-\beta)\log{w_1}
+(1-\mu)(1-\beta)\log{w_2}+\beta \log{w_3}+c_3\gamma)\right ) \nonumber \\
& \triangleq & I_{err,u}^{(box,ub)}(\alpha,\beta),\nonumber \\
\label{eq:boxldpthm2Ierrub1}
\end{eqnarray}
where
\begin{eqnarray}
% \nonumber % Remove numbering (before each equation)
I_{sph} &=& \widehat{\gamma}c_3-\frac{\alpha }{2}\log\left (1-\frac{c_3}{2\widehat{\gamma}}\right )\nonumber \\
  \widehat{\gamma} &=& \frac{c_3-\sqrt{(c_3)^2+4\alpha}}{4}\nonumber \\
w_1 &=& \frac{1}{2}\lp\frac{e^{\frac{c_3\nu^2/4/\gamma}{1-c_3/2/\gamma}}}{\sqrt{1-c_3/2/\gamma}}\erfc\left (\frac{\nu}{\sqrt{2}\sqrt{1-c_3/2/\gamma}}\right )+\erf\left (\frac{\nu}{\sqrt{2}}\right )+1\rp\nonumber \\
w_2 &=& \frac{1}{2}\lp\frac{e^{\frac{c_3\nu^2/4/\gamma}{1-c_3/2/\gamma}}}{\sqrt{1-c_3/2/\gamma}}\erfc\left (\frac{-\nu}{\sqrt{2}\sqrt{1-c_3/2/\gamma}}\right )+\erf\left (\frac{-\nu}{\sqrt{2}}\right )+1\rp.\nonumber \\
w_3 &=& \frac{e^{\frac{c_3\nu^2/4/\gamma}{1-c_3/2/\gamma}}}{\sqrt{1-c_3/2/\gamma}}.
\label{eq:boxldpthm2perrub2}
\end{eqnarray}\label{thm:boxldp2}
\end{theorem}
\begin{proof} As in the case of Theorem \ref{thm:ldp2} in Section \ref{sec:binl1}, follows in a fashion analogous to the one employed in \cite{Stojnicl1RegPosasymldp}.
\end{proof}
Below we will provide an explicit solution to the above optimization problem. We will follow the methodology of Section \ref{sec:binl1} as closely as possible. However, there will be quite a few differences and we will try to emphasize them.

%%%%%%%%%%%%%%%%%%%%%%%%%%%%%%%%%%%%%%%%%%%%%%%%%%%%%%%%%%%%%%%%%
\subsubsection{Determining $I_{err,u}^{(box,ub)}$}
\label{sec:boxanalysisIerr}
%%%%%%%%%%%%%%%%%%%%%%%%%%%%%%%%%%%%%%%%%%%%%%%%%%%%%%%%%%%%%%%%%

As in Section \ref{sec:binl1}  we start by setting
\begin{equation}\label{eq:boxdetanalIerr1}
  A_{0}\triangleq\sqrt{1-\frac{c_3}{2\gamma}},
\end{equation}
and then note that optimization problem from the above theorem can be rewritten as
\begin{equation}\label{eq:boxdetanalIeer2}
I_{err,u}^{(box,ub)}(\alpha,\beta)\triangleq \min_{c_3\geq 0,\nu\geq 0,A_0\leq 1}\zeta^{(box)}_{\alpha,\beta}(c_3,\nu,A_0),
\end{equation}
where
\begin{eqnarray}
% \nonumber % Remove numbering (before each equation)
\zeta^{(box)}_{\alpha,\beta}(c_3,\nu,A_0)&=&\left (-\frac{c_3^2}{2}+I_{sph}+\mu(1-\beta)\log{w_1}+(1-\mu)(1-\beta)\log{w_2}+\beta\log{w_3}+\frac{c_3^2}{2(1-A_0^2)}\right )\nonumber \\
I_{sph} &=& \widehat{\gamma}c_3-\frac{\alpha }{2}\log\left (1-\frac{c_3}{2\widehat{\gamma}}\right )\nonumber \\
  \widehat{\gamma} &=& \frac{c_3-\sqrt{(c_3)^2+4\alpha}}{4}\nonumber \\
w_1 &=& \frac{1}{2}\lp
  \frac{e^{\frac{(1-A_0^2)\nu^2}{2A_0^2}}}{A_0}\erfc\left (\frac{\nu}{\sqrt{2}A_0}\right )+\erf\left (\frac{\nu}{\sqrt{2}}\right )+1\rp\nonumber \\
w_2 &=& \frac{1}{2}\lp
  \frac{e^{\frac{(1-A_0^2)\nu^2}{2A_0^2}}}{A_0}\erfc\left (\frac{-\nu}{\sqrt{2}A_0}\right )+\erf\left (\frac{-\nu}{\sqrt{2}}\right )+1\rp \nonumber \\
 w_3 &=& \frac{e^{\frac{c_3\nu^2/4/\gamma}{1-c_3/2/\gamma}}}{\sqrt{1-c_3/2/\gamma}}.
\label{eq:boxdetanalIeer3}
\end{eqnarray}
As in Section \ref{sec:binl1}, to find an optimum in (\ref{eq:boxdetanalIeer2}) we will compute the derivatives of $\zeta^{(box)}_{\alpha,\beta}(c_3,\nu,A_0)$ with respect to $c_3$, $\nu$, and $A_0$ and solve the following system of three equations
\begin{equation}\label{eq:boxdetanalIeer3a}
  \frac{d \zeta^{(box)}_{\alpha,\beta}(c_3,\nu,A_0)}{dc_3}=  \frac{d \zeta^{(box)}_{\alpha,\beta}(c_3,\nu,A_0)}{d\nu}=  \frac{d \zeta^{(box)}_{\alpha,\beta}(c_3,\nu,A_0)}{dA_0}=0.
\end{equation}
We start by recognizing that the considerations related to the derivative with respect to $c_3$ are the same as in Section \ref{sec:binl1}. Namely,
\begin{equation}
% \nonumber % Remove numbering (before each equation)
\frac{d\zeta^{(box)}_{\alpha,\beta}(c_3,\nu,A_0)}{dc_3}=-c_3+\frac{c_3}{1-A_0^2}+\frac{c_3-\sqrt{(c_3)^2+4\alpha}}{2},
\label{eq:boxdetanalIeer9}
\end{equation}
and after setting further the above derivative to zero one has
\begin{equation}\label{eq:boxdetanalIeer12}
  c_3=\frac{(1-A_0^2)\sqrt{\alpha}}{A_0}.
\end{equation}
The derivatives with respect to $\nu$ and $A_0$ are different and a bit more involved. For the derivative with respect to $\nu$ we have
\begin{eqnarray}
% \nonumber % Remove numbering (before each equation)
\frac{d\zeta^{(box)}_{\alpha,\beta}(c_3,\nu,A_0)}{d\nu}&=&\frac{d}{d\nu}\left (-\frac{c_3^2}{2}+I_{sph}+\mu(1-\beta)\log{w_1}+(1-\mu)(1-\beta)\log{w_2}+\beta\log{w_3}+\frac{c_3^2}{2(1-A_0^2)}\right )\nonumber \\
&=& \frac{\mu(1-\beta)}{w_1}\left (\frac{(1-A_0^2)\nu}{A_0^2}\frac{e^{\frac{(1-A_0^2)\nu^2}{2A_0^2}}}{A_0}\erfc\left (\frac{\nu}{\sqrt{2}A_0}\right )-\frac{1-A_0^2}{A_0^2}\frac{2e^{-\frac{\nu^2}{2}}}{\sqrt{2}\sqrt{\pi}}\right)\nonumber \\
&& -\frac{(1-\mu)(1-\beta)}{w_2}\left (-\frac{(1-A_0^2)\nu}{A_0^2}\frac{e^{\frac{(1-A_0^2)\nu^2}{2A_0^2}}}{A_0}\erfc\left (\frac{-\nu}{\sqrt{2}A_0}\right )-\frac{1-A_0^2}{A_0^2}\frac{2e^{-\frac{\nu^2}{2}}}{\sqrt{2}\sqrt{\pi}}\right )+\frac{1-A_0^2}{A_0^2}\beta\nu.\nonumber \\
\label{eq:boxdetanalIeer4}
\end{eqnarray}
To facilitate the exposition we recall on the following from (\ref{eq:detanalIeer4a})
\begin{eqnarray}
y_1 & = & \frac{\nu}{\sqrt{2}}\nonumber \\
y_2 & = & \frac{\nu}{\sqrt{2}A_0}=\frac{y_1}{A_0}.\label{eq:boxdetanalIeer4a}
\end{eqnarray}
Similarly to (\ref{eq:boxdetanalIeer4b}) we also set
\begin{eqnarray}
z^{(box)}_{1,\nu} & = & \frac{\mu(1-\beta)y_1\lp\sqrt{2}y_2e^{y_2^2}\erfc(y_2)-\sqrt{\frac{2}{\pi}}\rp}{y_2e^{y_2^2}\erfc(y_2)+y_1e^{y_1^2}\erfc(-y_1)}\nonumber \\
z^{(box)}_{2,\nu} & = & -\frac{(1-\mu)(1-\beta) y_1\lp -\sqrt{2}y_2e^{y_2^2}\erfc(-y_2)-\sqrt{\frac{2}{\pi}}\rp}{y_2e^{y_2^2}\erfc(-y_2)+y_1e^{y_1^2}\erfc(y_1)}.\label{eq:boxdetanalIeer4b}
\end{eqnarray}
Combining (\ref{eq:boxdetanalIeer3}), (\ref{eq:boxdetanalIeer4}), (\ref{eq:boxdetanalIeer4a}), and (\ref{eq:boxdetanalIeer4b}) we obtain the following analogous version of (\ref{eq:detanalIeer4c})
\begin{eqnarray}
% \nonumber % Remove numbering (before each equation)
\frac{d\zeta^{(box)}_{\alpha,\beta}(c_3,\nu,A_0)}{d\nu}
& = & \frac{1-A_0^2}{A_0^2}\lp z^{(box)}_{1,\nu}+z^{(box)}_{2,\nu}+\beta\sqrt{2}y_1 \rp.
\label{eq:boxdetanalIeer4c}
\end{eqnarray}
Together, (\ref{eq:boxdetanalIeer4b}) and (\ref{eq:boxdetanalIeer4c}), then also give
\begin{eqnarray}
% \nonumber % Remove numbering (before each equation)
& & \frac{d\zeta^{(box)}_{\alpha,\beta}(c_3,\nu,A_0)}{d\nu}
 =  \frac{1-A_0^2}{A_0^2}\lp z^{(box)}_{1,\nu}+z^{(box)}_{2,\nu}+\beta\sqrt{2}y_1 \rp=0\nonumber \\
& \Longrightarrow & \frac{\mu(1-\beta)y_1\lp\sqrt{2}y_2e^{y_2^2}\erfc(y_2)-\sqrt{\frac{2}{\pi}}\rp}{y_2e^{y_2^2}\erfc(y_2)+y_1e^{y_1^2}\erfc(-y_1)}-\frac{(1-\mu)(1-\beta) y_1\lp -\sqrt{2}y_2e^{y_2^2}\erfc(-y_2)-\sqrt{\frac{2}{\pi}}\rp}{y_2e^{y_2^2}\erfc(-y_2)+y_1e^{y_1^2}\erfc(y_1)}+\beta\sqrt{2}y_1=0.\nonumber \\
\label{eq:boxdetanalIeer4d}
\end{eqnarray}
To continue further transformation of the above derivative we will also rely on the following derivative with respect to $A_0$
\begin{eqnarray}
% \nonumber % Remove numbering (before each equation)
\frac{d\zeta^{(box)}_{\alpha,\beta}(c_3,\nu,A_0)}{dA_0}&=&\frac{d}{dA_0}\left (-\frac{c_3^2}{2}+I_{sph}+\mu(1-\beta)\log{w_1}+(1-\mu)(1-\beta)\log{w_2}+\beta\log{w_3}+\frac{c_3^2}{2(1-A_0^2)}\right )\nonumber \\
&=& \mu(1-\beta)\frac{d\log{w_1}}{dA_0}+(1-\mu)(1-\beta)\frac{d\log{w_2}}{dA_0}+\beta\frac{d\log{w_3}}{dA_0}+\frac{c_3^2A_0}{(1-A_0^2)^2}\nonumber \\
&=& \mu(1-\beta)\frac{d\log{w_1}}{dA_0}+(1-\mu)(1-\beta)\frac{d\log{w_2}}{dA_0}+\beta\frac{-\nu^2-A_0^2}{A_0^3}+\frac{\alpha}{A_0}.\nonumber \\
\label{eq:boxdetanalIeer10}
\end{eqnarray}
In (\ref{eq:detanalIeer11a}) and (\ref{eq:detanalIeer11b}) we have already determined $\frac{d\log{w_1}}{dA_0}$ and $\frac{d\log{w_2}}{dA_0}$ as
\begin{eqnarray}
\frac{d\log{w_1}}{dA_0} & = &
-\frac{e^{\frac{\nu^2}{2A_0^2}}(A_0^2+\nu^2)\erfc(\frac{\nu}{\sqrt{2}A_0})-\sqrt{\frac{2}{\pi}}A_0\nu}
{A_0^3(e^{\frac{\nu^2}{2A_0^2}}\erfc(\frac{\nu}{\sqrt{2}A_0})+A_0e^{\frac{\nu^2}{2}}(\erf(\frac{\nu}{\sqrt{2}})+1))}.\nonumber \\
\frac{d\log{w_2}}{dA_0} & = & -\frac{e^{\frac{\nu^2}{2A_0^2}}(A_0^2+\nu^2)\erfc(\frac{-\nu}{\sqrt{2}A_0})+\sqrt{\frac{2}{\pi}}A_0\nu}
{A_0^3(e^{\frac{\nu^2}{2A_0^2}}\erfc(\frac{-\nu}{\sqrt{2}A_0})+A_0e^{\frac{\nu^2}{2}}(\erf(\frac{-\nu}{\sqrt{2}})+1))}.\nonumber \\
\label{eq:boxdetanalIeer11a}
\end{eqnarray}
Following closely (\ref{eq:boxdetanalIeer4b}) we set
\begin{eqnarray}
z^{(box)}_{1,A_0} & = &
-\mu(1-\beta)\frac{y_2^2}{y_1}\frac{\lp(1+2y_2^2)e^{y_2^2}\erfc(y_2)-\sqrt{2}y_2\sqrt{\frac{2}{\pi}}\rp}{(y_2e^{y_2^2}\erfc(y_2)+y_1 e^{y_1^2}\erfc(-y_1))}
\nonumber \\
z^{(box)}_{2,A_0} & = & -(1-\mu)(1-\beta)\frac{y_2^2}{y_1}\frac{\lp (1+2y_2^2)e^{y_2^2}\erfc(-y_2)+\sqrt{2}y_2\sqrt{\frac{2}{\pi}}\rp}{y_2e^{y_2^2}\erfc(-y_2)+y_1e^{y_1^2}\erfc(y_1)}.\label{eq:boxdetanalIeer11c}
\end{eqnarray}
A combination of (\ref{eq:boxdetanalIeer4}), (\ref{eq:boxdetanalIeer10}), (\ref{eq:boxdetanalIeer11a}), and (\ref{eq:boxdetanalIeer11c}) gives
\begin{eqnarray}
% \nonumber % Remove numbering (before each equation)
\frac{d\zeta^{(box)}_{\alpha,\beta}(c_3,\nu,A_0)}{dA_0}
& = & \mu(1-\beta)\frac{d\log{w_1}}{dA_0}+(1-\mu)(1-\beta)\frac{d\log{w_2}}{dA_0}
+\beta\frac{-\nu^2-A_0^2}{A_0^3}+\frac{\alpha}{A_0}\nonumber \\
& = & z^{(box)}_{1,A_0}+z^{(box)}_{2,A_0}+\frac{(\alpha-\beta-2y_2^2\beta) y_2}{y_1}.\nonumber \\
\label{eq:boxdetanalIeer11d}
\end{eqnarray}
Setting the derivative in (\ref{eq:boxdetanalIeer11d}) to zero and utilizing (\ref{eq:boxdetanalIeer11c}) we also have that
\begin{eqnarray}
% \nonumber % Remove numbering (before each equation)
\frac{d\zeta^{(box)}_{\alpha,\beta}(c_3,\nu,A_0)}{dA_0}
 =  z^{(box)}_{1,A_0}+z^{(box)}_{2,A_0}+\frac{(\alpha-\beta-2y_2^2\beta) y_2}{y_1}=0.
\label{eq:boxdetanalIeer11ea}
\end{eqnarray}
implies
\begin{multline} -\mu(1-\beta)y_2\frac{\lp(1+2y_2^2)e^{y_2^2}\erfc(y_2)-\sqrt{2}y_2\sqrt{\frac{2}{\pi}}\rp}{y_2e^{y_2^2}\erfc(y_2)+y_1 e^{y_1^2}\erfc(-y_1)}  -  (1-\mu)(1-\beta) y_2\frac{\lp (1+2y_2^2)e^{y_2^2}\erfc(-y_2)+\sqrt{2}y_2\sqrt{\frac{2}{\pi}}\rp}{y_2e^{y_2^2}\erfc(-y_2)+y_1e^{y_1^2}\erfc(y_1)} \\
+
\alpha-\beta-2y_2^2\beta
=0.
\label{eq:boxdetanalIeer11e}
\end{multline}
Combining further (\ref{eq:boxdetanalIeer4d}) and (\ref{eq:boxdetanalIeer11e}) we obtain
\begin{eqnarray}
% \nonumber % Remove numbering (before each equation)
& & \frac{d\zeta^{(box)}_{\alpha,\beta}(c_3,\nu,A_0)}{dA_0}
 =  z^{(box)}_{1,A_0}+z^{(box)}_{2,A_0}+\frac{(\alpha-\beta-2y_2^2\beta) y_2}{y_1}=0\nonumber \\
& \Longrightarrow &  -\frac{\mu(1-\beta)y_2 e^{y_2^2}\erfc(y_2)}{y_2e^{y_2^2}\erfc(y_2)+y_1 e^{y_1^2}\erfc(-y_1)}
-\frac{(1-\mu)(1-\beta) y_2 e^{y_2^2}\erfc(-y_2)}{y_2e^{y_2^2}\erfc(-y_2)+y_1e^{y_1^2}\erfc(y_1)}+
\alpha-\beta-2\beta y_2^2+2\beta y_2^2
=0\nonumber \\
& \Longrightarrow &  -\frac{\mu(1-\beta)\sqrt{\frac{1}{\pi}}}{y_2e^{y_2^2}\erfc(y_2)+y_1 e^{y_1^2}\erfc(-y_1)}
+ \frac{(1-\mu)(1-\beta)\sqrt{\frac{1}{\pi}}}{y_2e^{y_2^2}\erfc(-y_2)+y_1e^{y_1^2}\erfc(y_1)}+
\alpha
=0\nonumber \\
& \Longrightarrow &  -\frac{\mu(1-\beta)}{y_2e^{y_2^2}\erfc(y_2)+y_1 e^{y_1^2}\erfc(-y_1)}
+ \frac{(1-\mu)(1-\beta)}{y_2e^{y_2^2}\erfc(-y_2)+y_1e^{y_1^2}\erfc(y_1)}+
\alpha\sqrt{\pi}
=0.\nonumber \\
\label{eq:boxdetanalIeer11f}
\end{eqnarray}
Additionally, a simple algebraic transformation of (\ref{eq:boxdetanalIeer4d}) gives
\begin{eqnarray}
\frac{(1-\mu)(1-\beta)}{y_2e^{y_2^2}\erfc(-y_2)+y_1e^{y_1^2}\erfc(y_1)} & = &
\frac{\mu(1-\beta)}{(y_2e^{y_2^2}\erfc(y_2)+y_1e^{y_1^2}\erfc(-y_1))}
\frac{\lp\sqrt{2}y_2e^{y_2^2}\erfc(y_2)-\sqrt{\frac{2}{\pi}}\rp}{\lp -\sqrt{2}y_2e^{y_2^2}\erfc(-y_2)-\sqrt{\frac{2}{\pi}}\rp}\nonumber \\
&&+\frac{\beta\sqrt{2}}{\lp -\sqrt{2}y_2e^{y_2^2}\erfc(-y_2)-\sqrt{\frac{2}{\pi}}\rp}.
\label{eq:boxdetanalIeer11g}
\end{eqnarray}
After plugging the right side of (\ref{eq:boxdetanalIeer11g}) in (\ref{eq:boxdetanalIeer11f}) we obtain
\begin{eqnarray}
% \nonumber % Remove numbering (before each equation)
& & \frac{d\zeta^{(box)}_{\alpha,\beta}(c_3,\nu,A_0)}{dA_0}
 =  z^{(box)}_{1,A_0}+z^{(box)}_{2,A_0}+\frac{(\alpha -\beta -2\beta y_2^2)y_2}{y_1}=0\nonumber \\
& \Longrightarrow &  -\frac{\mu(1-\beta)}{y_2e^{y_2^2}\erfc(y_2)+y_1 e^{y_1^2}\erfc(-y_1)}
+ \frac{(1-\mu)(1-\beta)}{y_2e^{y_2^2}\erfc(-y_2)+y_1e^{y_1^2}\erfc(y_1)}+
\alpha\sqrt{\pi}
=0\nonumber \\
& \Longrightarrow &  \frac{\mu(1-\beta)}{y_2e^{y_2^2}\erfc(y_2)+y_1 e^{y_1^2}\erfc(-y_1)}
\lp -1+
\frac{\lp\sqrt{2}y_2e^{y_2^2}\erfc(y_2)-\sqrt{\frac{2}{\pi}}\rp}{\lp -\sqrt{2}y_2e^{y_2^2}\erfc(-y_2)-\sqrt{\frac{2}{\pi}}\rp}\rp
\nonumber \\
&&+\frac{\beta\sqrt{2}}{\lp -\sqrt{2}y_2e^{y_2^2}\erfc(-y_2)-\sqrt{\frac{2}{\pi}}\rp}+
\alpha \sqrt{\pi}
=0\nonumber \\
& \Longrightarrow &  y_2e^{y_2^2}\erfc(y_2)+y_1 e^{y_1^2}\erfc(-y_1)=\mu(1-\beta)
\lp
\frac{y_2e^{y_2^2}\erfc(-y_2)+y_2e^{y_2^2}\erfc(y_2)}
{\alpha (\sqrt{\pi}y_2e^{y_2^2}\erfc(-y_2)+1)-\beta}\rp\nonumber \\
& \Longrightarrow &  y_1 e^{y_1^2}\erfc(-y_1)=
\lp
\frac{\mu(1-\beta)2y_2e^{y_2^2}}{\alpha (\sqrt{\pi}y_2e^{y_2^2}\erfc(-y_2)+1)-\beta}\rp-y_2e^{y_2^2}\erfc(y_2)\triangleq f^{(box)}_1(y_2;\alpha,\beta,\mu).\nonumber \\
\label{eq:boxdetanalIeer11h}
\end{eqnarray}
Similarly to what was done in (\ref{eq:detanalIeer11i}), one can also combine (\ref{eq:boxdetanalIeer11f}) and (\ref{eq:boxdetanalIeer11g}) in the following alternative way to obtain
\begin{eqnarray}
% \nonumber % Remove numbering (before each equation)
& & \frac{d\zeta^{(box)}_{\alpha,\beta}(c_3,\nu,A_0)}{dA_0}
 =  z^{(box)}_{1,A_0}+z^{(box)}_{2,A_0}+\frac{(\alpha -\beta-2\beta y_2^2)y_2}{y_1}=0\nonumber \\
 & \Longrightarrow &  -\frac{\mu(1-\beta)}{y_2e^{y_2^2}\erfc(y_2)+y_1 e^{y_1^2}\erfc(-y_1)}
+ \frac{(1-\mu)(1-\beta)}{y_2e^{y_2^2}\erfc(-y_2)+y_1e^{y_1^2}\erfc(y_1)}+
\alpha \sqrt{\pi}
=0\nonumber \\
& \Longrightarrow &  \frac{(1-\mu)(1-\beta)}{y_2e^{y_2^2}\erfc(-y_2)+y_1 e^{y_1^2}\erfc(y_1)}
\lp 1-
\frac{\lp -\sqrt{2}y_2e^{y_2^2}\erfc(-y_2)-\sqrt{\frac{2}{\pi}}\rp}{\lp\sqrt{2}y_2e^{y_2^2}\erfc(y_2)-\sqrt{\frac{2}{\pi}}\rp}\rp
\nonumber \\
&&+\frac{\beta\sqrt{2}}{\lp \sqrt{2}y_2e^{y_2^2}\erfc(y_2)-\sqrt{\frac{2}{\pi}}\rp}+
\alpha \sqrt{\pi}
=0\nonumber \\
& \Longrightarrow &  y_2e^{y_2^2}\erfc(-y_2)+y_1 e^{y_1^2}\erfc(y_1)=-(1-\mu)(1-\beta)
\lp
\frac{y_2e^{y_2^2}\erfc(-y_2)+y_2e^{y_2^2}\erfc(y_2)}{\alpha (\sqrt{\pi}y_2e^{y_2^2}\erfc(y_2)-1)-\beta}\rp\nonumber \\
& \Longrightarrow &  y_1 e^{y_1^2}\erfc(y_1)=
\lp
\frac{(1-\mu)(1-\beta)2y_2e^{y_2^2}}{\alpha(-\sqrt{\pi}y_2e^{y_2^2}\erfc(y_2)+1)-\beta}\rp-y_2e^{y_2^2}\erfc(-y_2)=-f^{(box)}_1(-y_2;\alpha,\beta,1-\mu).\nonumber \\
\label{eq:boxdetanalIeer11i}
\end{eqnarray}
%From (\ref{eq:boxdetanalIeer11h}) and (\ref{eq:boxdetanalIeer11i}) one also has
%\begin{equation}
%\frac{1+\erf(y_1)}{1-\erf(y_1)}= \frac{\erfc(-y_1)}{\erfc(y_1)}=\frac{\frac{1-\beta}{\alpha}
%\lp
%\frac{2y_2e^{y_2^2}}{\sqrt{\pi}y_2e^{y_2^2}\erfc(-y_2)+1}\rp-y_2e^{y_2^2}\erfc(y_2)}{\frac{\beta}{\alpha}
%\lp
%\frac{2y_2e^{y_2^2}}{-\sqrt{\pi}y_2e^{y_2^2}\erfc(y_2)+1}\rp-y_2e^{y_2^2}\erfc(-y_2)}=
%\frac{\frac{1-\beta}{\alpha}
%\lp
%\frac{2}{\sqrt{\pi}y_2e^{y_2^2}\erfc(-y_2)+1}\rp-\erfc(y_2)}{\frac{\beta}{\alpha}
%\lp
%\frac{2}{-\sqrt{\pi}y_2e^{y_2^2}\erfc(y_2)+1}\rp-\erfc(-y_2)}. \\
%\label{eq:boxdetanalIeer11j}
%\end{equation}
%Finally after solving over $\erf(y_1)$ we obtain
%\begin{equation}
%\erf(y_1)=
%\frac{\frac{1-\beta}{\alpha}
%\lp
%\frac{2}{\sqrt{\pi}y_2e^{y_2^2}\erfc(-y_2)+1}\rp-\erfc(y_2)-\lp\frac{\beta}{\alpha}
%\lp
%\frac{2}{-\sqrt{\pi}y_2e^{y_2^2}\erfc(y_2)+1}\rp-\erfc(-y_2)\rp}{\frac{1-\beta}{\alpha}
%\lp
%\frac{2}{\sqrt{\pi}y_2e^{y_2^2}\erfc(-y_2)+1}\rp-\erfc(y_2)+\frac{\beta}{\alpha}
%\lp
%\frac{2}{-\sqrt{\pi}y_2e^{y_2^2}\erfc(y_2)+1}\rp-\erfc(-y_2)},\nonumber \\
%\label{eq:boxdetanalIeer11k}
%\end{equation}
%or
%\begin{equation}
%\erf(y_1)=
%\frac{
%\frac{1-\beta}{\sqrt{\pi}y_2e^{y_2^2}\erfc(-y_2)+1}-
%\frac{\beta}{-\sqrt{\pi}y_2e^{y_2^2}\erfc(y_2)+1}+\alpha\erf(y_2)}{
%\frac{1-\beta}{\sqrt{\pi}y_2e^{y_2^2}\erfc(-y_2)+1}+
%\frac{\beta}{-\sqrt{\pi}y_2e^{y_2^2}\erfc(y_2)+1}-\alpha}\triangleq f^{(box)}_1(y_2;\alpha,\beta).\nonumber \\
%\label{eq:boxdetanalIeer11l}
%\end{equation}
A simple combination of (\ref{eq:boxdetanalIeer11h}) and (\ref{eq:boxdetanalIeer11i}) gives
\begin{equation}
\frac{1+\erf(y_1)}{1-\erf(y_1)}= \frac{\erfc(-y_1)}{\erfc(y_1)}=\frac{f^{(box)}_1(y_2;\alpha,\beta,\mu)}{-f^{(box)}_1(-y_2;\alpha,\beta,1-\mu)}. \\
\label{eq:boxdetanalIeer11j}
\end{equation}
Similarly to (\ref{eq:detanalIeer11k}), after solving over $\erf(y_1)$ from (\ref{eq:boxdetanalIeer11j}) we obtain
\begin{eqnarray}
& & \erf(y_1)  =
\frac{f^{(box)}_1(y_2;\alpha,\beta,\mu)+f^{(box)}_1(-y_2;\alpha,\beta,1-\mu)}{f^{(box)}_1(y_2;\alpha,\beta,\mu)-f^{(box)}_1(-y_2;\alpha,\beta,1-\mu)}\nonumber \\
&\Longleftrightarrow & y_1  =\erfinv\lp
\frac{f^{(box)}_1(y_2;\alpha,\beta,\mu)+f^{(box)}_1(-y_2;\alpha,\beta,1-\mu)}{f^{(box)}_1(y_2;\alpha,\beta,\mu)-f^{(box)}_1(-y_2;\alpha,\beta,1-\mu)}\rp.
\label{eq:boxdetanalIeer11k}
\end{eqnarray}
Combining (\ref{eq:boxdetanalIeer11h}) (or (\ref{eq:boxdetanalIeer11i})) and (\ref{eq:boxdetanalIeer11k}) we also have
\begin{eqnarray}
2\erfinv\lp
\frac{f^{(box)}_1(y_2;\alpha,\beta,\mu)+f^{(box)}_1(-y_2;\alpha,\beta,1-\mu)}{f^{(box)}_1(y_2;\alpha,\beta,\mu)-f^{(box)}_1(-y_2;\alpha,\beta,1-\mu)}\rp
\frac{e^{\erfinv\lp
\frac{f^{(box)}_1(y_2;\alpha,\beta,\mu)+f^{(box)}_1(-y_2;\alpha,\beta,1-\mu)}{f^{(box)}_1(y_2;\alpha,\beta,\mu)-f^{(box)}_1(-y_2;\alpha,\beta,1-\mu)}\rp^2}}
{f^{(box)}_1(y_2;\alpha,\beta,\mu)-f^{(box)}_1(-y_2;\alpha,\beta,1-\mu)}
=1.
\label{eq:boxdetanalIeer11l}
\end{eqnarray}
(\ref{eq:boxdetanalIeer11l}) can be used to determine $y_2$ which can be reused to obtain $y_1$ from (\ref{eq:boxdetanalIeer11h}) and (\ref{eq:boxdetanalIeer11i}) (basically (\ref{eq:boxdetanalIeer11k})). Using (\ref{eq:boxdetanalIerr1}), (\ref{eq:boxdetanalIeer12}), and (\ref{eq:boxdetanalIeer4a}), one can then obtain $\nu$, $A_0$, $c_3$, and $\gamma$ as in (\ref{eq:detanalIeer11m})
\begin{eqnarray}
  \nu & = & y_2\sqrt{2} \nonumber \\
  A_0 & = & \frac{y_1}{y_2} \nonumber \\
  c_3 & = & \frac{(1-A_0^2)\sqrt{\alpha}}{A_0}\nonumber \\
  \gamma & = & \frac{c_3}{2(1-A_0)}.
\label{eq:boxdetanalIeer11m}
\end{eqnarray}
Finally, after all of the above is determined one can compute $\zeta^{(box)}_{\alpha,\beta}(c_3,\nu,A_0)$ in (\ref{eq:boxdetanalIeer3}) following the methodology showcased in (\ref{eq:boxdetanalIeer11ma}) and (\ref{eq:boxdetanalIeer11mb}). As in (\ref{eq:boxdetanalIeer11ma}), we first note that from \cite{Stojnicl1RegPosasymldp} one has
\begin{equation}\label{eq:boxdetanalIeer11ma}
  I_{sph}=-\frac{(1-A_0^2)\alpha}{2}+\alpha\log(A_0).
\end{equation}
A combination of (\ref{eq:boxdetanalIeer11h}), (\ref{eq:boxdetanalIeer11i}), (\ref{eq:boxdetanalIeer11m}), and (\ref{eq:boxdetanalIeer11ma}) then finally produces
\begin{eqnarray}
I^{(box,ub)}_{err,u}& = & \zeta^{(box)}_{\alpha,\beta}(c_3,\nu,A_0) = -\frac{c_3^2}{2}+I_{sph}+\mu(1-\beta)\log{w_1}+(1-\mu)(1-\beta)\log{w_2}+\beta\log{w_3}+\frac{c_3^2}{2(1-A_0^2)}\nonumber \\
& = & -\frac{c_3^2}{2}-\frac{(1-A_0^2)\alpha}{2}+\alpha\log(A_0)+\mu(1-\beta)\log{w_1}+(1-\mu)(1-\beta)\log{w_2}+\beta\log{w_3}+\frac{c_3^2}{2(1-A_0^2)} \nonumber \\
& = & \alpha\log(A_0)+\mu(1-\beta)\log{w_1}+(1-\mu)(1-\beta)\log{w_2}+\beta\log{w_3} \nonumber \\
& = & \alpha\log\lp\frac{y_1}{y_2}\rp+\mu(1-\beta)\log \lp e^{y_2^2}y_2\erfc(y_2)+e^{y_1^2}y_1\erfc(-y_1)\rp\nonumber \\
& & +(1-\mu)(1-\beta)\log\lp e^{y_2^2}y_2\erfc(-y_2)+e^{y_1^2}y_1\erfc(y_1)\rp -(1-\beta)y_1^2-(1-\beta)\log(y_1)\nonumber \\
& & -(1-\beta)\log(2) + \beta (y_2^2-y_1^2)-\beta\log\lp\frac{y_1}{y_2}\rp\nonumber \\
& = & \alpha\log\lp\frac{y_1}{y_2}\rp+\mu(1-\beta)\log \lp \frac{\mu(1-\beta)2y_2e^{y_2^2}}
{\alpha (\sqrt{\pi}y_2e^{y_2^2}\erfc(-y_2)+1)-\beta}\rp\nonumber \\
& & +(1-\mu)(1-\beta)\log\lp \frac{(1-\mu)(1-\beta)2y_2e^{y_2^2}}
{\alpha(-\sqrt{\pi}y_2e^{y_2^2}\erfc(y_2)+1)-\beta}\rp \nonumber \\
& & -(1-\beta)y_1^2-(1-\beta)\log(y_1)
 -(1-\beta)\log(2) + \beta (y_2^2-y_1^2)-\beta\log\lp\frac{y_1}{y_2}\rp\nonumber \\
& = & (\alpha-1)\log\lp\frac{y_1}{y_2}\rp+\mu(1-\beta)\log \lp \frac{\mu(1-\beta)}
{\alpha (\sqrt{\pi}y_2e^{y_2^2}\erfc(-y_2)+1)-\beta}\rp\nonumber \\
& & +(1-\mu)(1-\beta)\log\lp \frac{(1-\mu)(1-\beta)}
{\alpha(-\sqrt{\pi}y_2e^{y_2^2}\erfc(y_2)+1)-\beta}\rp + y_2^2-y_1^2.\nonumber \\
\label{eq:boxdetanalIeer11mb}
\end{eqnarray}
It is rather clear from the presented discussion that the choice for $\nu$, $A_0$, and $c_3$ given in (\ref{eq:boxdetanalIeer11l}) and (\ref{eq:boxdetanalIeer11m}) ensures that (\ref{eq:boxdetanalIeer3a}) is satisfied. In fact, not only that, one can also argue that this choice is besides being a stationary point also a global optimum in (\ref{eq:boxdetanalIeer2}). As mentioned after (\ref{eq:detanalIeer11mb}) in Section \ref{sec:binl1}, we will not pursue these considerations further here. Instead, we will below present a different set of considerations which we will then connect to what we presented above. At that time it will become clear that not only is the choice for $\nu$, $A_0$, and $c_3$ given in (\ref{eq:boxdetanalIeer11l}) and (\ref{eq:boxdetanalIeer11m}) precisely the one that solves the optimization in (\ref{eq:boxdetanalIeer2}) but also precisely the one that determines $I^{(box)}_{err}(\alpha,\beta)$). We summarize the above results in the following theorem.
\begin{theorem}
Assume the setup of Theorem \ref{thm:boxldp2} and assume that a pair $(\alpha,\beta)$  is given. Also, assume that $\alpha>\alpha_w$ where $\alpha_w$ is obtained from the phase transition curve as the value for $\alpha$ that corresponds to the given $\beta$. Set
\begin{equation}\label{eq:boxthmldp3l1PTa}
 f^{(box)}_1(y_2;\alpha,\beta,\mu) \triangleq
\frac{2\mu(1-\beta)y_2e^{y_2^2}}{\alpha(\sqrt{\pi}y_2e^{y_2^2}\erfc(-y_2)+1)-\beta}-y_2e^{y_2^2}\erfc(y_2).
\end{equation}
Also let $y_2$ and $y_1$ satisfy the following \textbf{fundamental characterizations of the box $\ell_1$'s LDP:}

%\begin{center}
%\shadowbox{$
\begin{eqnarray}
2\erfinv\lp
\frac{f^{(box)}_1(y_2;\alpha,\beta,\mu)+f^{(box)}_1(-y_2;\alpha,\beta,1-\mu)}{f^{(box)}_1(y_2;\alpha,\beta,\mu)-f^{(box)}_1(-y_2;\alpha,\beta,1-\mu)}\rp
\frac{e^{\erfinv\lp
\frac{f^{(box)}_1(y_2;\alpha,\beta,\mu)+f^{(box)}_1(-y_2;\alpha,\beta,1-\mu)}{f^{(box)}_1(y_2;\alpha,\beta,\mu)-f^{(box)}_1(-y_2;\alpha,\beta,1-\mu)}\rp^2}}
{f^{(box)}_1(y_2;\alpha,\beta,\mu)-f^{(box)}_1(-y_2;\alpha,\beta,1-\mu)}
& = & 1\nonumber \\
\label{eq:boxthmldp3l1PT}
\end{eqnarray}
and
\begin{eqnarray}
y_1  =  \erfinv\lp
\frac{f^{(box)}_1(y_2;\alpha,\beta,\mu)+f^{(box)}_1(-y_2;\alpha,\beta,1-\mu)}{f^{(box)}_1(y_2;\alpha,\beta,\mu)-f^{(box)}_1(-y_2;\alpha,\beta,1-\mu)}\rp.\nonumber \\
\label{eq:boxthmldp3l1PTa}
\end{eqnarray}
%$}
%-\vspace{-.5in}\begin{equation}
%\label{eq:boxthmldp3l1PT}
%\end{equation}
%\end{center}

\noindent Then
\begin{eqnarray}
I^{(box)}_{err}(\alpha,\beta)\triangleq\lim_{n\rightarrow\infty}\frac{\log{P^{(box)}_{err}}}{n}
& \leq &
(\alpha-1)\log\lp\frac{y_1}{y_2}\rp+\mu(1-\beta)\log \lp \frac{\mu(1-\beta)}{\alpha\lp\sqrt{\pi}e^{y_2^2}y_2\erfc(-y_2)+1\rp-\beta}\rp\nonumber \\
& & +(1-\mu)(1-\beta)\log\lp \frac{(1-\mu)(1-\beta)}{\alpha\lp -\sqrt{\pi}e^{y_2^2}y_2\erfc(y_2)+1\rp-\beta}\rp + y_2^2-y_1^2 \nonumber \\
& \triangleq & I^{(box)}_{ldp}(\alpha,\beta).
\label{eq:boxldpthm3Ierrub1}
\end{eqnarray}
\noindent Moreover, for the above choice of $y_2$ and  $y_1$, $\nu$, $A_0$, $c_3$, and $\gamma$ in (\ref{eq:boxdetanalIeer11m}) achieve an optimum in (\ref{eq:boxdetanalIeer2}) (and ultimately in (\ref{eq:boxldpthm2Ierrub1})).
\label{thm:boxldp3}
\end{theorem}
\begin{proof} Follows from the above discussion.
\end{proof}
As in Section \ref{sec:binl1}, the above upper tail results remain correct in the lower tail regime as well. A short argument that we for the completeness present below (and that closely follows what we presented in Section \ref{sec:lowertail}) confirms this.

%%%%%%%%%%%%%%%%%%%%%%%%%%%%%%%%%%%%%%%%%%%%%%%%%%%%%%%%%%%%%%%%%
\subsubsection{Lower tail}
\label{sec:boxlowertail}
%%%%%%%%%%%%%%%%%%%%%%%%%%%%%%%%%%%%%%%%%%%%%%%%%%%%%%%%%%%%%%%%%

%We will split the presentation in this subsection into two parts. In the first part we will discuss the key components of the analysis needed to establish the lower bound on the lower tail.  In the second part we will then design a reverse strategy that produces the corresponding upper bound and ultimately makes the lower tail estimates fully exact.
%
%%%%%%%%%%%%%%%%%%%%%%%%%%%%%%%%%%%%%%%%%%%%%%%%%%%%%%%%%%%%%%%%%%
%\subsubsection{Lower bound}
%\label{sec:boxlowerlowertail}
%%%%%%%%%%%%%%%%%%%%%%%%%%%%%%%%%%%%%%%%%%%%%%%%%%%%%%%%%%%%%%%%%%

As in Section \ref{sec:lowertail}, we rely on the strategy introduced in
\cite{Stojnicl1RegPosasymldp} and write
\begin{equation}
P_{cor}
\leq  \min_{t_1}\min_{c_3\geq 0}
Ee^{c_3\|\g\|_2}Ee^{-c_3w(\h,S^{(box)}_w)}e^{-c_3t_1}/P(g\geq t_1),
\label{eq:boxldpprob3lower}
\end{equation}
where $P_{cor}=1-P^{(box)}_{err}$ is the probability that (\ref{eq:boxl1bin}) does produce the solution of (\ref{eq:l0}). Following (\ref{eq:ldpasymp1lower}) we also introduce the rate of $P_{cor}$'s decay
\begin{equation}\label{eq:boxldpasymp1lower}
  I^{(box)}_{cor}(\alpha,\beta)\triangleq\lim_{n\rightarrow\infty}\frac{\log{P_{cor}}}{n}.
\end{equation}
The lower tail analogue to Theorem \ref{thm:boxldp2} is then the following theorem.
\begin{theorem}
Assume the setup of Theorem \ref{thm:boxldp2}. Then
\begin{eqnarray}
I^{(box)}_{cor}(\alpha,\beta)& \triangleq & \lim_{n\rightarrow\infty}\frac{\log{P_{cor}}}{n}\nonumber \\
& \leq  &\min_{c_3\geq 0}\left (-\frac{c_3^2}{2}+I_{sph}^++\max_{\nu\geq 0,\gamma^{(s)}\geq 0} (\mu(1-\beta)\log{w_1}+(1-\mu)(1-\beta)\log{w_2}+\beta\log{w_3}-c_3\gamma)\right )\nonumber \\
& \triangleq & I_{cor,l}^{(box,ub)}(\alpha,\beta), \nonumber \\
\label{eq:boxldpthm2Icorub1}
\end{eqnarray}
where
\begin{eqnarray}
% \nonumber % Remove numbering (before each equation)
I_{sph}^+ &=& \widehat{\gamma_+}c_3-\frac{\alpha d}{2}\log\left (1-\frac{c_3}{2\widehat{\gamma_+}}\right )\nonumber \\
  \widehat{\gamma_+} &=& \frac{2c_3+\sqrt{4c_3^2+16\alpha }}{8}\nonumber \\
w_1 &=& \frac{1}{2}\lp\frac{1}{\sqrt{2\pi}}\int_{h}e^{-h^2/2}e^{-c_3\max(h-\nu,0)^2/4/\gamma}dh
  =\frac{e^{\frac{-c_3\nu^2/4/\gamma}{1+c_3/2/\gamma}}}{\sqrt{1+c_3/2/\gamma}}\erfc\left (\frac{\nu}{\sqrt{2}\sqrt{1+c_3/2/\gamma}}\right )+\erf\left (\frac{\nu}{\sqrt{2}}\right )+1\rp\nonumber \\
w_2 &=& \frac{1}{2}\lp\frac{1}{\sqrt{2\pi}}\int_{h}e^{-h^2/2}e^{-c_3\max(h+\nu,0)^2/4/\gamma}dh
  =\frac{e^{\frac{-c_3\nu^2/4/\gamma}{1+c_3/2/\gamma}}}{\sqrt{1+c_3/2/\gamma}}\erfc\left (\frac{-\nu}{\sqrt{2}\sqrt{1+c_3/2/\gamma}}\right )+\erf\left (\frac{-\nu}{\sqrt{2}}\right )+1\rp\nonumber \\
  w_3 &=&\frac{1}{\sqrt{2\pi}}\int_{h}e^{-h^2/2}e^{-c_3(h+\nu)^2/4/\gamma}dh= \frac{e^{\frac{-c_3\nu^2/4/\gamma}{1+c_3/2/\gamma}}}{\sqrt{1+c_3/2/\gamma}}.
\label{eq:boxldpthm2perrub2lower}
\end{eqnarray}\label{thm:boxldp2lower}
\end{theorem}
\begin{proof} Follows in exactly the same way as the proof of Theorem \ref{thm:ldp2lower} and ultimately the corresponding result for the lower tail of the standard $\ell_1$ LDP in \cite{Stojnicl1RegPosasymldp}.
\end{proof}
Instead solving the above optimization problem, one can just quickly observe that
the change $c_3\rightarrow -c_3$ gives as in Section \ref{sec:boxanalysisIerr}
\begin{equation}\label{eq:boxdetanalIcor1}
  A_{0}\triangleq\sqrt{1-\frac{c_3}{2\gamma}},
\end{equation}
and
\begin{equation}\label{eq:boxdetanalIcor2}
I_{cor,l}^{(box,ub)}(\alpha,\beta)\triangleq \min_{c_3\leq 0}\max_{\nu\geq 0,A_0\leq 1}\zeta^{(box)}_{\alpha,\beta}(c_3,\nu,A_0)
\end{equation}
where
\begin{eqnarray}
% \nonumber % Remove numbering (before each equation)
\zeta^{(box)}_{\alpha,\beta}(c_3,\nu,A_0)&=&\left (-\frac{c_3^2}{2}+I_{sph}^+
+\mu(1-\beta)\log{w_1}+(1-\mu)(1-\beta)\log{w_2}+\beta\log{w_3}+\frac{c_3^2}{2(1-A_0^2)}\right )\nonumber \\
I_{sph}^+ &=& -\widehat{\gamma^+}c_3-\frac{\alpha }{2}\log\left (1+\frac{c_3}{2\widehat{\gamma^+}}\right )\nonumber \\
  \widehat{\gamma^+} &=& \frac{-c_3+\sqrt{c_3^2+4\alpha}}{4}=-\widehat{\gamma},\nonumber \\
\label{eq:boxdetanalIcor3}
\end{eqnarray}
and $w_1$, $w_2$, and $w_3$ are as in (\ref{eq:boxdetanalIeer3}). One then observes that $\zeta^{(box)}_{\alpha,\beta}(c_3,\nu,A_0)$ defined in (\ref{eq:boxdetanalIcor3}) is exactly the same as the corresponding one in (\ref{eq:boxdetanalIeer3}) which means that one can proceed with the computation of all the derivatives as earlier and the values we have chosen for $c_3$, $\nu$, $\gamma$, and $A_0$ in the upper tail regime will have the same form. The following theorem summarizes the final results.
\begin{theorem}
Assume the setup of Theorem \ref{thm:boxldp3} and assume that a pair $(\alpha,\beta)$  is given. Differently from Theorem \ref{thm:boxldp3}, assume that $\alpha<\alpha_w$. Also let $y_2$ and $y_1$ satisfy the \textbf{fundamental \emph{box} $\ell_1$'s LDP characterizations} as in Theorem \ref{thm:boxldp3}. Then choosing $\nu$, $c_3$, and $\gamma$ in the optimization problem in (\ref{eq:boxldpthm2Icorub1}) as $\nu$, $-c_3$, and $\gamma$  from Theorem \ref{thm:boxldp3} (or equivalently, choosing $\nu$, $c_3$, and $A_0$ in the optimization problem in (\ref{eq:boxdetanalIcor2}) as $\nu$, $c_3$, and $A_0$  from Theorem \ref{thm:boxldp3}) gives needed
$\zeta^{(box)}_{\alpha,\beta}(c_3,\nu,A_0)$.
\label{thm:boxldp3lower}
\end{theorem}
\begin{proof} Follows from the considerations leading to Theorem \ref{thm:boxldp3}.
\end{proof}
Similarly to what was observed after Theorem \ref{thm:ldp3lower} we have for $\alpha<\alpha_w$, $y_1>y_2$ which means $A_0>1$ and finally $c_3<0$. This is of course different from the upper tail regime (i.e. Theorem \ref{thm:boxldp3}), where the reasoning is reversed and $c_3>0$.

%%%%%%%%%%%%%%%%%%%%%%%%%%%%%%%%%%%%%%%%%%%%%%%%%%%%%%%%%%%%%%%%%
\subsection{High-dimensional geometry}
\label{sec:boxhdg}
%%%%%%%%%%%%%%%%%%%%%%%%%%%%%%%%%%%%%%%%%%%%%%%%%%%%%%%%%%%%%%%%%

In this section we provide an analysis that is analogous to the one provided in Section \ref{sec:hdg} for the binary $\ell_1$.
As in Section \ref{sec:hdg}, the analysis that we will present below relies on a high-dimensional integral geometry approach. Many aspects of the analysis presented in Section \ref{sec:hdg} will be directly applicable here as well. Some of them though will be different. As usual, we will focus on highlighting the key differences.

As earlier, we will here again mostly focus on the upper tail regime (the results for the lower tail will automatically follow). Mathematically, the upper tail regime will assume that we are given a pair $(\alpha,\beta)$ such that $\alpha>\alpha_w$, where $\alpha_w$ is as in Theorems \ref{thm:boxldp3}, \ref{thm:boxldp2lower}, and \ref{thm:boxldp3lower}.

We will rely on the following observations from \cite{Stojnicl1BnBxfinn}
\begin{equation}\label{eq:boxhdg1}
  \Psi^{(box)}_{net}(\alpha,\beta)=I^{(box)}_{err}(\alpha,\beta)\triangleq\lim_{n\rightarrow\infty}\frac{\log{P^{(box)}_{err}}}{n}
  =\max_{\gamma_g\in(0,\min(1-\alpha,(1-\mu)(1-\beta)))}
  \lp \psicom^{(box)}+\psiint^{(box)}+\psiext^{(box)}\rp,
\end{equation}
where
\begin{eqnarray}
H(x) & = & x\log(x)+(1-x)\log(1-x)\nonumber \\
\psicom^{(box)} & = & -\mu(1-\beta)H\lp \frac{1-\alpha-\gamma_g}{\mu(1-\beta)} \rp
-(1-\mu)(1-\beta)H\lp \frac{(1-\mu)(1-\beta)-\gamma_g}{(1-\mu)(1-\beta)} \rp\nonumber \\
\psiint^{(box)} & = & \min_{y_i\geq 0} (\alpha y_i^2 +((1-\beta)\mu-(1-\alpha-\gamma_g))\log(\erfc(y_i))+((1-\beta)(1-\mu)-\gamma_g)\log(\erfc(-y_i)))\nonumber \\
& & - (\alpha-\beta) \log(2)\nonumber\\
\psiext^{(box)} & = & \max_{y_e\geq 0} (-\alpha y_e^2 +(1-\alpha-\gamma_g)\log(\erfc(-y_e))+\gamma_g\log(\erfc(y_e)))-(1-\alpha)\log(2). \label{eq:boxhdg2}
\end{eqnarray}
As in Section \ref{sec:binl1}, instead of solving the above problem numerically, we will here raise the bar a bit higher and look for an explicit solution. We will start with the following analogue to (\ref{eq:hdg1a})
\begin{equation}\label{eq:boxhdg1a}
\max_{\gamma_g\in(0,\beta),y_e\geq 0}\min_{y_i\geq 0} \zeta^{(box)}_{\alpha,\beta}(\gamma_g,y_e,y_i)
\end{equation}
where
\begin{eqnarray}\label{eq:boxhdg1b}
\zeta^{(box)}_{\alpha,\beta}(\gamma_g,y_e,y_i) & = &
-\mu(1-\beta)H\lp \frac{1-\alpha-\gamma_g}{\mu(1-\beta)} \rp
-(1-\mu)(1-\beta)H\lp \frac{(1-\mu)(1-\beta)-\gamma_g}{(1-\mu)(1-\beta)} \rp\nonumber \\
&& + \alpha y_i^2 +((1-\beta)\mu-(1-\alpha-\gamma_g))\log(\erfc(y_i))+((1-\beta)(1-\mu)-\gamma_g)\log(\erfc(-y_i)))\nonumber \\
& & - (\alpha-\beta) \log(2)\nonumber\\
& & +
(-\alpha y_e^2 +(1-\alpha-\gamma_g)\log(\erfc(-y_e))+\gamma_g\log(\erfc(y_e)))-(1-\alpha)\log(2).
\end{eqnarray}
Following further Section \ref{sec:binl1}, we below briefly discuss a few useful properties of $\zeta^{(box)}_{\alpha,\beta}(\gamma_g,y_e,y_i)$.

%%%%%%%%%%%%%%%%%%%%%%%%%%%%%%%%%%%%%%%%%%%%%%%%%%%%%%%%%%%%%%%%%
\subsubsection{Properties of $\zeta^{(box)}_{\alpha,\beta}(\gamma_g,y_e,y_i)$}
\label{sec:boxhdgpropzeta}
%%%%%%%%%%%%%%%%%%%%%%%%%%%%%%%%%%%%%%%%%%%%%%%%%%%%%%%%%%%%%%%%%

Here we will quickly establish that for any fixed $(\alpha,\beta)\in (0,1)\times(0,\alpha)$, $\zeta^{(box)}_{\alpha,\beta}(\gamma_g,y_e,y_i)$ is concave in $\gamma_g$ and $y_e$ and convex in $y_i$ in the optimizing domain (these are precisely the same properties that  $\zeta_{\alpha,\beta}(\gamma_g,y_e,y_i)$ exhibits).

Concavity in $y_e$ follows automatically from the concavity of $\zeta_{\alpha,\beta}(\gamma_g,y_e,y_i)$. Concavity in $\gamma_g$ follows in a very similar manner. We first compute the first derivative with respect to $\gamma_g$
\begin{eqnarray}\label{eq:boxhdg1o}
\frac{d\zeta^{(box)}_{\alpha,\beta}(\gamma_g,y_e,y_i)}{d\gamma_g}
& = &-\mu(1-\beta)\frac{dH\lp \frac{1-\alpha-\gamma_g}{\mu(1-\beta)} \rp}{d\gamma_g}-(1-\mu)(1-\beta)\frac{d H\lp \frac{(1-\mu)(1-\beta)-\gamma_g}{(1-\mu)(1-\beta)} \rp}{d\gamma_g}\nonumber \\
& & +\log\lp\frac{\erfc(y_i)\erfc(y_e)}{\erfc(-y_i)\erfc(-y_e)}\rp\nonumber \\
& = &\log\lp \frac{\frac{1-\alpha-\gamma_g}{\mu(1-\beta)}}{1-\frac{1-\alpha-\gamma_g}{\mu(1-\beta)}} \rp
+\log \lp \frac{\frac{(1-\mu)(1-\beta)-\gamma_g}{(1-\mu)(1-\beta)}}
{1-\frac{(1-\mu)(1-\beta)-\gamma_g}{(1-\mu)(1-\beta)}} \rp+\log\lp\frac{\erfc(y_i)\erfc(y_e)}{\erfc(-y_i)\erfc(-y_e)}\rp\nonumber \\
& = &\log\lp \frac{1-\alpha-\gamma_g}{\mu(1-\beta)-(1-\alpha-\gamma_g)} \rp+\log \lp \frac{(1-\mu)(1-\beta)-\gamma_g}
{\gamma_g} \rp\nonumber \\
& & +\log\lp\frac{\erfc(y_i)\erfc(y_e)}{\erfc(-y_i)\erfc(-y_e)}\rp,
\end{eqnarray}
and then the second one as well
\begin{eqnarray}\label{eq:boxhdg1oa}
\frac{d^2\zeta^{(box)}_{\alpha,\beta}(\gamma_g,y_e,y_i)}{d\gamma_g^2}
& = & \frac{d\lp\log\lp \frac{1-\alpha-\gamma_g}{\mu(1-\beta)-(1-\alpha-\gamma_g)} \rp+\log \lp \frac{(1-\mu)(1-\beta)-\gamma_g}
{\gamma_g} \rp+\log\lp\frac{\erfc(y_i)\erfc(y_e)}{\erfc(-y_i)\erfc(-y_e)}\rp\rp}{d\gamma_g}\nonumber \\
& = & \frac{-\frac{1-\alpha-\gamma_g}{(\mu(1-\beta)-(1-\alpha-\gamma_g))^2}-\frac{1}{\mu(1-\beta)-(1-\alpha-\gamma_g)}}
{\frac{1-\alpha-\gamma_g}{\mu(1-\beta)-(1-\alpha-\gamma_g)}}
+\frac{-\frac{(1-\mu)(1-\beta)-\gamma_g}
{\gamma_g^2}-\frac{1}
{\gamma_g}}{\frac{(1-\mu)(1-\beta)-\gamma_g}
{\gamma_g}}\nonumber \\
& = & \frac{-\mu(1-\beta)}
{(1-\alpha-\gamma_g)(\mu(1-\beta)-(1-\alpha-\gamma_g))}
+\frac{-(1-\mu)(1-\beta)}{\gamma_g((1-\mu)(1-\beta)-\gamma_g)}
\nonumber \\
& < & 0.
\end{eqnarray}

To check convexity in $y_i$ we will need a bit of adaptation of the arguments from Section \ref{sec:binl1}. We first recall that, as in Section \ref{sec:binl1}, we will below rely on the following
\begin{eqnarray}\label{eq:boxhdg1c}
\frac{d\log(\erfc(x))}{dx}=-\frac{2e^{-x^2}}{\sqrt{\pi}\erfc(x)} \quad \mbox{and} \quad \frac{d\log(\erfc(-x))}{dx}=\frac{2e^{-x^2}}{\sqrt{\pi}\erfc(-x)},
\end{eqnarray}
and
\begin{eqnarray}\label{eq:boxhdg1d}
\frac{d^2\log(\erfc(x))}{dx^2}=\frac{4e^{-x^2}\sqrt{\pi}x\erfc(x)-4e^{-2x^2}}{\pi\erfc(x)^2} \quad \mbox{and} \quad \frac{d^2\log(\erfc(-x))}{dx^2}=\frac{-4e^{-x^2}\sqrt{\pi}x\erfc(-x)-4e^{-2x^2}}{\pi\erfc(-x)^2}.\nonumber \\
\end{eqnarray}
Now, we have for the first derivative with respect to $y_1$
\begin{eqnarray}\label{eq:boxhdg1e}
\frac{d\zeta^{(box)}_{\alpha,\beta}(\gamma_g,y_e,y_i)}{dy_i}
& = &\frac{d(\alpha y_i^2 +((1-\beta)\mu-(1-\alpha-\gamma_g))\log(\erfc(y_i))+((1-\beta)(1-\mu)-\gamma_g)\log(\erfc(-y_i)))}{dy_i}\nonumber \\
& = & 2\alpha y_i-\frac{2((1-\beta)\mu-(1-\alpha-\gamma_g))e^{-y_i^2}}{\sqrt{\pi}\erfc(y_i)}+\frac{2((1-\beta)(1-\mu)-\gamma_g)e^{-y_i^2}}{\sqrt{\pi}\erfc(-y_i)},
\end{eqnarray}
and for the second
\begin{eqnarray}\label{eq:boxhdg1f}
\frac{d^2\zeta^{(box)}_{\alpha,\beta}(\gamma_g,y_e,y_i)}{dy_i^2}
& = & \frac{d^2(\alpha y_i^2 +((1-\beta)\mu-(1-\alpha-\gamma_g))\log(\erfc(y_i))+((1-\beta)(1-\mu)-\gamma_g)\log(\erfc(-y_i)))}{dy_i^2}\nonumber \\
& = & 2\alpha+ ((1-\beta)\mu-(1-\alpha-\gamma_g))\frac{4e^{-y_i^2}\sqrt{\pi}y_i\erfc(y_i)-4e^{-2y_i^2}}{\pi\erfc(y_i)^2}\nonumber \\
& & + ((1-\beta)(1-\mu)-\gamma_g)\frac{-4e^{-y_i^2}\sqrt{\pi}y_i\erfc(-y_i)-4e^{-2y_i^2}}{\pi\erfc(-y_i)^2}\nonumber \\
& \geq & 2(\alpha-\beta)+ ((1-\beta)\mu-(1-\alpha-\gamma_g))\frac{4e^{-y_i^2}\sqrt{\pi}y_i\erfc(y_i)-4e^{-2y_i^2}}{\pi\erfc(y_i)^2}\nonumber \\
& & + ((1-\beta)(1-\mu)-\gamma_g)\frac{-4e^{-y_i^2}\sqrt{\pi}y_i\erfc(-y_i)-4e^{-2y_i^2}}{\pi\erfc(-y_i)^2}\nonumber \\
& = & 2((1-\beta)\mu-(1-\alpha-\gamma_g))\frac{\pi\erfc(y_i)^2+2e^{-y_i^2}\sqrt{\pi}y_i\erfc(y_i)-2e^{-2y_i^2}}{\pi\erfc(y_i)^2}\nonumber \\
& & + 2((1-\beta)(1-\mu)-\gamma_g)\frac{\pi\erfc(-y_i)^2-2e^{-y_i^2}\sqrt{\pi}y_i\erfc(-y_i)-2e^{-2y_i^2}}{\pi\erfc(-y_i)^2}\nonumber \\
& > & 0,
\end{eqnarray}
where the last inequality follows by repeating step by step the line of arguments after (\ref{eq:hdg1f}).

%%%%%%%%%%%%%%%%%%%%%%%%%%%%%%%%%%%%%%%%%%%%%%%%%%%%%%%%%%%%%%%%%
\subsubsection{Solving the derivative equations}
\label{sec:boxhdgdereqns}
%%%%%%%%%%%%%%%%%%%%%%%%%%%%%%%%%%%%%%%%%%%%%%%%%%%%%%%%%%%%%%%%%

After establishing the above properties of $\zeta^{(box)}_{\alpha,\beta}(\gamma_g,y_e,y_i)$ we now focus on solving the following system of derivative equations
\begin{equation}\label{eq:boxhdg1n}
  \frac{d\zeta^{(box)}_{\alpha,\beta}(\gamma_g,y_e,y_i)}{d\gamma_g}= \frac{d\zeta^{(box)}_{\alpha,\beta}(\gamma_g,y_e,y_i)}{dy_e}
=  \frac{d\zeta^{(box)}_{\alpha,\beta}(\gamma_g,y_e,y_i)}{dy_i}=0.
\end{equation}
The derivatives with respect to $\gamma_g$ and $y_i$ are computed in (\ref{eq:boxhdg1o}) and (\ref{eq:boxhdg1e}), respectively. We also recall on the derivative with respect to $y_e$ from (\ref{eq:hdg1h})
\begin{eqnarray}\label{eq:boxhdg1h}
\frac{d\zeta^{(box)}_{\alpha,\beta}(\gamma_g,y_e,y_i)}{dy_e}
& = &\frac{d(-\alpha y_e^2 +(1-\alpha-\gamma_g)\log(\erfc(-y_e))+\gamma_g\log(\erfc(y_e)))}{dy_e}\nonumber \\
& = & -2\alpha y_e+\frac{2(1-\alpha-\gamma_g)e^{-y_e^2}}{\sqrt{\pi}\erfc(-y_e)}-\frac{2\gamma_ge^{-y_e^2}}{\sqrt{\pi}\erfc(y_e)}.
\end{eqnarray}
From (\ref{eq:boxhdg1e}) (after setting the derivative to zero) we have
\begin{eqnarray}\label{eq:boxhdg1p}
& & \gamma_g\lp\frac{1}{\erfc(y_i)}+ \frac{1}{\erfc(-y_i)}\rp=
\sqrt{\pi}\alpha e^{y_i^2} y_i-\frac{(1-\beta)\mu-(1-\alpha)}{\erfc(y_i)}+\frac{(1-\mu)(1-\beta)}{\erfc(-y_i)}\nonumber \\
&\Longleftrightarrow & \gamma_g=
\frac{\alpha}{2}\lp\sqrt{\pi} e^{y_i^2} y_i\erfc(y_i)\erfc(-y_i)-\erfc(-y_i)\rp-(1-\beta)\mu+1-\frac{\erfc(y_i)\beta}{2}.
\end{eqnarray}
Analogously to (\ref{eq:hdg1q}) we set
\begin{equation}\label{eq:boxhdg1q}
  A_{box}=\frac{(1-\beta)\mu-(1-\alpha-\gamma_g)}{1-\alpha-\gamma_g}\frac{\gamma_g}{(1-\mu)(1-\beta)-\gamma_g}.
\end{equation}
Then (\ref{eq:boxhdg1o}) gives
\begin{eqnarray}\label{eq:boxhdg1r}
& &   A_{box}=\frac{\erfc(y_i)\erfc(y_e)}{\erfc(-y_i)\erfc(-y_e)}\nonumber \\
& \Longleftrightarrow &  \frac{\erfc(-y_e)}{\erfc(y_e)}=\frac{\erfc(y_i)}{\erfc(-y_i)A_{box} }\nonumber \\
& \Longleftrightarrow &  y_e=\erfinv\lp\frac{\erfc(y_i)-A_{box}\erfc(-y_i)}{\erfc(y_i)+A_{box}\erfc(-y_i)}\rp.
\end{eqnarray}
As in Section \ref{sec:binl1}, one can now use $\gamma_g$ from (\ref{eq:boxhdg1p}) and $y_e$ from (\ref{eq:boxhdg1r}) and combine it in the right side of (\ref{eq:boxhdg1h}). One can then equal the right side of (\ref{eq:boxhdg1h}) to zero and effectively obtain one equation with $y_i$ as the only unknown. After determining $y_i$ from such an equation one can use it to compute $\gamma_g$ through (\ref{eq:boxhdg1p}) and $y_e$ through (\ref{eq:boxhdg1r}). As in Section \ref{sec:binl1}, we will focus on presenting the solution in a bit more explicit way and if possible in a way that is a bit more connected to what was presented in Section \ref{sec:boxldp}. To that end, we start, as usual, by setting the derivative in (\ref{eq:boxhdg1h}) to zero to obtain
\begin{eqnarray}\label{eq:boxhdg1s}
& & \gamma_g\lp\frac{1}{\erfc(y_e)}+ \frac{1}{\erfc(-y_e)}\rp=
-\sqrt{\pi}\alpha e^{y_e^2} y_e+\frac{1-\alpha}{\erfc(-y_e)}\nonumber \\
&\Longleftrightarrow & \gamma_g=
\frac{\alpha}{2}\lp -\sqrt{\pi} e^{y_e^2} y_e\erfc(-y_e)-1\rp\erfc(y_e)+\frac{\erfc(y_e)}{2}.
\end{eqnarray}
Using $\gamma_g$ from (\ref{eq:boxhdg1p}) in (\ref{eq:boxhdg1q}) gives
\begin{eqnarray}\label{eq:boxhdg1t}
  A_{box} & = & \frac{\alpha+\frac{\alpha}{2}\lp\sqrt{\pi} e^{y_i^2} y_i\erfc(y_i)\erfc(-y_i)-\erfc(-y_i)\rp-\frac{\erfc(y_i)\beta}{2}}{1-\alpha-\lp\frac{\alpha}{2}\lp\sqrt{\pi} e^{y_i^2} y_i\erfc(y_i)\erfc(-y_i)-\erfc(-y_i)\rp-(1-\beta)\mu+1-\frac{\erfc(y_i)\beta}{2}\rp}\nonumber \\
  & & \times
  \frac{\frac{\alpha}{2}\lp\sqrt{\pi} e^{y_i^2} y_i\erfc(y_i)\erfc(-y_i)-\erfc(-y_i)\rp-(1-\beta)\mu+1-\frac{\erfc(y_i)\beta}{2}}
  {-\frac{\alpha}{2}\lp\sqrt{\pi} e^{y_i^2} y_i\erfc(y_i)\erfc(-y_i)-\erfc(-y_i)\rp-\frac{\erfc(-y_i)\beta}{2}}.
\end{eqnarray}
Transforming a bit more $A_{box}$ becomes
\begin{equation}\label{eq:boxhdg1u}
  A_{box}=\frac{2(1-\beta)(1-\mu)+\lp\alpha\lp\sqrt{\pi} e^{y_i^2} y_i\erfc(y_i)-1\rp+\beta\rp\erfc(-y_i)}{2(1-\beta)\mu-\lp\alpha\lp\sqrt{\pi} e^{y_i^2} y_i\erfc(-y_i)+1\rp-\beta\rp\erfc(y_i)}\frac{\alpha\lp\sqrt{\pi} e^{y_i^2} y_i\erfc(-y_i)+1\rp-\beta}
  {\alpha\lp-\sqrt{\pi} e^{y_i^2} y_i\erfc(y_i)+1\rp-\beta}\frac{\erfc(y_i)}{\erfc(-y_i)}.
\end{equation}
Combining (\ref{eq:boxhdg1r}) and (\ref{eq:boxhdg1u}) we also have
\begin{eqnarray}\label{eq:boxhdg1v}
  \frac{\erfc(-y_e)}{\erfc(y_e)}=\frac{\erfc(y_i)}{\erfc(-y_i)A_{box} }
  & = & \frac{2(1-\beta)\mu-\lp\alpha\lp\sqrt{\pi} e^{y_i^2} y_i\erfc(-y_i)+1\rp-\beta\rp\erfc(y_i)}{2(1-\beta)(1-\mu)+\lp\alpha\lp\sqrt{\pi} e^{y_i^2} y_i\erfc(y_i)-1\rp+\beta\rp\erfc(-y_i)}\nonumber \\
  & & \times \frac{\alpha\lp-\sqrt{\pi} e^{y_i^2} y_i\erfc(y_i)+1\rp-\beta}{\alpha\lp\sqrt{\pi} e^{y_i^2} y_i\erfc(-y_i)+1\rp-\beta}.
\end{eqnarray}
We also recall on (\ref{eq:boxdetanalIeer11h}) and (\ref{eq:boxdetanalIeer11i}) and note that
\begin{eqnarray}\label{eq:boxhdg1w}
\frac{\lp
\frac{\mu(1-\beta)2y_2e^{y_2^2}}{\alpha (\sqrt{\pi}y_2e^{y_2^2}\erfc(-y_2)+1)-\beta}\rp-y_2e^{y_2^2}\erfc(y_2)}
{\lp
\frac{(1-\mu)(1-\beta)2y_2e^{y_2^2}}{\alpha(-\sqrt{\pi}y_2e^{y_2^2}\erfc(y_2)+1)-\beta}\rp-y_2e^{y_2^2}\erfc(-y_2)}
=\frac{f^{(box)}_1(y_2;\alpha,\beta,\mu)}{-f^{(box)}_1(-y_2;\alpha,\beta,1-\mu)},
\end{eqnarray}
and
\begin{eqnarray}\label{eq:boxhdg1x}
\frac{
\mu(1-\beta)2y_2e^{y_2^2}-\lp\alpha (\sqrt{\pi}y_2e^{y_2^2}\erfc(-y_2)+1)-\beta\rp y_2e^{y_2^2}\erfc(y_2)}
{(1-\mu)(1-\beta)2y_2e^{y_2^2}-\lp\alpha(-\sqrt{\pi}y_2e^{y_2^2}\erfc(y_2)+1)-\beta\rp y_2e^{y_2^2}\erfc(-y_2)}
& \times & \frac{\alpha\lp-\sqrt{\pi} e^{y_2^2} y_i\erfc(y_2)+1\rp-\beta}{\alpha\lp\sqrt{\pi} e^{y_2^2} y_2\erfc(-y_2)+1\rp-\beta}\nonumber \\
& & =\frac{f^{(box)}_1(y_2;\alpha,\beta,\mu)}{-f^{(box)}_1(-y_2;\alpha,\beta,1-\mu)}.\nonumber \\
\end{eqnarray}
From (\ref{eq:boxhdg1v}) and (\ref{eq:boxhdg1x}) one also has
\begin{eqnarray}\label{eq:boxhdg1y}
  \frac{\erfc(-y_e)}{\erfc(y_e)}
 & = & \frac{2(1-\beta)\mu-\lp\alpha\lp\sqrt{\pi} e^{y_i^2} y_i\erfc(-y_i)+1\rp-\beta\rp\erfc(y_i)}{2(1-\beta)(1-\mu)+\lp\alpha\lp\sqrt{\pi} e^{y_i^2} y_i\erfc(y_i)-1\rp+\beta\rp\erfc(-y_i)}
   \frac{\alpha\lp-\sqrt{\pi} e^{y_i^2} y_i\erfc(y_i)+1\rp-\beta}{\alpha\lp\sqrt{\pi} e^{y_i^2} y_i\erfc(-y_i)+1\rp-\beta}\nonumber \\
 & = & \frac{f^{(box)}_1(y_i;\alpha,\beta,\mu)}{-f^{(box)}_1(-y_i;\alpha,\beta,1-\mu)}.
\end{eqnarray}
A combination of (\ref{eq:boxdetanalIeer11j}), (\ref{eq:boxdetanalIeer11k}), and (\ref{eq:boxhdg1y}) gives
\begin{eqnarray}
& & \erf(y_e)  =
\frac{f^{(box)}_1(y_i;\alpha,\beta,\mu)+f^{(box)}_1(-y_i;\alpha,\beta,1-\mu)}{f^{(box)}_1(y_i;\alpha,\beta,\mu)-f^{(box)}_1(-y_i;\alpha,\beta,1-\mu)}\nonumber \\
&\Longleftrightarrow & y_e  =\erfinv\lp
\frac{f^{(box)}_1(y_i;\alpha,\beta,\mu)+f^{(box)}_1(-y_i;\alpha,\beta,1-\mu)}{f^{(box)}_1(y_i;\alpha,\beta,\mu)-f^{(box)}_1(-y_i;\alpha,\beta,1-\mu)}\rp.
\label{eq:boxhdg1z}
\end{eqnarray}
Connecting (\ref{eq:boxhdg1p}), (\ref{eq:boxhdg1s}), and ultimately (\ref{eq:boxdetanalIeer11h}) gives
\begin{eqnarray}\label{eq:boxhdg1za}
\gamma_g & = &
\frac{\alpha}{2}\lp -\sqrt{\pi} e^{y_e^2} y_e\erfc(-y_e)-1\rp\erfc(y_e)+\frac{\erfc(y_e)}{2} \nonumber \\
 & = & \frac{\alpha}{2}\lp\sqrt{\pi} e^{y_i^2} y_i\erfc(y_i)\erfc(-y_i)-\erfc(-y_i)\rp-(1-\beta)\mu+1-\frac{\erfc(y_i)\beta}{2}\nonumber \\
& = & \frac{-f^{(box)}_1(-y_i;\alpha,\beta,1-\mu)\lp\alpha\lp -\sqrt{\pi} e^{y_i^2} y_i\erfc(y_i)+1\rp-\beta\rp}{2e^{y_i^2} y_i}.
\end{eqnarray}
Combining further (\ref{eq:boxhdg1z}) and (\ref{eq:boxhdg1za}) one obtains
\begin{multline}\label{eq:boxhdg1zb}
\alpha\lp -\sqrt{\pi} e^{y_e^2} y_e\erfc(-y_e)-1\rp+1
 =  \frac{\lp f^{(box)}_1(y_i;\alpha,\beta,\mu)-f^{(box)}_1(-y_i;\alpha,\beta,1-\mu)\rp}{2e^{y_i^2} y_i} \\
 \times
\lp\alpha\lp -\sqrt{\pi} e^{y_i^2} y_i\erfc(y_i)+1\rp-\beta\rp,
\end{multline}
and
\begin{multline}\label{eq:boxhdg1zc}
e^{y_e^2} y_e\erfc(-y_e)
=\frac{\lp f^{(box)}_1(y_i;\alpha,\beta,\mu)-f^{(box)}_1(-y_i;\alpha,\beta,1-\mu)\rp \lp\alpha\lp -\sqrt{\pi} e^{y_i^2} y_i\erfc(y_i)+1\rp-\beta\rp
}{-\sqrt{\pi}2e^{y_i^2} y_i\alpha} \\
+
\frac{-2e^{y_i^2} y_i+2e^{y_i^2} y_i\alpha}{-\sqrt{\pi}2e^{y_i^2} y_i\alpha}.
\end{multline}
Similarly to what ws done in Section \ref{sec:binl1}, we will now argue that the right side in (\ref{eq:boxhdg1zc}) is equal to $f^{(box)}_1(y_i;\alpha,\beta,\mu)$. That will follow if
\begin{multline}\label{eq:boxhdg1zd}
\lp f^{(box)}_1(y_i;\alpha,\beta,\mu)-f^{(box)}_1(-y_i;\alpha,\beta,1-\mu)\rp \lp\alpha\lp -\sqrt{\pi} e^{y_i^2} y_i\erfc(y_i)+1\rp-\beta\rp-2e^{y_i^2} y_i+2e^{y_i^2} y_i\alpha\\
=-\sqrt{\pi}2e^{y_i^2} y_i\alpha f^{(box)}_1(y_i;\alpha,\beta,\mu),
\end{multline}
or
\begin{multline}\label{eq:boxhdg1ze}
-f^{(box)}_1(-y_i;\alpha,\beta,1-\mu) \lp\alpha\lp -\sqrt{\pi} e^{y_i^2} y_i\erfc(y_i)+1\rp-\beta\rp-2e^{y_i^2} y_i+2e^{y_i^2} y_i\alpha \\
=
 f^{(box)}_1(y_i;\alpha,\beta,\mu) \lp\alpha\lp -\sqrt{\pi} e^{y_i^2} y_i\erfc(-y_i)-1\rp +\beta\rp.
\end{multline}
Recalling again on (\ref{eq:boxdetanalIeer11h}) and (\ref{eq:boxdetanalIeer11i}), (\ref{eq:boxhdg1ze}) can be rewritten in the following way
\begin{multline}\label{eq:boxhdg1zf}
2(1-\mu)(1-\beta)-\lp\alpha\lp -\sqrt{\pi} e^{y_i^2} y_i\erfc(y_i)+1\rp-\beta\rp\erfc(-y_i)-2+2\alpha \\
=
-\lp 2\mu(1-\beta)
-\lp\alpha\lp\sqrt{\pi}y_ie^{y_i^2}\erfc(-y_i)+1\rp-\beta\rp\erfc(y_i)\rp.
\end{multline}
Similarly to (\ref{eq:hdg1zf}), (\ref{eq:boxhdg1zf}) indeed holds and one then has that (\ref{eq:boxhdg1ze}) holds as well. Through (\ref{eq:boxhdg1zd}) we have then that the right side of (\ref{eq:boxhdg1zc}) is indeed equal to $f^{(box)}_1(y_i;\alpha,\beta,\mu)$. One additional combination of (\ref{eq:boxhdg1z}) and (\ref{eq:boxhdg1zc}) brings us finally to
\begin{eqnarray}\label{eq:boxhdg1zg}
& & e^{y_e^2} y_e\erfc(-y_e)
=f^{(box)}_1(y_i;\alpha,\beta,\mu)\nonumber \\
& \Longleftrightarrow &  2\erfinv\lp
\frac{f^{(box)}_1(y_i;\alpha,\beta,\mu)+f^{(box)}_1(-y_i;\alpha,\beta,1-\mu)}{f^{(box)}_1(y_i;\alpha,\beta,\mu)-f^{(box)}_1(-y_i;\alpha,\beta,1-\mu)}\rp
\frac{e^{\erfinv\lp
\frac{f^{(box)}_1(y_i;\alpha,\beta,\mu)+f^{(box)}_1(-y_i;\alpha,\beta,1-\mu)}{f^{(box)}_1(y_i;\alpha,\beta,\mu)-f^{(box)}_1(-y_i;\alpha,\beta,1-\mu)}\rp
^2}}{f^{(box)}_1(y_i;\alpha,\beta,\mu)-f^{(box)}_1(-y_i;\alpha,\beta,1-\mu)}
=1.\nonumber \\
\end{eqnarray}
(\ref{eq:boxhdg1zg}) is sufficient to compute $y_i$. One can then utilize such $y_i$ and from (\ref{eq:boxhdg1z}) obtain $y_e$ and from (\ref{eq:boxhdg1p}) or (\ref{eq:boxhdg1s}) $\gamma_g$ (of course, a quick comparison of (\ref{eq:boxhdg1zg}) and (\ref{eq:boxhdg1z}) on the one side and (\ref{eq:boxdetanalIeer11l}) and (\ref{eq:boxdetanalIeer11k}) on the other side gives $y_i=y_2$ and $y_e=y_1$). After these values for $y_i$, $y_e$, and $\gamma_g$ are determined one can then use them to determine the optimal value of $\zeta^{(box)}_{\alpha,\beta}(\gamma_g,y_e,y_i)$ in (\ref{eq:boxhdg1a}) through (\ref{eq:boxhdg1b}). That eventually gives $I^{(box)}_{err}(\alpha,\beta)$ in (\ref{eq:boxhdg1}). Following Section \ref{sec:binl1} though, we will below try to provide a bit more explicit connection between the optimal $\zeta^{(box)}_{\alpha,\beta}(\gamma_g,y_e,y_i)$ and what we presented in Section \ref{sec:ldp}.

%%%%%%%%%%%%%%%%%%%%%%%%%%%%%%%%%%%%%%%%%%%%%%%%%%%%%%%%%%%%%%%%%
\subsubsection{Computing $I^{(box)}_{err}(\alpha,\beta)$}
\label{sec:boxhdgcompI}
%%%%%%%%%%%%%%%%%%%%%%%%%%%%%%%%%%%%%%%%%%%%%%%%%%%%%%%%%%%%%%%%%

As in Section \ref{sec:boxhdgcompI} when we discussed $\zeta^{(box)}_{\alpha,\beta}(\gamma_g,y_e,y_i)$, we here recognize that instead of working directly with $\zeta^{(box)}_{\alpha,\beta}(\gamma_g,y_e,y_i)$ it will be a bit easier to first deal separately with each of $\psicom^{(box)}$, $\psiint^{(box)}$, and $\psiext^{(box)}$. From this point on we assume that $y_i$, $y_e$, and $\gamma_g$ take the values determined through the procedure explained above. We start with $\psicom^{(box)}$ and write
\begin{eqnarray}\label{eq:boxhdg1zh}
  \psicom^{(box)} &=& -\mu(1-\beta)H\lp \frac{1-\alpha-\gamma_g}{\mu(1-\beta)} \rp
-(1-\mu)(1-\beta)H\lp \frac{(1-\mu)(1-\beta)-\gamma_g}{(1-\mu)(1-\beta)}\rp \nonumber \\
&=& -(1-\alpha-\gamma_g)\log\lp \frac{1-\alpha-\gamma_g}{\mu(1-\beta)}\rp
-(\mu(1-\beta)-(1-\alpha-\gamma_g))\log\lp\frac{(\mu(1-\beta)-(1-\alpha-\gamma_g))}{\mu(1-\beta)}\rp\nonumber \\
&& -((1-\mu)(1-\beta)-\gamma_g) \log\lp \frac{(1-\mu)(1-\beta)-\gamma_g}{(1-\mu)(1-\beta)} \rp-\gamma_g \log\lp \frac{\gamma_g}{(1-\mu)(1-\beta)} \rp \nonumber \\
&=& -\frac{1}{2}\lp 2\mu(1-\beta)-\lp\alpha\lp\sqrt{\pi} e^{y_i^2} y_i\erfc(-y_i)+1\rp-\beta\rp\erfc(y_i)\rp\nonumber \\
& &  \times \log\lp \frac{\lp2\mu(1-\beta)-\lp\alpha\lp\sqrt{\pi} e^{y_i^2} y_i\erfc(-y_i)+1\rp-\beta\rp\erfc(y_i)\rp}{2\mu(1-\beta)}\rp \nonumber \\
&& -\frac{1}{2}\lp\lp\alpha\lp\sqrt{\pi} e^{y_i^2} y_i\erfc(-y_i)+1\rp-\beta\rp\erfc(y_i)\rp \nonumber \\
& & \times  \log\lp\frac{\lp\alpha\lp\sqrt{\pi} e^{y_i^2} y_i\erfc(-y_i)+1\rp-\beta\rp\erfc(y_i)}{2\mu(1-\beta)}\rp\nonumber \\
&& -\frac{1}{2}\lp\lp\alpha\lp -\sqrt{\pi} e^{y_i^2} y_i\erfc(y_i)+1\rp-\beta\rp\erfc(-y_i)\rp \nonumber \\
& & \times  \log\lp\frac{\lp\alpha\lp -\sqrt{\pi} e^{y_i^2} y_i\erfc(y_i)+1\rp-\beta\rp\erfc(-y_i)}{2(1-\mu)(1-\beta)}\rp\nonumber \\
&& -\frac{1}{2}\lp 2(1-\mu)(1-\beta)-\lp\alpha\lp -\sqrt{\pi} e^{y_i^2} y_i\erfc(y_i)+1\rp-\beta\rp\erfc(-y_i)\rp\nonumber \\
& &  \times \log\lp \frac{\lp 2(1-\mu)(1-\beta)-\lp\alpha\lp -\sqrt{\pi} e^{y_i^2} y_i\erfc(y_i)+1\rp-\beta\rp\erfc(-y_i)\rp}{2(1-\mu)(1-\beta)}\rp. \nonumber \\
\end{eqnarray}
For $\psiint^{(box)}$ we in a similar fashion have
\begin{eqnarray}\label{eq:boxhdg1zi}
  \psiint^{(box)} &=& \alpha y_i^2 +((1-\beta)\mu-(1-\alpha-\gamma_g))\log(\erfc(y_i))+((1-\beta)(1-\mu)-\gamma_g)\log(\erfc(-y_i)))\nonumber \\
& & - (\alpha-\beta) \log(2)\nonumber\\
&=& \alpha y_i^2 +\frac{1}{2}\lp\lp\alpha\lp\sqrt{\pi} e^{y_i^2} y_i\erfc(-y_i)+1\rp-\beta\rp\erfc(y_i)\rp\log(\erfc(y_i))\nonumber \\
& & +\frac{1}{2}\lp\lp\alpha\lp -\sqrt{\pi} e^{y_i^2} y_i\erfc(y_i)+1\rp-\beta\rp\erfc(-y_i)\rp\log(\erfc(-y_i))- (\alpha-\beta) \log(2). \nonumber \\
\end{eqnarray}
Finally for $\psiext^{(box)}$ we obtain
\begin{eqnarray}\label{eq:boxhdg1zj}
  \psiext^{(box)} &=& -\alpha y_e^2 +(1-\alpha-\gamma_g)\log(\erfc(-y_e))+\gamma_g\log(\erfc(y_e))-(1-\alpha)\log(2) \nonumber \\
&=& -\alpha y_e^2 +\frac{1}{2}\lp 2\mu(1-\beta)-\lp\alpha\lp\sqrt{\pi} e^{y_i^2} y_i\erfc(-y_i)+1\rp-\beta\rp\erfc(y_i)\rp
\log(\erfc(-y_e))\nonumber \\
& & +\frac{1}{2}\lp 2(1-\mu)(1-\beta)+\lp\alpha\lp -\sqrt{\pi} e^{y_i^2} y_i\erfc(y_i)+1\rp-\beta\rp\erfc(-y_i)\rp\log(\erfc(y_e))-(1-\alpha)\log(2). \nonumber \\
\end{eqnarray}
Combining (\ref{eq:boxhdg1zh}), (\ref{eq:boxhdg1zi}), and (\ref{eq:boxhdg1zj}) we also have
\begin{eqnarray}\label{eq:boxhdg1zk}
  \psinet^{(box)}
&=& -\frac{1}{2}\lp 2\mu(1-\beta)-\lp\alpha\lp\sqrt{\pi} e^{y_i^2} y_i\erfc(-y_i)+1\rp-\beta\rp\erfc(y_i)\rp\nonumber \\
& &  \times \log\lp \frac{\lp2\mu(1-\beta)-\lp\alpha\lp\sqrt{\pi} e^{y_i^2} y_i\erfc(-y_i)+1\rp-\beta\rp\erfc(y_i)\rp}{2\mu(1-\beta)\erfc(-y_e)}\rp \nonumber \\
&& -\frac{1}{2}\lp\lp\alpha\lp\sqrt{\pi} e^{y_i^2} y_i\erfc(-y_i)+1\rp-\beta\rp\erfc(y_i)\rp \nonumber \\
& &  \times
 \log\lp\frac{\lp\alpha\lp\sqrt{\pi} e^{y_i^2} y_i\erfc(-y_i)+1\rp-\beta\rp\erfc(y_i)}{2\mu(1-\beta)\erfc(y_i)}\rp\nonumber \\
&& -\frac{1}{2}\lp\lp\alpha\lp -\sqrt{\pi} e^{y_i^2} y_i\erfc(y_i)+1\rp-\beta\rp\erfc(-y_i)\rp \nonumber \\
& &  \times
 \log\lp\frac{\lp\alpha\lp -\sqrt{\pi} e^{y_i^2} y_i\erfc(y_i)+1\rp-\beta\rp\erfc(-y_i)}{2(1-\mu)(1-\beta)\erfc(-y_i)}\rp\nonumber \\
&& -\frac{1}{2}\lp 2(1-\mu)(1-\beta)-\lp\alpha\lp -\sqrt{\pi} e^{y_i^2} y_i\erfc(y_i)+1\rp-\beta\rp\erfc(-y_i)\rp\nonumber \\
& &  \times \log\lp \frac{\lp 2(1-\mu)(1-\beta)-\lp\alpha\lp -\sqrt{\pi} e^{y_i^2} y_i\erfc(y_i)+1\rp-\beta\rp\erfc(-y_i)\rp}{2(1-\mu)(1-\beta)\erfc(y_e)}\rp. \nonumber \\
&& +\alpha y_i^2-\alpha y_e^2-(1-\beta)\log(2).
\end{eqnarray}
Recalling once again (\ref{eq:boxdetanalIeer11h}) and (\ref{eq:boxdetanalIeer11i}), (\ref{eq:boxhdg1zk}) can be further transformed
\begin{eqnarray}\label{eq:boxhdg1zl}
  \psinet^{(box)}
&=& -\frac{1}{2}\lp 2\mu(1-\beta)-\lp\alpha\lp\sqrt{\pi} e^{y_i^2} y_i\erfc(-y_i)+1\rp-\beta\rp\erfc(y_i)\rp \nonumber \\
& & \times \log\lp \frac{ y_ee^{y_e^2}\erfc(-y_e) \lp \alpha\lp\sqrt{\pi}y_ie^{y_i^2}\erfc(-y_i)+1\rp-\beta\rp}{2\mu(1-\beta)\erfc(-y_e)y_ie^{y_i^2}}\rp \nonumber \\
&& -\frac{1}{2}\lp\lp\alpha\lp\sqrt{\pi} e^{y_i^2} y_i\erfc(-y_i)+1\rp-\beta\rp\rp
 \log\lp\frac{\lp\alpha\lp\sqrt{\pi} e^{y_i^2} y_i\erfc(-y_i)+1\rp-\beta\rp\erfc(y_i)}{2\mu(1-\beta)}\rp\nonumber \\
&& -\frac{1}{2}\lp\lp\alpha\lp -\sqrt{\pi} e^{y_i^2} y_i\erfc(y_i)+1\rp-\beta\rp\rp
 \log\lp\frac{\lp\alpha\lp -\sqrt{\pi} e^{y_i^2} y_i\erfc(y_i)+1\rp-\beta\rp\erfc(-y_i)}{2(1-\mu)(1-\beta)}\rp\nonumber \\
&& -\frac{1}{2}\lp 2(1-\mu)(1-\beta)-\lp\alpha\lp -\sqrt{\pi} e^{y_i^2} y_i\erfc(y_i)-1\rp-\beta\rp \erfc(-y_i)\rp \nonumber \\
& & \times \log\lp \frac{y_ee^{y_e^2}\erfc(y_e)\lp\alpha\lp -\sqrt{\pi}y_ie^{y_i^2}\erfc(y_i)+1\rp-\beta\rp}{2(1-\mu)(1-\beta)\erfc(y_e)y_ie^{y_i^2}} \rp \nonumber \\
&& +\alpha y_i^2-\alpha y_e^2-(1-\beta)\log(2).
\end{eqnarray}
Continuing further we also obtain
\begin{eqnarray}\label{eq:boxhdg1zm}
  \psinet^{(box)}
&=& -\frac{1}{2}\lp 2\mu(1-\beta)-\lp\alpha\lp\sqrt{\pi} e^{y_i^2} y_i\erfc(-y_i)+1\rp-\beta\rp\erfc(y_i)\rp\log\lp \frac{ y_ee^{y_e^2}}{y_ie^{y_i^2}}\rp \nonumber \\
&& -\frac{1}{2}\lp 2(1-\mu)(1-\beta)-\lp\alpha\lp -\sqrt{\pi} e^{y_i^2} y_i\erfc(y_i)+1\rp-\beta\rp\erfc(-y_i)\rp\log\lp \frac{y_ee^{y_e^2}}{y_ie^{y_i^2}} \rp \nonumber \\
&& -\frac{1}{2}\lp 2\mu(1-\beta)-\lp\alpha\lp\sqrt{\pi} e^{y_i^2} y_i\erfc(-y_i)+1\rp-\beta\rp\erfc(y_i)\rp \nonumber \\
& & \times
\log\lp \frac{ \alpha\lp\sqrt{\pi}y_ie^{y_i^2}\erfc(-y_i)+1\rp-\beta}{2\mu(1-\beta)}\rp \nonumber \\
&& -\frac{1}{2}\lp 2(1-\mu)(1-\beta)-\lp\alpha\lp -\sqrt{\pi} e^{y_i^2} y_i\erfc(y_i)+1\rp-\beta\rp\erfc(-y_i)\rp \nonumber \\
& & \times \log\lp \frac{\alpha\lp -\sqrt{\pi}y_ie^{y_i^2}\erfc(y_i)+1\rp-\beta}{2(1-\mu)(1-\beta)} \rp \nonumber \\
&& -\frac{1}{2}\lp\alpha\lp\sqrt{\pi} e^{y_i^2} y_i\erfc(-y_i)+1\rp-\beta\rp\erfc(y_i)\log\lp\frac{\lp\alpha\lp\sqrt{\pi} e^{y_i^2} y_i\erfc(-y_i)+1\rp-\beta\rp}{2\mu(1-\beta)}\rp\nonumber \\
&& -\frac{1}{2}\lp \alpha \lp -\sqrt{\pi} e^{y_i^2} y_i\erfc(y_i)+1\rp-\beta\rp\erfc(-y_i) \log\lp \frac{\alpha\lp-\sqrt{\pi} e^{y_i^2} y_i\erfc(y_i)+1\rp-\beta}{2(1-\mu)(1-\beta)} \rp\nonumber \\
&& +\alpha y_i^2-\alpha y_e^2-(1-\beta)\log(2),
\end{eqnarray}
and
\begin{eqnarray}\label{eq:boxhdg1zn}
  \psinet^{(box)}
&=& -(1-\alpha)\log\lp \frac{ y_ee^{y_e^2}}{y_ie^{y_i^2}}\rp +\mu(1-\beta)\log\lp \frac{2\mu(1-\beta)}{ \alpha\lp\sqrt{\pi}y_ie^{y_i^2}\erfc(-y_i)+1\rp-\beta}\rp \nonumber \\
&& +(1-\mu)(1-\beta)\log\lp \frac{2(1-\mu)(1-\beta)}{\alpha\lp -\sqrt{\pi}y_ie^{y_i^2}\erfc(y_i)+1\rp-\beta}\rp  +\alpha y_i^2-\alpha y_e^2-(1-\beta)\log(2).
\end{eqnarray}
Finally, a couple of simple algebraic transformations give
\begin{eqnarray}\label{eq:boxhdg1zo}
  \psinet^{(box)}=I^{(box)}_{err}(\alpha,\beta)=
&=& (\alpha-1)\log\lp \frac{ y_e}{y_i}\rp +\mu(1-\beta)\log\lp \frac{1-\beta}{ \alpha\lp\sqrt{\pi}y_ie^{y_i^2}\erfc(-y_i)+1\rp-\beta}\rp \nonumber \\
&& +(1-\mu)(1-\beta)\log\lp \frac{(1-\mu)(1-\beta)}{\alpha\lp -\sqrt{\pi}y_ie^{y_i^2}\erfc(y_i)+1\rp-\beta}\rp  + y_i^2-y_e^2=I^{(box)}_{ldp}(\alpha,\beta).\nonumber \\
\end{eqnarray}
It is not that had to see from (\ref{eq:boxdetanalIeer11mb}) and (\ref{eq:boxhdg1zo}) that $I^{(box)}_{err}=I^{(box)}_{ldp}=I^{(bin,ub)}_{err,u}$ (of course $y_e\leftrightarrow y_1$ and $y_i\leftrightarrow y_2$) which ensures that the choice for $\nu$, $A_0$, $c_3$, and $\gamma$ made in (\ref{eq:boxdetanalIeer11m}) is indeed optimal. Moreover, in the lower tail regime ($\alpha<\alpha_w$, considerations from \cite{Stojnicl1BnBxfinn} ensure that one also has
\begin{equation}\label{eq:boxhdg1aa}
  \Psi^{(box)}_{net}(\alpha,\beta)=I^{(box)}_{cor}(\alpha,\beta)\triangleq\lim_{n\rightarrow\infty}
  \frac{\log{P^{(box)}_{cor}}}{n}=\psicom^{(box)}+\psiint^{(box)}+\psiext^{(box)},
\end{equation}
where $\psicom^{(box)}$, $\psiint^{(box)}$, and $\psiext^{(box)}$ are as in (\ref{eq:boxhdg2}). As was the case for binary $\ell_1$ in Section \ref{sec:binl1}, what we presented above then automatically characterizes the box $\ell_1$'s LDP. The following theorem summarizes what we presented above.
\begin{theorem}[Box $\ell_1$'s LDP]
Assume the setup of Theorem \ref{thm:boxldp3} and assume that a pair $(\alpha,\beta)$ is given. Let $P^{(box)}_{err}$ be the probability that the solutions of (\ref{eq:l0}) and (\ref{eq:boxl1bin}) coincide and let $P_{cor}$ be the probability that the solutions of (\ref{eq:l0}) and (\ref{eq:boxl1bin}) do \emph{not} coincide. Set
\begin{equation}\label{eq:boxthmfinalldpl11}
 f^{(box)}_1(y_2;\alpha,\beta,\mu) \triangleq
\lp
\frac{2\mu(1-\beta)y_2e^{y_2^2}}{\alpha\lp\sqrt{\pi}y_2e^{y_2^2}\erfc(-y_2)+1\rp-\beta}\rp-y_2e^{y_2^2}\erfc(y_2).
\end{equation}
Also let $y_2$ and $y_1$ satisfy the following \textbf{fundamental characterizations of the box $\ell_1$'s LDP} and achieve the optimum in (\ref{eq:boxhdg1a}):

%\begin{center}
%\shadowbox{$
\begin{eqnarray}
2\erfinv\lp
\frac{f^{(box)}_1(y_2;\alpha,\beta,\mu)+f^{(box)}_1(-y_2;\alpha,\beta,1-\mu)}{f^{(box)}_1(y_2;\alpha,\beta,\mu)-f^{(box)}_1(-y_2;\alpha,\beta,1-\mu)}\rp
\frac{e^{\erfinv\lp
\frac{f^{(box)}_1(y_2;\alpha,\beta,\mu)+f^{(box)}_1(-y_2;\alpha,\beta,1-\mu)}{f^{(box)}_1(y_2;\alpha,\beta,\mu)-f^{(box)}_1(-y_2;\alpha,\beta,1-\mu)}\rp^2}}
{f^{(box)}_1(y_2;\alpha,\beta,\mu)-f^{(box)}_1(-y_2;\alpha,\beta,1-\mu)}
& = & 1.\nonumber \\
\label{eq:boxthmfinalldpl12a}
\end{eqnarray}
and
\begin{eqnarray}
y_1  =  \erfinv\lp
\frac{f^{(box)}_1(y_2;\alpha,\beta,\mu)+f^{(box)}_1(-y_2;\alpha,\beta,1-\mu)}{f^{(box)}_1(y_2;\alpha,\beta,\mu)-f^{(box)}_1(-y_2;\alpha,\beta,1-\mu)}\rp.\nonumber \\
\label{eq:boxthmfinalldpl12b}
\end{eqnarray}
%$}
%-\vspace{-.5in}\begin{equation}
%\label{eq:boxthmldp3l1PT}
%\end{equation}
%\end{center}

\noindent Finally, let $I^{(box)}_{ldp}(\alpha,\beta)$ be defined through the following \textbf{box $\ell_1$'s fundamental LDP rate function} characterization
\begin{eqnarray}
I^{(box)}_{ldp}(\alpha,\beta)&\triangleq & (\alpha-1)\log\lp\frac{y_1}{y_2}\rp+\mu(1-\beta)\log \lp \frac{1-\beta}{\alpha\lp\sqrt{\pi}e^{y_2^2}y_2\erfc(-y_2)+1\rp-\beta}\rp\nonumber \\
& & +(1-\mu)(1-\beta)\log\lp \frac{(1-\mu)(1-\beta)}{\alpha\lp -\sqrt{\pi}e^{y_2^2}y_2\erfc(y_2)+1\rp-\beta}\rp + y_2^2-y_1^2.
\label{eq:boxthmfinalldpl13}
\end{eqnarray}
Then if $\alpha>\alpha_w$
\begin{equation}
I^{(box)}_{err}(\alpha,\beta)\triangleq\lim_{n\rightarrow\infty}\frac{\log{P^{(box)}_{err}}}{n}=I^{(box)}_{ldp}(\alpha,\beta).\label{eq:boxthmfinalldpl14}
\end{equation}
Moreover, if $\alpha<\alpha_w$
\begin{equation}
I^{(box)}_{cor}(\alpha,\beta)\triangleq\lim_{n\rightarrow\infty}\frac{\log{P^{(box)}_{cor}}}{n}=I^{(box)}_{ldp}(\alpha,\beta).\label{eq:boxthmfinalldpl15}
\end{equation}\label{thm:boxfinalldpl1}
\end{theorem}
\begin{proof} Follows from the above discussion.
\end{proof}

%%%%%%%%%%%%%%%%%%%%%%%%%%%%%%%%%%%%%%%%%%%%%%%%%%%%%%%%%%%%%%%%%
\subsubsection{Phase transitions}
\label{sec:boxrephasetrans}
%%%%%%%%%%%%%%%%%%%%%%%%%%%%%%%%%%%%%%%%%%%%%%%%%%%%%%%%%%%%%%%%%

In this section we will show how one can quickly determine the phase transitions for the box $\ell_1$ utilizing Theorem \ref{thm:boxfinalldpl1} and the above considerations leading up to Theorem \ref{thm:boxfinalldpl1}. We start by closely following what we presented in Section \ref{sec:rephasetrans}, and focus on those pairs $(\alpha,\beta)$  for which $y_1=y_2$ in Theorem \ref{thm:boxfinalldpl1} (as mentioned in Section \ref{sec:rephasetrans}, we may not know a priori if such pairs do exist; however, the derivation below will confirm that such an assumption is actually correct). From (\ref{eq:boxdetanalIeer11h}) and (\ref{eq:boxdetanalIeer11i}) we have
\begin{eqnarray}\label{eq:boxrephasetran1}
y_2e^{y_2^2}\erfc(-y_2) & = & y_1e^{y_1^2}\erfc(-y_1)=
\lp
\frac{2\mu(1-\beta)y_2e^{y_2^2}}{\alpha\lp\sqrt{\pi}y_2e^{y_2^2}\erfc(-y_2)+1\rp-\beta}\rp-y_2e^{y_2^2}\erfc(y_2)=f^{(box)}_1(y_2;\alpha,\beta,\mu)\nonumber \\
 y_2e^{y_2^2}\erfc(y_2) & = & y_1e^{y_1^2}\erfc(y_1) =
\lp
\frac{2(1-\mu)(1-\beta)y_2e^{y_2^2}}{\alpha\lp -\sqrt{\pi}y_2e^{y_2^2}\erfc(y_2)+1\rp-\beta}\rp-y_2e^{y_2^2}\erfc(-y_2)\nonumber \\
& = & -f^{(box)}_1(-y_2;\alpha,\beta,1-\mu).\nonumber \\
\end{eqnarray}
Transforming (\ref{eq:boxrephasetran1}) a bit further we arrive at the following analogue of (\ref{eq:rephasetran2})
\begin{eqnarray}\label{eq:boxrephasetran2}
1=
\lp
\frac{\mu(1-\beta)}{\alpha\lp\sqrt{\pi}y_2e^{y_2^2}\erfc(-y_2)+1\rp-\beta}\rp\nonumber \\
1 =
\lp
\frac{(1-\mu)(1-\beta)}{\alpha\lp -\sqrt{\pi}y_2e^{y_2^2}\erfc(y_2)+1\rp-\beta}\rp.
\end{eqnarray}
From (\ref{eq:boxrephasetran2}) we then easily have
\begin{eqnarray}\label{eq:boxrephasetran3}
\sqrt{\pi}y_2e^{y_2^2}\erfc(-y_2) & = & \frac{\mu(1-\beta)+\beta}{\alpha}-1
\nonumber \\
\sqrt{\pi} y_2e^{y_2^2}\erfc(y_2) & = & 1+\frac{-\beta-(1-\mu)(1-\beta)}{\alpha},
\end{eqnarray}
and
\begin{eqnarray}\label{eq:boxrephasetran4}
\frac{1+\erf(y_2)}{1-\erf(y_2)}=\frac{\erfc(-y_2)}{\erfc(y_2)}=\frac{\mu(1-\beta)+\beta-\alpha}{\alpha-\beta-(1-\mu)(1-\beta)}.
\end{eqnarray}
From (\ref{eq:boxrephasetran4}) we then determine $\erf(y_2)$ and $y_2$
\begin{eqnarray}\label{eq:boxrephasetran5}
& & \erf(y_2)=\frac{\frac{\mu(1-\beta)+\beta-\alpha}{\alpha-\beta-(1-\mu)(1-\beta)}-1}{\frac{\mu(1-\beta)+\beta-\alpha}{\alpha-\beta-(1-\mu)(1-\beta)}+1}
=\frac{1+\beta-2\alpha}{(2\mu-1)(1-\beta)}\nonumber \\
\Longleftrightarrow & & y_2=\erfinv\lp\frac{1+\beta-2\alpha}{(2\mu-1)(1-\beta)}\rp.
\end{eqnarray}
A combination of (\ref{eq:boxrephasetran3}) and (\ref{eq:boxrephasetran5}) then gives
\begin{eqnarray}\label{eq:boxrephasetran6}
\sqrt{\pi}y_2e^{y_2^2}\erfc(-y_2) & = & \sqrt{\pi}y_2e^{y_2^2}(1+\erf(y_2))=\sqrt{\pi}y_2e^{y_2^2}\frac{2\mu(1-\beta)+2\beta-2\alpha}{(2\mu-1)(1-\beta)}=
\frac{\mu(1-\beta)+\beta-\alpha}{\alpha}.\nonumber \\
\end{eqnarray}
Finally combining (\ref{eq:boxrephasetran5}) and (\ref{eq:boxrephasetran6}) we obtain
\begin{eqnarray}\label{eq:boxrephasetran7}
\frac{(2\mu-1)(1-\beta))e^{-y_2^2}}{2\sqrt{\pi}\alpha y_2}=\frac{(2\mu-1)(1-\beta)
e^{-\lp\erfinv\lp\frac{1+\beta-2\alpha}{(2\mu-1)(1-\beta)}\rp\rp^2}}{2\sqrt{\pi}\alpha \erfinv\lp\frac{1+\beta-2\alpha}{(2\mu-1)(1-\beta)}\rp}=1.
\end{eqnarray}
(\ref{eq:boxrephasetran7}) effectively establishes the phase transition curve. Also, from (\ref{eq:boxrephasetran1}) we have
\begin{eqnarray}\label{eq:boxrephasetran8}
f_1(y_2;\alpha,\beta)-f^{(box)}_1(-y_2;\alpha,\beta,1-\mu)=
y_2e^{y_2^2}\erfc(-y_2)+ y_2e^{y_2^2}\erfc(y_2)=2y_2e^{y_2^2},
\end{eqnarray}
and
\begin{eqnarray}\label{eq:boxrephasetran9}
f_1(y_2;\alpha,\beta)+f_1(-y_2;\alpha,1-\beta)=
y_2e^{y_2^2}\erfc(-y_2)- y_2e^{y_2^2}\erfc(y_2)=2y_2e^{y_2^2}\erf(y_2).
\end{eqnarray}
Assuming that the pair $(\alpha,\beta)$ is such that (\ref{eq:boxrephasetran7}) holds one then also has $y_1=y_2$ and (\ref{eq:boxrephasetran8}) and (\ref{eq:boxrephasetran9}) ensure that
(\ref{eq:boxthmfinalldpl12a}) and (\ref{eq:boxthmfinalldpl12b}) hold and that
$I^{(box)}_{ldp}(\alpha,\beta)=0$ in (\ref{eq:boxthmfinalldpl13}) which is exactly the value that $I^{(box)}_{ldp}(\alpha,\beta)$ takes at the phase transition. The following theorem summarizes the above discussion and the obtained PT characterization. It is essentially a box $\ell_1$ analogue to Theorem \ref{thm:thmweakthr} which characterizes the binary $\ell_1$'s PT.
\begin{theorem}(Exact box $\ell_1$'s weak threshold/PT)
Let $A$ be an $m\times n$ matrix in (\ref{eq:l0})
with i.i.d. standard normal components. Let $\mu$ be a real number such that $\mu\in [\frac{1}{2},1]$. Further, let
the unknown $\x$ in (\ref{eq:l0}) be box-constrained $k$-sparse and let the locations of the elements of $\x$ from $(0,1)$ be arbitrarily chosen but fixed. Let $k,m,n$ be large
and let $\alpha_w=\frac{m}{n}$ and $\beta_w=\frac{k}{n}$ be constants
independent of $m$ and $n$. Further, let $\alpha_w$ and $\beta_w$ satisfy the following \textbf{fundamental characterization of the \emph{box} $\ell_1$'s PT}

\begin{center}
\shadowbox{$
%\begin{equation}
\xi^{(box)}_{\alpha_{w},\mu}(\beta_w)\triangleq\psi^{(box)}_{\beta_w,\mu}(\alpha_{w})\triangleq
\frac{(2\mu-1)(1-\beta_w)\sqrt{\frac{1}{2\pi}}e^{-\lp\erfinv\lp\frac{1+\beta_w-2\alpha_w}{(2\mu-1)(1-\beta_w)}\rp\rp^2}}{\alpha_w\sqrt{2}\erfinv \lp\frac{1+\beta_w-2\alpha_w}{(2\mu-1)(1-\beta_w)}\rp}=1.
%\end{equation}
$}
-\vspace{-.5in}\begin{equation}
\label{eq:boxthmweaktheta2}
\end{equation}
\end{center}

Then:
\begin{enumerate}
\item If $\alpha>\alpha_w$ then with overwhelming probability the solution of (\ref{eq:boxl1bin}) is the box-constrained $k$-sparse $\x$ that solves (\ref{eq:l0}).
\item If $\alpha<\alpha_w$ then with overwhelming probability the box-constrained $k$-sparse $\x$ with given fixed locations of nonzero components is the solution of (\ref{eq:l0}) and is \textbf{not} the solution of (\ref{eq:boxl1bin}).
    \end{enumerate}
\label{thm:boxthmweakthr}
\end{theorem}
\begin{proof}
Follows from the above discussion.
\end{proof}
We do mention that the above way of deriving the PT curve is presented as a box $\ell_1$ analogue to what was done in Section \ref{sec:rephasetrans}. If one is interested solely in the phase transition and not necessarily in the LDPs the box $\ell_1$ PT curve can be quickly derived using the methodology that we introduced in \cite{StojnicCSetam09,StojnicCSetamBlock09,StojnicISIT2010binary,StojnicICASSP10knownsupp}. Namely, looking at (\ref{eq:boxldpwhSw}) one would simply need to determine for any $\alpha$ a $\beta\in(0,\alpha)$ such that
\begin{equation}
\lim_{n\rightarrow\infty}\frac{E w(\h,\Sw^{(box)})}{\sqrt{n}}
 =  \min_{\nu\geq0}\sqrt{E\lp\mu(1-\beta)\max(\h_i-\nu,0)^2+(1-\mu)(1-\beta)\max(\h_i+\nu,0)^2+\beta(\h_i+\nu)^2\rp}
 =\sqrt{\alpha}.\\
\label{eq:rephaseboxldpwhSw}
\end{equation}
Following the standards that we set in \cite{StojnicCSetam09,StojnicCSetamBlock09,StojnicISIT2010binary,StojnicICASSP10knownsupp} this is now a fairly routine task and we leave it as an exercise to confirm that one indeed obtains that $\beta$ satisfies the fundamental box $\ell_1$ PT from (\ref{eq:boxthmweaktheta2}).

The fundamental PT characterizations given in the above theorem are indeed well defined. Namely, for any given fixed $\alpha\in (0,1)$ there will be a unique $\beta\in(0,\alpha)$ such that $\xi^{(box)}_{\alpha,\mu}(\beta)=1$ and for any given fixed $\beta\in (0,1)$ there will be a unique $\alpha\in(\beta,1)$ such that $\psi^{(box)}_{\beta,\mu}(\alpha)=1$. The arguments are similar to the ones that we presented in \cite{Stojnicl1RegPosasymldp}. We leave their adaptation to the box $\ell_1$ PTs given in the above theorem as an easy exercise as well.
Finally, in Figure \ref{fig:weakl1PTbox} we show the theoretical PT curves for the box $\ell_1$ that one can obtain based on (\ref{eq:boxthmweaktheta2}).
\begin{figure}[htb]
%\begin{minipage}[b]{.5\linewidth}
\centering
\centerline{\epsfig{figure=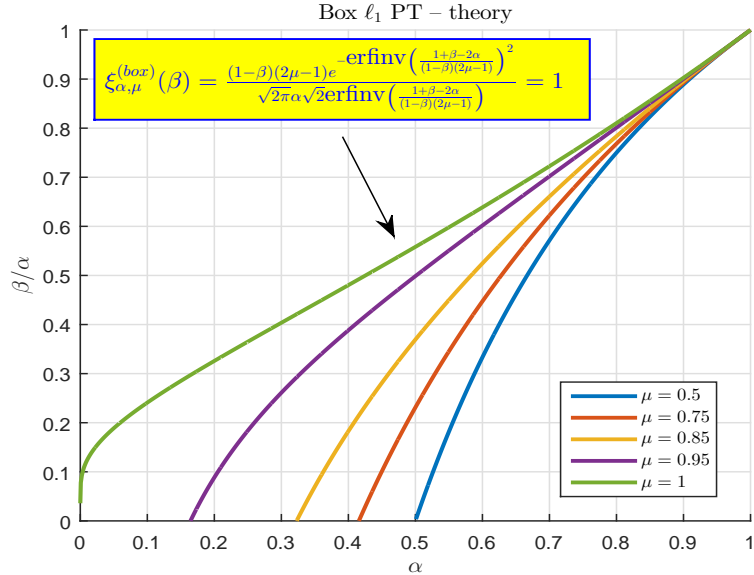,width=11.5cm,height=8cm}}
%\end{minipage}
%\begin{minipage}[b]{.5\linewidth}
%\centering
%\centerline{\epsfig{figure=finprerral08.eps,width=9cm,height=6.5cm}}
%\end{minipage}
\caption{Box $\ell_1$'s weak PT; $\{(\alpha,\beta)|\xi^{(box)}_{\alpha,\mu}(\beta)=1\}$}
\label{fig:weakl1PTbox}
\end{figure}

%\begin{figure}[htb]
%\begin{minipage}[b]{.5\linewidth}
%\centering
%\centerline{\epsfig{figure=CSetamBlockWeak.eps,width=7.5cm,height=7cm}}
%%\end{minipage}
%%\begin{minipage}[b]{.5\linewidth}
%%\centering
%%\centerline{\epsfig{figure=finprerral08.eps,width=9cm,height=6.5cm}}
%\end{minipage}
%\begin{minipage}[b]{.5\linewidth}
%\centering
%\centerline{\epsfig{figure=SimulBlSpWeakd151.eps,width=7.5cm,height=7cm}}
%%\end{minipage}
%%\begin{minipage}[b]{.5\linewidth}
%%\centering
%%\centerline{\epsfig{figure=finprerral08.eps,width=9cm,height=6.5cm}}
%\end{minipage}
%\caption{\emph{Weak} threshold, $\ell_2/\ell_1$-optimization; theory -- left, simulations -- right}
%\label{fig:weak}
%\end{figure}

%%%%%%%%%%%%%%%%%%%%%%%%%%%%%%%%%%%%%%%%%%%%%%%%%%%%%%%%%%%%%%%%%
\subsection{Theoretical and numerical LDP results}
\label{sec:boxthnumresutsnonn}
%%%%%%%%%%%%%%%%%%%%%%%%%%%%%%%%%%%%%%%%%%%%%%%%%%%%%%%%%%%%%%%%%

In this section we finally give a little bit of a flavor as to what is actually proven in Theorem \ref{thm:boxfinalldpl1}. These results are essentially box $\ell_1$ analogues to the results presented in Section \ref{sec:thnumresuts}. Consequently, in presentation of the results, we try to maintain as much of a parallelism with Section \ref{sec:thnumresuts} as possible. In Figure \ref{fig:l1regldpIerrubnonn} we show the theoretical LDP rate function curve that one can obtain based on Theorem \ref{thm:boxfinalldpl1}. This figure is complemented by Table \ref{tab:Ildptab1nonn} where we show the numerical values for all quantities of interest in Theorems \ref{thm:boxldp3} and \ref{thm:boxfinalldpl1} for several $\alpha$'s from the transition zone (i.e. for several $\alpha$'s around the breaking point; here $\beta=0.18469$ is chosen such that the breaking point/threshold for $\alpha=0.5$ and $\mu=0.85$). Finally, in Figure \ref{fig:weakl1LDPthrsimnonn} and Table \ref{tab:Ildptab2nonn} we show the comparison between the simulated values and the theoretical ones. As was the case for the binary $\ell_1$ in Section \ref{sec:thnumresuts}, here we again observe that even for fairly small dimensions (of order $100$) one already approaches the theoretical curves derived assuming an infinite dimensional asymptotic regime.

%\begin{figure}[htb]
%\begin{minipage}[b]{.5\linewidth}
%\centering
%\centerline{\epsfig{figure=CSetamBlockWeak.eps,width=7.5cm,height=7cm}}
%%\end{minipage}
%%\begin{minipage}[b]{.5\linewidth}
%%\centering
%%\centerline{\epsfig{figure=finprerral08.eps,width=9cm,height=6.5cm}}
%\end{minipage}
%\begin{minipage}[b]{.5\linewidth}
%\centering
%\centerline{\epsfig{figure=SimulBlSpWeakd151.eps,width=7.5cm,height=7cm}}
%%\end{minipage}
%%\begin{minipage}[b]{.5\linewidth}
%%\centering
%%\centerline{\epsfig{figure=finprerral08.eps,width=9cm,height=6.5cm}}
%\end{minipage}
%\caption{\emph{Weak} threshold, $\ell_2/\ell_1$-optimization; theory -- left, simulations -- right}
%\label{fig:weak}
%\end{figure}

\begin{figure}[htb]
\begin{minipage}[b]{.5\linewidth}
\centering
\centerline{\epsfig{figure=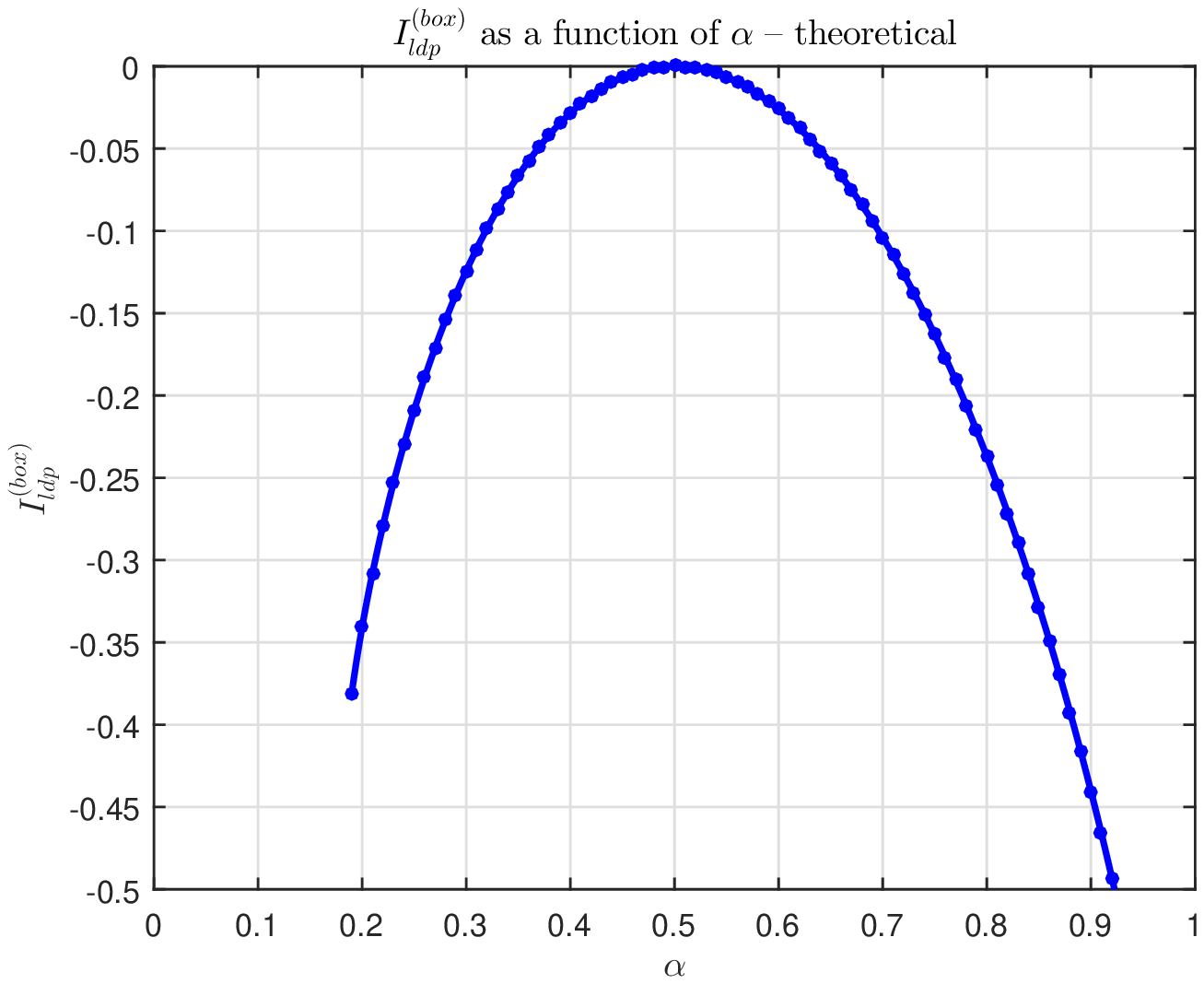,width=9cm,height=7cm}}
%\end{minipage}
%\begin{minipage}[b]{.5\linewidth}
%\centering
%\centerline{\epsfig{figure=finprerral08.eps,width=9cm,height=6.5cm}}
\end{minipage}
\begin{minipage}[b]{.5\linewidth}
\centering
\centerline{\epsfig{figure=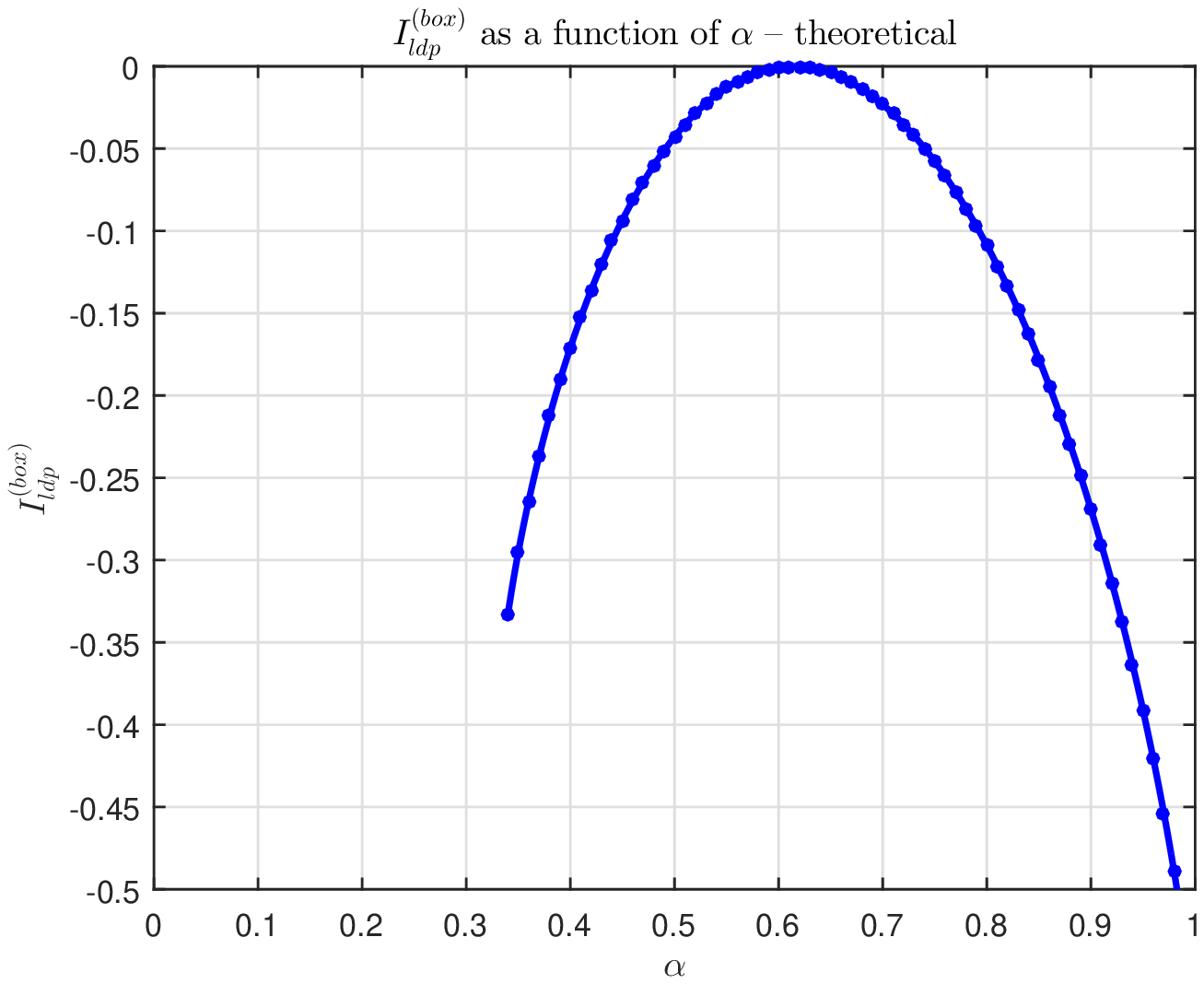,width=9cm,height=7cm}}
%\end{minipage}
%\begin{minipage}[b]{.5\linewidth}
%\centering
%\centerline{\epsfig{figure=finprerral08.eps,width=9cm,height=6.5cm}}
\end{minipage}
\caption{$I^{(box)}_{ldp}$ as a function of $\alpha$ for $\mu=0.85$; left -- $\beta=0.18469$; right -- $\beta=\frac{1}{3}$}
\label{fig:l1regldpIerrubnonn}
\end{figure}

\begin{table}[h]
\caption{A collection of values for $y_1$, $\gamma_{box}$, $A_{box}$, $y_2$, $\nu$, $A_0$, $c_3$, $\gamma$, and $I^{(box)}_{ldp}$ in Theorem \ref{thm:boxldp3}; $\beta=0.18469$, $\mu=0.85$}\vspace{.1in}
\hspace{-0in}\centering
\begin{tabular}{||c||c|c|c|c|c||}\hline\hline
$\alpha$ & $ 0.40 $ & $ 0.45 $ & $ 0.50 $ & $ 0.55 $ & $ 0.60 $ \\ \hline\hline
$y_1$        & $ 0.3588 $ & $ 0.3270 $ & $ 0.2951 $ & $ 0.2635 $ & $ 0.2321 $ \\ \hline
$\gamma_{box}$& $ 0.0607 $ & $ 0.0504 $ & $ 0.0414 $ & $ 0.0336 $ & $ 0.0269 $ \\ \hline
$A_{box}$    & $ 0.2806 $ & $ 0.2708 $ & $ 0.2612 $ & $ 0.2515 $ & $ 0.2418 $ \\ \hline
$y_2$        & $ 0.1994 $ & $ 0.2473 $ & $ 0.2951 $ & $ 0.3429 $ & $ 0.3904 $ \\ \hline\hline
%$\beta_1$    & $ 0.1965 $ & $ 0.3325 $ & $ 0.5000 $ & $ 0.7101 $ & $ 0.9792 $ \\ \hline
%$\beta_0$    & $ -0.0605 $ & $ 0.2760 $ & $ 0.5000 $ & $ 0.6592 $ & $ 0.7773 $ \\ \hline\hline
$\nu$        & $ 0.5074 $ & $ 0.4624 $ & $ 0.4173 $ & $ 0.3726 $ & $ 0.3282 $ \\ \hline
$A_0$        & $ 1.7994 $ & $ 1.3223 $ & $ 1.0000 $ & $ 0.7684 $ & $ 0.5945 $ \\ \hline
$c_3$        & $ -0.7866 $ & $ -0.3797 $ & $ 0.0000 $ & $ 0.3952 $ & $ 0.8424 $ \\ \hline
$\gamma$     & $ 0.1757 $ & $ 0.2537 $ & $ 0.3536 $ & $ 0.4825 $ & $ 0.6514 $ \\ \hline\hline
$I^{(box)}_{ldp}$    & $ \mathbf{-0.0284} $ & $ \mathbf{-0.0069} $ & $ \mathbf{0.0000} $ & $ \mathbf{-0.0066} $ & $ \mathbf{-0.0262} $ \\ \hline\hline
\end{tabular}
\label{tab:Ildptab1nonn}
\end{table}

\begin{figure}[htb]
%\begin{minipage}[b]{.5\linewidth}
\centering
\centerline{\epsfig{figure=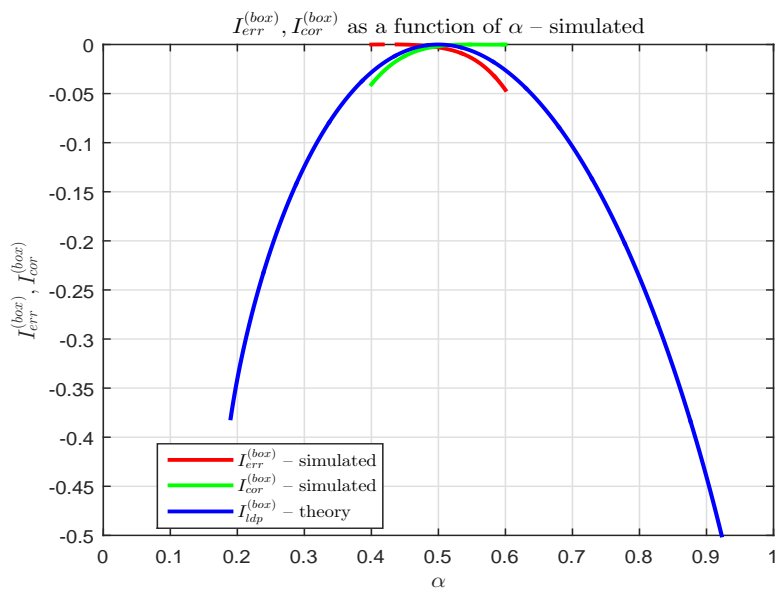,width=11.5cm,height=8cm}}
%\end{minipage}
%\begin{minipage}[b]{.5\linewidth}
%\centering
%\centerline{\epsfig{figure=finprerral08.eps,width=9cm,height=6.5cm}}
%\end{minipage}
\caption{Box $\ell_1$'s weak LDP rate function -- theory and simulation; $\beta=0.18469$, $\mu=0.85$}
\label{fig:weakl1LDPthrsimnonn}
\end{figure}

\begin{table}[h]
\caption{$I^{(box)}_{err}$, $I^{(box)}_{cor}$ -- simulated; $I^{(box)}_{ldp}$ calculated for $\beta=0.18469$ and $\mu=0.85$}\vspace{.1in}
\hspace{-0in}\centering
\begin{tabular}{||c||c|c|c|c|c||}\hline\hline
$\alpha$ & $ 0.40 $ & $ 0.45 $ & $ 0.50 $ & $ 0.55 $ & $ 0.60 $ \\ \hline\hline
$\mu(n-k)$& $ 87 $ & $ 208 $ & $ 208 $ & $ 208 $ & $ 87 $ \\ \hline
$k$       & $ 23 $ & $ 55 $ & $ 55 $ & $ 55 $ & $ 23 $ \\ \hline
$m$       & $ 50 $ & $ 135 $ & $ 150 $ & $ 165 $ & $ 75 $ \\ \hline
$n$       & $ 125 $ & $ 300 $ & $ 300 $ & $ 300 $ & $ 125 $ \\ \hline\hline
$I^{(box)}_{err}$ -- simulated & $ -0.0000 $ & $ -0.0001 $ & \red{$ \mathbf{ -0.0030} $} & \red{$ \mathbf{ -0.0140} $} & \red{$ \mathbf{ -0.0463} $} \\ \hline
$I^{(box)}_{cor}$ -- simulated & \gr{$ \mathbf{ -0.0407} $} & \gr{$ \mathbf{ -0.0118} $} & \gr{$ \mathbf{ -0.0017} $} & $ -0.0001 $ & $ -0.0000 $ \\ \hline\hline
$I^{(box)}_{ldp}$ -- theory   & \bl{$ \mathbf{-0.0284} $} & \bl{$ \mathbf{-0.0069} $} & \bl{$ \mathbf{0.0000} $} & \bl{$ \mathbf{-0.0066} $} & \bl{$ \mathbf{-0.0262} $} \\ \hline\hline
\end{tabular}
\label{tab:Ildptab2nonn}
\end{table}

%%%%%%%%%%%%%%%%%%%%%%%%%%%%%%%%%%%%%%%%%%%%%%%%%%%%%%%%%%%%%%%%%%%%%%%%%%%%%%%%
\section{Conclusion}
\label{sec:boxconc}
%%%%%%%%%%%%%%%%%%%%%%%%%%%%%%%%%%%%%%%%%%%%%%%%%%%%%%%%%%%%%%%%%%%%%%%%%%%%%%%%

This paper revisits the standard $\ell_1$ heuristic and its a modification when used for solving random linear systems with structured sparse solutions. Two types of structuring on top of the standard sparsity are considered here: 1) the solutions are binary, i.e. each component of the solution vector is from a set of only two values (these values are assumed to be a priori known); 2) the solutions are box-constrained, i.e. each component of the solution vector is either at one of the edges of the given interval or inside the interval. We looked at a relaxed modification of the standard $\ell_1$ that we referred to as the binary or box $\ell_1$ and how it fares when used for solving systems known to have solutions structured as above.

For both of these problems we presented the standard phase transition characterizations as well as a much deeper understanding of these phenomena by connecting them to the large deviations principles from the classical probability theory. A collection of very powerful probabilistic results that we obtained recently was often utilized. They turned out to be quite powerful even in the contexts of interest here and enabled us to explicitly characterize the large deviations in a manner similar to the one we showcased earlier when characterizing the phase transitions. Of particular importance in our view is that we were able to parallel the elegance we achieved earlier in the phase transitions characterizations in various other scenarios.

We also presented an alternative, high-dimensional geometry type of, view of the binary/box $\ell_1$. Through such an analysis we were again able to fully characterize the performance behavior of the modified $\ell_1$s. Consequently, we were able to show that the two substantially different analysis paths produce the same results (a conclusion certainly expected if the axioms of mathematics are properly set). To give a bit of a flavor as to what we actually proved in the paper, we also presented quite a few numerical results. They are in a very good agreement with all of our theoretical predictions (in fact, the simulated results indicate that this theoretical/numerical agreement already happens for systems of rather small dimensions of order of few hundreds which is perhaps somewhat surprising given that the theoretical results, by the definitions of the LDPs, assume systems of very large, basically infinite, dimensions). Clearly, there are many ways that one can exploit to continue further and study various other aspects of the algorithms/problems at hand. A couple of cosmetic adjustments of the techniques introduced here and in a few of our earlier works are on occasion needed. However, the above mentioned conceptual elegance of the approaches that we presented ensures that these adjustments are now fairly routine tasks. Still, we will present some of them in several companion papers for a few related problems that we find interesting.

%\newpage1
%\setcounter{page}{1}
\begin{singlespace}
\bibliographystyle{plain}
\bibliography{l1bnbxldpasym1Refs}

\begin{thebibliography}{10}

\bibitem{CRT}
E.~Candes, J.~Romberg, and T.~Tao.
\newblock Robust uncertainty principles: exact signal reconstruction from
  highly incomplete frequency information.
\newblock {\em IEEE Trans. on Information Theory}, 52(12):489--509, 2006.

\bibitem{DaiMil08}
W.~Dai and O.~Milenkovic.
\newblock Subspace pursuit for compressive sensing signal reconstruction.
\newblock available online at \bl{\url{https://arxiv.org/abs/0803.0811}}.

\bibitem{DaiMil09}
W.~Dai and O.~Milenkovic.
\newblock Weighted superimposed codes and constrained integer compressed
  sensing.
\newblock {\em IEEE Trans. on Information Theory}, 55(9):2215--2219, 2009.

\bibitem{DonohoUnsigned}
D.~Donoho.
\newblock Neighborly polytopes and sparse solutions of underdetermined linear
  equations.
\newblock 2004.
\newblock Technical report, Department of Statistics, Stanford University.

\bibitem{DonohoPol}
D.~Donoho.
\newblock High-dimensional centrally symmetric polytopes with neighborlines
  proportional to dimension.
\newblock {\em Disc. Comput. Geometry}, 35(4):617--652, 2006.

\bibitem{DonMalMon09}
D.~Donoho, A.~Maleki, and A.~Montanari.
\newblock Message-passing algorithms for compressed sensing.
\newblock {\em Proc. National Academy of Sciences}, 106(45):18914--18919, 2009.

\bibitem{DTbern}
D.~Donoho and J.~Tanner.
\newblock Counting the faces of randomly projected hypercubes and orthants with
  application.
\newblock 2008.
\newblock available online at \bl{\url{http://www.dsp.ece.rice.edu/cs/}}.

\bibitem{Donoho06CS}
D.~L. Donoho.
\newblock Compressed sensing.
\newblock {\em IEEE Trans. on Information Theory}, 52(4):1289--1306, 2006.

\bibitem{DTDSomp}
D.~L. Donoho, Y.~Tsaig, I.~Drori, and J.L. Starck.
\newblock Sparse solution of underdetermined linear equations by stagewise
  orthogonal matching pursuit.
\newblock {\em 2007}.
\newblock available online at \bl{\url{http://www.dsp.ece.rice.edu/cs/}}.

\bibitem{ManFer09}
O.~L. Mangasarian and M.~C. Ferris.
\newblock Uniqueness of integer solution of linear equations.
\newblock {\em Data Mining Institute Technical Report 09-01}, 2009.
\newblock available online at
  \bl{\url{http://www.cs.wisc.edu/math-prog/tech-reports}}.

\bibitem{ManRec09}
O.~L. Mangasarian and B.~Recht.
\newblock Probability of unique integer solution to a syatem of linear
  equations.
\newblock {\em Data Mining Institute Technical Report 09-02}, 2009.
\newblock available online at
  \bl{\url{http://www.cs.wisc.edu/math-prog/tech-reports}}.

\bibitem{NT08}
D.~Needell and J.~A. Tropp.
\newblock {CoSaMP}: Iterative signal recovery from incomplete and inaccurate
  samples.
\newblock {\em Applied and Computational Harmonic Analysis}, 26(3):301--321,
  2009.

\bibitem{NeVe07}
D.~Needell and R.~Vershynin.
\newblock Unifrom uncertainly principles and signal recovery via regularized
  orthogonal matching pursuit.
\newblock {\em Foundations of Computational Mathematics}, 9(3):317--334, 2009.

\bibitem{StojnicCSetamBlock09}
M.~Stojnic.
\newblock Block-length dependent thresholds in block-sparse compressed sensing.
\newblock available online at \bl{\url{http://arxiv.org/abs/0907.3679}}.

\bibitem{Stojnicl1BnBxfinn}
M.~Stojnic.
\newblock Box constrained $\ell_1$ optimization in random linear systems --
  finite dimensions.
\newblock available online at arXiv.

\bibitem{StojnicLiftStrSec13}
M.~Stojnic.
\newblock Lifting $\ell_1$-optimization strong and sectional thresholds.
\newblock available online at \bl{\url{http://arxiv.org/abs/1306.3770}}.

\bibitem{Stojnicl1RegPosasymldp}
M.~Stojnic.
\newblock Random linear systems with sparse solutions -- asymptotics and large
  deviations.
\newblock available online at \bl{\url{http://arxiv.org/abs/1612.06361}}.

\bibitem{StojnicBlockasymldpfinn15}
M.~Stojnic.
\newblock Random linear under-determined systems with block-sparse solutions --
  asymptotics, large deviations, and finite dimensions.
\newblock available online at arXiv.

\bibitem{StojnicRegRndDlt10}
M.~Stojnic.
\newblock Regularly random duality.
\newblock available online at \bl{\url{http://arxiv.org/abs/1303.7295}}.

\bibitem{StojnicTowBettCompSens13}
M.~Stojnic.
\newblock Towards a better compressed sensing.
\newblock available online at \bl{\url{http://arxiv.org/abs/1306.3801}}.

\bibitem{StojnicUpperSec13}
M.~Stojnic.
\newblock Upper-bounding $\ell_1$-optimization sectional thresholds.
\newblock available online at \bl{\url{http://arxiv.org/abs/1306.3778}}.

\bibitem{StojnicUpper10}
M.~Stojnic.
\newblock Upper-bounding $\ell_1$-optimization weak thresholds.
\newblock available online at \bl{\url{http://arxiv.org/abs/1303.7289}}.

\bibitem{StojnicCSetam09}
M.~Stojnic.
\newblock Various thresholds for $\ell_1$-optimization in compressed sensing.
\newblock available online at \bl{\url{http://arxiv.org/abs/0907.3666}}.

\bibitem{StojnicICASSP09}
M.~Stojnic.
\newblock A simple performance analysis of $\ell_1$-optimization in compressed
  sensing.
\newblock {\em ICASSP, International Conference on Acoustics, Signal and Speech
  Processing}, pages 3021--3024, April 2009.
\newblock Taipei, Taiwan.

\bibitem{StojnicISIT2010binary}
M.~Stojnic.
\newblock Recovery thresholds for $\ell_1$ optimization in binary compressed
  sensing.
\newblock {\em ISIT, IEEE International Symposium on Information Theory}, pages
  1593 -- 1597, 13-18 June 2010.
\newblock Austin, TX.

\bibitem{StojnicICASSP10knownsupp}
M.~Stojnic.
\newblock Towards improving $\ell_1$ optimization in compressed sensing.
\newblock {\em ICASSP, IEEE International Conference on Acoustics, Signal and
  Speech Processing}, pages 3938--3941, 14-19 March 2010.
\newblock Dallas, TX.

\bibitem{JATGomp}
J.~Tropp and A.~Gilbert.
\newblock Signal recovery from random measurements via orthogonal matching
  pursuit.
\newblock {\em IEEE Trans. on Information Theory}, 53(12):4655--4666, 2007.

\bibitem{YinZhang05nonneg}
Y.~Zhang.
\newblock A simple proof for recoverability of ell-1-minimization (ii): the
  nonnegative case.
\newblock {\em Rice CAAM Department Technical Report TR05-10}, 2005.
\newblock available online at \bl{\url{http://www.dsp.ece.rice.edu/cs/}}.

\end{thebibliography}
\end{singlespace}

\end{document}